\documentclass[openright,a4paper,american,11pt]{amsart}

\usepackage[latin1]{inputenc}

\usepackage[dvips]{graphicx}

\usepackage{psfrag}

\usepackage{amsmath}

\usepackage{amssymb}

\usepackage{babel}

\usepackage{amsfonts}

\input xy

\xyoption{all}

\usepackage[dvips]{graphicx}

\usepackage{boxedminipage}

\usepackage{color}

\newcommand{\ZZ}{{\mathbb Z}}

\newcommand{\NN}{{\mathbb N}}

\newcommand{\KK}{{\mathbb K}}

\newcommand{\CC}{{\mathbb C}}

\newcommand{\RR}{{\mathbb R}}

\newcommand{\QQ}{{\mathbb Q}}

\newcommand{\TT}{{\mathbb T}}

\newcommand{\E}{{\mathcal E}}

\newcommand{\I}{{\mathfrak I}}
\newcommand{\td}{{\text{''}}}
\newcommand{\tg}{{\text{``}}}

\newtheorem{them}{Theorem}
\newtheorem{coro}{Corollary}

\newtheorem{thm}{Theorem}[section]

\newtheorem{defi}[thm]{Definition}

\newtheorem{prop}[thm]{Proposition}

\newtheorem{lemma}[thm]{Lemma}

\newtheorem{cor}[thm]{Corollary}

         {\theoremstyle{definition}

}

         {\theoremstyle{definition}

\newtheorem{exa}[thm]{Example}}

\begin{document}

\title{INFLECTION POINTS OF  REAL AND TROPICAL PLANE CURVES}

\date{\today}

\author[E.  Brugallé]{Erwan A. Brugallé}

\address{Université Pierre et Marie Curie,  Paris 6, 
4 place Jussieu, 75005 Paris, France}

\email{brugalle@math.jussieu.fr}

\author[L.  L\'opez de Medrano]{Lucia M. López de Medrano}

\address{Unidad Cuernavaca del Instituto de Matemáticas,Universidad Nacional
Autonoma de México. Cuernavaca, México}

\email{lucia@matcuer.unam.mx}

\subjclass[2000]{}

\keywords{Tropical geometry, Patchworking, Inflection points, Tropical
  modifications, Real algebraic curves}

\begin{abstract}
We prove that Viro's patchworking produces real algebraic curves with
the maximal number of real inflection points. In particular this
implies that maximally inflected
 real algebraic $M$-curves realize many isotopy types.
The strategy we adopt in this paper is tropical:
we study tropical limits of
inflection points of
classical plane algebraic curves. The main tropical tool
we use to understand these tropical inflection points 
 are tropical modifications.
\end{abstract}

\maketitle

\section{Introduction}
Let $k$ be any field of characteristic 0, and
consider a plane algebraic curve $X$ in  $kP^2$
given by the homogeneous  equation $P(z,w,u)=0$. The Hessian of the polynomial
$P(z,w,u)$, denoted by $Hess_P(z,w,u)$, is  the homogeneous
polynomial defined as

$$Hess_P(z,w,u)=\det
\left(\begin{matrix}
\frac{\partial^2 P}{\partial ^2z}

           & $ $ & \frac{\partial^2 P}{\partial z\partial w}

              & $ $ & \frac{\partial^2 P}{\partial z\partial u}\\
              
               $ $ & $ $ & $ $\\

        \frac{\partial^2 P}{\partial z\partial w}

           & $ $ &\frac{\partial^2 P}{\partial ^2w}

               & $ $ & \frac{\partial^2 P}{\partial w\partial u}\\
               
               $ $ & $ $ & $ $\\

       \frac{\partial^2 P}{\partial z\partial u}

           & $ $ & \frac{\partial^2 P}{\partial w\partial u}

              & $ $  &\frac{\partial^2 P}{\partial ^2u}

   \end{matrix}\right). $$

If $Hess_P(z,w,u)$ is not the null polynomial, it
defines a curve $Hess_X$  called the \textbf{Hessian} of $X$. 
Note that $Hess_X$
only depends on $X$, and is invariant under projective transformations
of $kP^2$.
An \textbf{inflection point} of the curve $X$ is by definition a point $p$
in $X\cap Hess_X$, of multiplicity $m$ if $(X\circ Hess_X)_p=m$.

A plane algebraic curve
 has two kinds of inflection points: its singular points, and
non-singular points having a contact of order 
$l\ge 3$
with their
tangent line. In this latter case, the multiplicity of such an
inflection point is exactly 
$l-2$.

\vspace{1ex}
If $k$ is algebraically closed, Bézout Theorem implies that
 an algebraic curve $X$ in $kP^2$ 
of degree $d\ge 2$ which is reduced and does not contain any line
has exactly $3d(d-2)$ inflection points (counted
with multiplicity). Moreover, a non-singular generic curve $X$ has
exactly  $3d(d-2)$ inflection points, all of them of multiplicity 1.
 
\vspace{1ex}
When $k$ is not algebraically closed, the situation
becomes more subtle. First, the number of inflection points of an
algebraic curve in $kP^2$ depends not only on its degree, but also on
the coefficients of its equation. 
In the case $k=\RR$, it has been known for a long time 
that a non-singular
real cubic has only 3 real  points among its
 9 inflection points.  
More generally, Klein proved that at most one
third of the complex inflection points of a non-singular real algebraic
curve can actually be real.

\begin{thm}[Klein \cite{Kle2}, see also \cite{Ron1},
\cite{Schuh}, and \cite{V10}]\label{klein}
A non-singular real algebraic curve in $\RR P^2$ of degree $d\ge 3$ cannot
have more than $d(d-2)$ real inflection points.
\end{thm}
 Klein also proved that this upper bound is sharp. Following
 \cite{Kha4}, we say that a non-singular real algebraic curve of
 degree $d$ in $\RR P^2$ is
 \textbf{maximally inflected} if it has $d(d-2)$ distinct real
 inflexion points.
If a real algebraic
 curve has a node $p$ with two real branches such that each branch is
 locally strictly  convex around $p$, then any smoothing of $p$
 produces 
 two real inflection points. 
 Applying Hilbert's method of construction, the previous observation implies 
 immediately the existence of maximally inflected curves in any
 degree at least 2.
However, real inflection points of maximally inflected  curves remains
quite mysterious. For example, which rigid-isotopy classes of real algebraic
curves contain a maximally inflected curve? How real inflection
points can be distributed among the connected component of a maximally
inflected curve?

The first step to answer  questions of this sort  is of course to find a
systematic way to construct maximally real inflected curves.
Invented by Viro at the end of the seventies (see \cite{V1}), the
patchworking technique 
turned out to be one of the most powerful 
 method to construct real algebraic curves with controlled
topology.
One of the main contribution of this paper is to prove that patchworking also
  provides
  a systematic method to construct maximally inflected real
  curves.

For the sake of shortness we do not recall this technique
here, we refer
 instead to the tropical presentation made
in \cite{V9}, \cite{Mik12}, or \cite{Br11}.
In non-tropical terms, Theorem \ref{main intro} states that any
real primitive $T$-curve, under a mild assumption on the corresponding
convex function, is maximally inflected. Note that this result is of
the same flavor 
as the fact that $T$-curves have asymptotically maximal
total curvature (see \cite{Lop} and \cite{Ris}).
We denote by $T_d$ the triangle in $\RR^2$ with vertices $(0,0)$,
$(d,0)$ and $(0,d)$.
All precise definitions needed in Theorem \ref{main intro} are given in
section \ref{standard trop}.
\begin{them}\label{main intro}
Let $C$ be a non-singular tropical curve in $\RR^2$ defined by the
tropical polynomial $\tg \sum_{i,j}a_{i,j} x^iy^j \td$ with Newton
polygon the triangle $T_d$ with $d\ge 2$. Suppose that if $v$ is a
vertex of $C$ dual to $T_1$, 
then its 3 adjacent edges have 3 different length.
Then the real algebraic curve   defined by the polynomial
 $ P(z,w)=\sum_{i,j} \alpha_{i,j}t^{-a_{i,j}} z^iw^j$ with
$\alpha_{i,j}\in\RR$ has  exactly
$d(d-2)$ inflection points in $\RR P²$ for $t>0$ small enough. 
\end{them}

As an example of application of Theorem \ref{main intro}, combined
  with
classical results in topology of real algebraic curves (see \cite{V1}
and \cite{V3}
for example), we get the following 
 corollary.

\begin{coro}
Any rigid isotopy class of non-singular real algebraic curves of degree
at most 6 with non-empty real part contains a maximally inflected curve.

Any real scheme with non-empty real part realized by a non-singular
real algebraic curve of degree 7 is realized by a maximally inflected
curve of degree 7.
\end{coro}

Theorem \ref{main intro} is a weak version of Theorem \ref{main real}:
the  polynomials $P(z,w)$ are in fact polynomials over the
field of generalized Puiseux series, and
 we give in addition the distribution of real inflection points
among the connected components of a real algebraic curve obtained by patchworking.
See  Figures \ref{honey real} and \ref{honey real 2} from Example
\ref{exa1},
 as well as section
\ref{constructions} for some examples of such patchworkings.

\vspace{2ex}

A
plane \textbf{tropical curve} $C$
can be thought as a combinatorial encoding 
of a 1-parametric degeneration of 
 plane 
complex algebraic curves 
$X(t)$ (see section \ref{standard trop} for
definitions).
The main part in the proof of Theorem 
 \ref{main intro} is then to understand which points of $C$
 represent a limit of inflection points of the
 algebraic curves $X(t)$. Since plane tropical curves are piecewise linear
 objects, 
the location of these tropical intersection points
 is not obvious at first sight, and we need to refine the tropical
 limit process. 
\textbf{Tropical modifications}, introduced by Mikhalkin in \cite{Mik3},
allow such a refinement.

It follows from Kapranov's Theorem that
the tropicalization $C$ of a family of
plane complex algebraic curves $X(t)$
only depends on the first order term in $t$  of the coefficients 
of the  equation of $X(t)$.
As rough as it may seem, the curve $C$ keeps track of a
non-negligible amount of information about the family $(X(t))$. For
example, if $C$ is non-singular, the genus of $X(t)$ is
equal to the first Betti number of $C$.
 However, some information
depending on more than just first order terms
might be lost when passing from $(X(t))$ to $C$. Tropical
modifications  refine  the tropicalization
process, and allows one to recover some information about $(X(t))$ sensitive
to higher order terms.

By
means of these tropical modifications, we identify a finite number of
\textbf{inflection components} on any non-singular tropical curve $C$ 
(Proposition \ref{inclusion}). These inflection components 
 are the tropical analogues of inflection
points. Using further tropical modifications, we prove that 
the \textbf{multiplicity} $\mu(\E)$ of such a component $\E$ (i.e. 
the
number of inflection points of $X(t)$ which tropicalize in $\E$)
 only depends on the combinatoric of $C$ (Theorem \ref{main}). 
Now suppose that $X(t)$ is a family of \textit{real} algebraic curves.
As an immediate consequence,
 we get that the number of real inflection points of $X(t)$ which tropicalize
 in $\E$ has the same parity 
 
 as $\mu(\E)$. In Theorem
 \ref{main real}, we establish that a generic tropical curve has
 exactly $d(d-2)$ inflection components with odd multiplicity.
Hence Theorem
 \ref{main real} together with Klein Theorem  imply 
that $X(t)$ has exactly $d(d-2)$ real inflection points when $t$ is
small enough.

At several places in the text, we will see that tropical
modifications can also be used to \textbf{localize} a problem. For
example, relation between classical and tropical intersections 
(Proposition \ref{set intersection}), or
intersections between a curve and its Hessian (Theorem \ref{main
  real}), 
are reduced to easy
local considerations after a suitable tropical modification.

\tableofcontents

\vspace{2ex}
\textbf{Acknowledgment: } We are grateful to Viatcheslav Kharlamov for
pointing us the fact that any smoothing of a real node with two local
real convex branches produces two real inflection points, as well as
for his simplification of the proof of Proposition \ref{classification d=4}.
We are also indebted to Grigory Mikhalkin for many
useful discussions about  tropical modifications.

Finally, we  thank
Jean-Jacques Risler and Frank Sottile for their 
encouragements, as well as 
Cristhian Garay 
and the unkonwn referee
for many useful comments
on preliminary versions of this paper.

A major part of this work has been done during the visit of E.B. at the Universidad Nacional Autónoma de México (Instituo de Matemáticas. Unidad Cuernavaca), and during the
 Tropical Geometry semester at  MSRI in Fall 2009. We thank these institutions for their support and excellent working conditions
 they offered us. Both authors were partially founded by UNAM-PAPIIT IN-117110 and CONACyT 55084.
E.B. is also
partially supported by the ANR-09-BLAN-0039-01 and ANR-09-JCJC-0097-01.
L.L. is also partially supported by UNAM-DGAPA 
and Laboratorio Internacional Solomon Lefschetz (LAISLA), asociated to the CNRS (France) and CONACYT (Mexico).

\section{Convention}\label{convention}
Here we pose once for all some notations and conventions we will
use throughout the paper. Almost all of them are commonly used in the
literature. 

An integer convex polytope in $\RR^n$ is a convex
polytope with vertices in $\ZZ^n$. The \textbf{integer volume} is the
Euclidean volume normalized so 
that the standard simplex
 with vertices $0$, $(1,0,\ldots,0)$, $(0,1, 0,\ldots,0)$, $\ldots$, 
$(0,\ldots,0,1)$ has volume 1. That is to say, the integer volume in $\RR^n$ is
 $n!$ times the Euclidean volume in $\RR^n$.
 In this paper,  we only consider
  integer volumes.
A simplex $\Delta$ in $\RR^n$ will be called \textbf{primitive} if it
has volume 1. Equivalently, $\Delta$ is primitive if and only if it is 
the image  
of the standard simplex
under an element of $GL_n(\ZZ)$ composed with a translation.

Given $d\ge 1$, we denote by $T_d$ the integer triangle in $\RR^2$ with vertices
$(0,0)$, $(d,0)$, and $(0,d)$.

A \textbf{facet} of a polyhedral complex is a face of maximal dimension.

The letter $k$ denotes an arbitrary
field of characteristic 0. 
Given $P(z)$ a polynomial in $n$ variables over  $k$, 
we denote by
$V(P)$ the hypersurface of $(k^*)^n$ defined by $P(z)$.
We write $P(z)=\sum a_{i}z^i$ with $i=(i_1,\ldots,i_n)$,
$z=(z_n,\ldots, z_n)$, and $a_iz^i=a_iz_{1}^{i_1}\ldots z_{n}^{i_n}$.
The Newton polytope of $P(z)$ 
is denoted by $\Delta(P)\subset\RR^n$, and given
a subset $\Delta'$ of $\Delta(P)$, we define the \textbf{restriction} of
  $P(z)$ along $\Delta'$, by
$$P^{\Delta'}(z):=\sum_{i\in \Delta'\cap\ZZ^n} a_{i}z^i .$$
If $X$ and $X'$ are two algebraic curves in the projective plane
$kP^2$,  the intersection multiplicity of $X$
and $X'$ at a point $p\in kP^2$ is denoted by $(X\circ X')_p$.

\section{Standard tropical geometry}\label{standard trop}

In this section we review briefly some  standard facts about tropical
geometry, and we fix the notations used in this paper. For a more
educational exposition, we refer, for example, to \cite{Mik3}, \cite{Mik9},
 \cite{St2}, and \cite{BIT}. 
There exist several non-equivalent definitions of tropical
varieties in the literature.
 In this paper, we have chosen for practical reasons
to present them via
non-archimedean amoebas. 

\subsection{Non-archimedean amoebas}\label{non-arch}

A \textbf{locally convergent generalized Puiseux
series} is a formal series of the form
$$a(t) = \sum_{r\in R}\alpha_r t^r $$
where $R\subset\RR$ is a well-ordered set,  $\alpha_r\in\CC$, and the
series is convergent for $t>0$ small enough.
We denote by $\KK$  the set of all locally convergent generalized Puiseux
series. It is naturally a field of characteristic $0$, which turns out
to be algebraically closed. 
An element $a(t) = \sum_{r\in R}\alpha_r t^r $ of $\KK$ is said to be 
 \textbf{real} if  $\alpha_r\in\RR$ for all $r$, and we denote by
 $\RR\KK$ the subfield of $\KK$ composed of real series.

Since elements of $\KK$ are convergent for $t>0$ small
enough, an algebraic variety over $\KK$ (resp. $\RR\KK$)
can be  seen as a one parametric family
of algebraic varieties over $\CC$ (resp. $\RR$).

\vspace{1ex}
The field $\KK$ has a natural non-archimedean valuation 
defined as follows:
 $$\begin{array}{cccc}
val:&\KK&\longrightarrow& \RR\cup\{-\infty\}
\\ & 0 &\longmapsto & -\infty
\\ & \sum_{r\in R}\alpha_r t^r \ne 0&\longmapsto &-min\{r\mid \alpha_r\neq
0\}.
 \end{array}$$
Note that 
we call $val$ a valuation, although it is rather the opposite of a
valuation
for  classical litterature.
This valuation extends naturally to a map 
$Val:\KK^n\rightarrow (\RR\cup\{-\infty\})^n$ by evaluating $val$ coordinate-wise, i.e.
$Val(z_1,\dots,z_n)=(val(z_1),\dots,val(z_n))$.  
If $X\subset(\KK^*)^n$ is an algebraic variety,
$Val(X)\subset\RR^n$  is called the \textbf{non-archimedean amoeba} of $X$.

\begin{exa}
An integer matrix $M\in\mathcal M_{n,m}(\ZZ)$
 defines a multiplicative map
$\Phi_{M}:(\KK^*)^m\to (\KK^*)^n$. The non-archimedean
amoeba of $\Phi_{M}((\KK^*)^m)$ is the vector subspace
of $\RR^n$ spanned by the columns of $M$.
\end{exa}

Let $X$ be an irreducible algebraic variety of dimension $m$. In this
case, Bieri
and Groves proved in \cite{BiGr} that $Val(X)$ 
 is a finite rational polyhedral
complex of pure dimension $m$ (rational means that each of its
faces has a direction defined over $\QQ$). Given a facet $F$ 
 of $Val(X)$, we associate a positive integer number $w(F)$, called
\textbf{the weight} of $F$, as
follows: pick a point $(p_1,\ldots,p_n)$ 
in the relative interior of $F$, and choose
a basis $(e_1,\ldots,e_n)$ of $\ZZ^n\subset\RR^n$ such that
$(e_1,\ldots,e_m)$ is a basis of the direction of $F$; denote by
$Y_F\subset(\KK^*)^n$ 
the 
multiplicative translation of $\Phi_{(e_{m+1},\ldots,e_{n})}((\KK^*)^{n-m})$ along
$(t^{p_1},\ldots,t^{p_n})$, and define
$w(F)$ as the number (counted with multiplicity) of intersection
points of $X$ and $Y_F$ with valuation $(p_1,\ldots,p_n)$.
Note that $w(F)$ does not depend on the choice of the point $(p_1,\ldots,p_n)$.

\begin{exa}

A matrix $M\in\mathcal M_{n,m}(\ZZ)$ with $Ker\ M=\{0\}$ maps the
lattice $\ZZ^m$ to a sub-lattice $\Lambda'$ of $\Lambda=\ZZ^n\cap Im \ M$.
The weight of the non-archimedean
amoeba 
of $\Phi_{M}((\KK^*)^m)$ is the index of $\Lambda'$ in $\Lambda$.
\end{exa}

\begin{defi}
The non-archimedean amoeba of $X$ equipped with the weight function on
its facets is called the tropicalization of $X$,
and is denoted by $Trop(X)$.
\end{defi}
The notion of tropicalization extends naturally to any algebraic
subvariety of $(\KK^*)^n$, not necessarily of pure dimension. 
In this paper, a \textbf{tropical variety} is 
a finite rational polyhedral
complex in $\RR^n$ equipped with a weight function, and
 which is the tropicalization of some algebraic
subvariety of $(\KK^*)^n$.

\begin{exa}\label{trop plane}
A plane tropical curve, a tropical plane in $\RR^3$, and a tropical curve
contained in this tropical plane are depicted 
in Figures \ref{example}a, \ref{example}b and \ref{example}c.
\end{exa}

\begin{figure} [htbp]

\begin{center}
\begin{tabular}{ccc}
\includegraphics[height=5cm, angle=0]{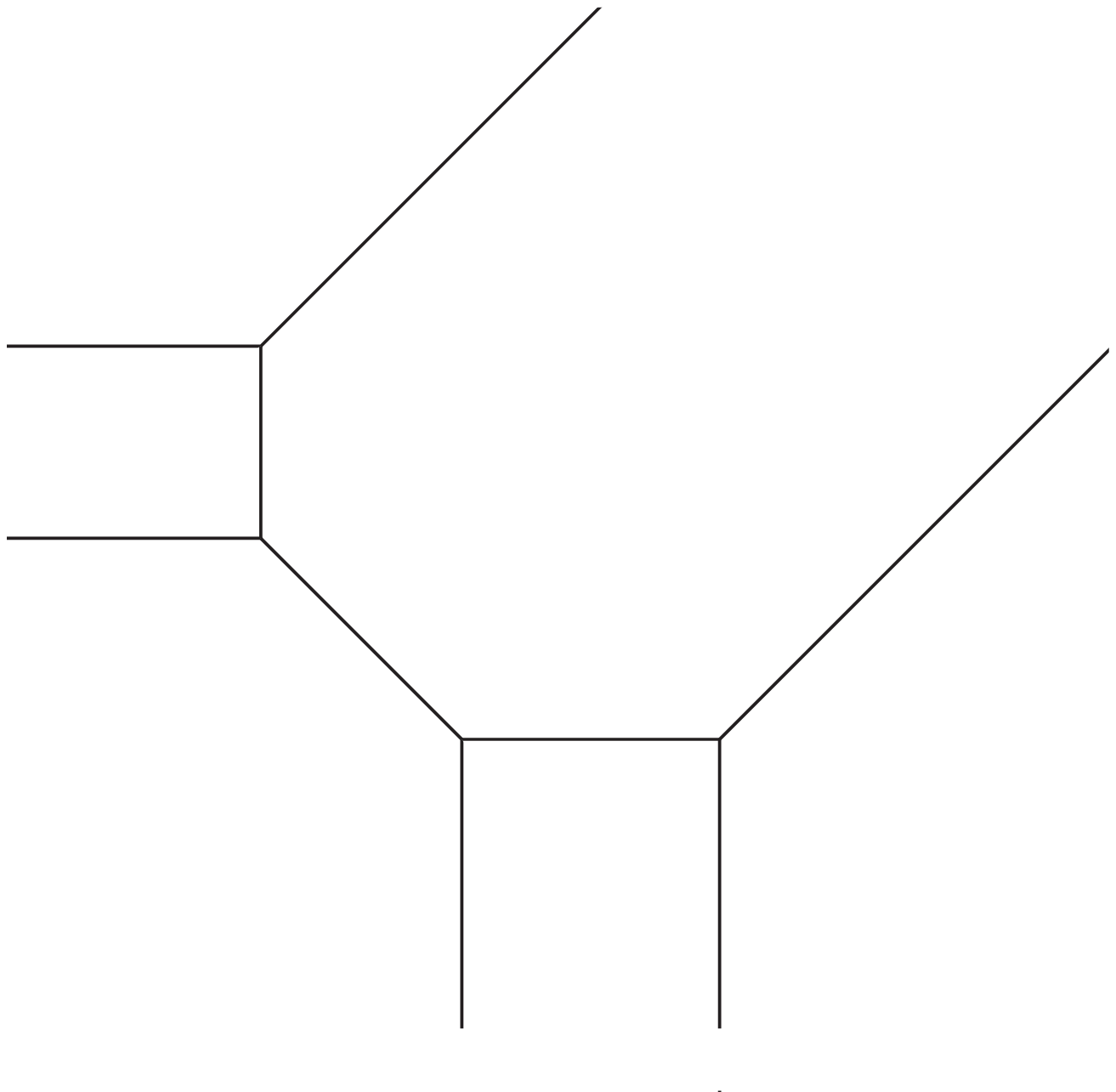}&
\includegraphics[height=5cm, angle=0]{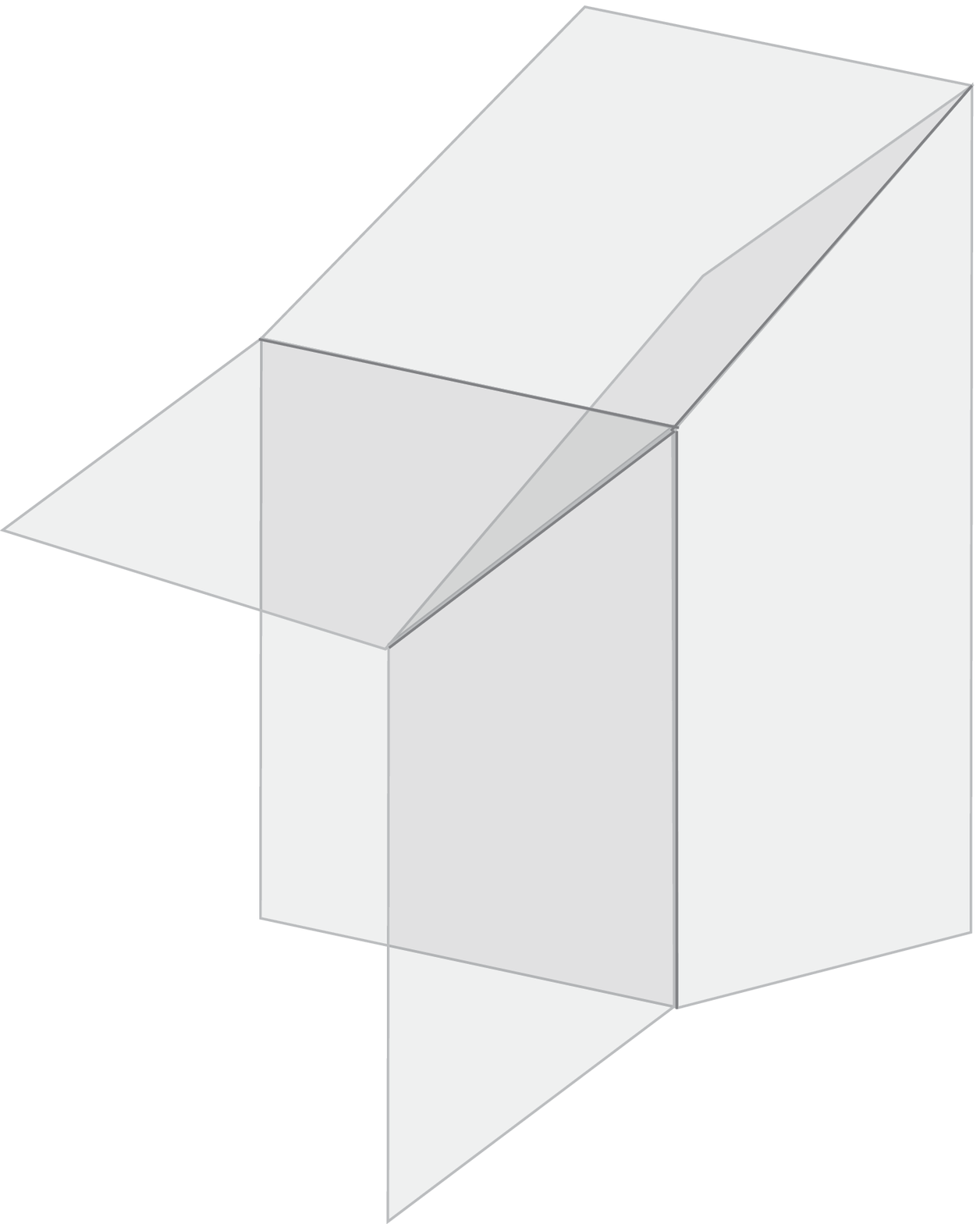}&
\includegraphics[height=5cm, angle=0]{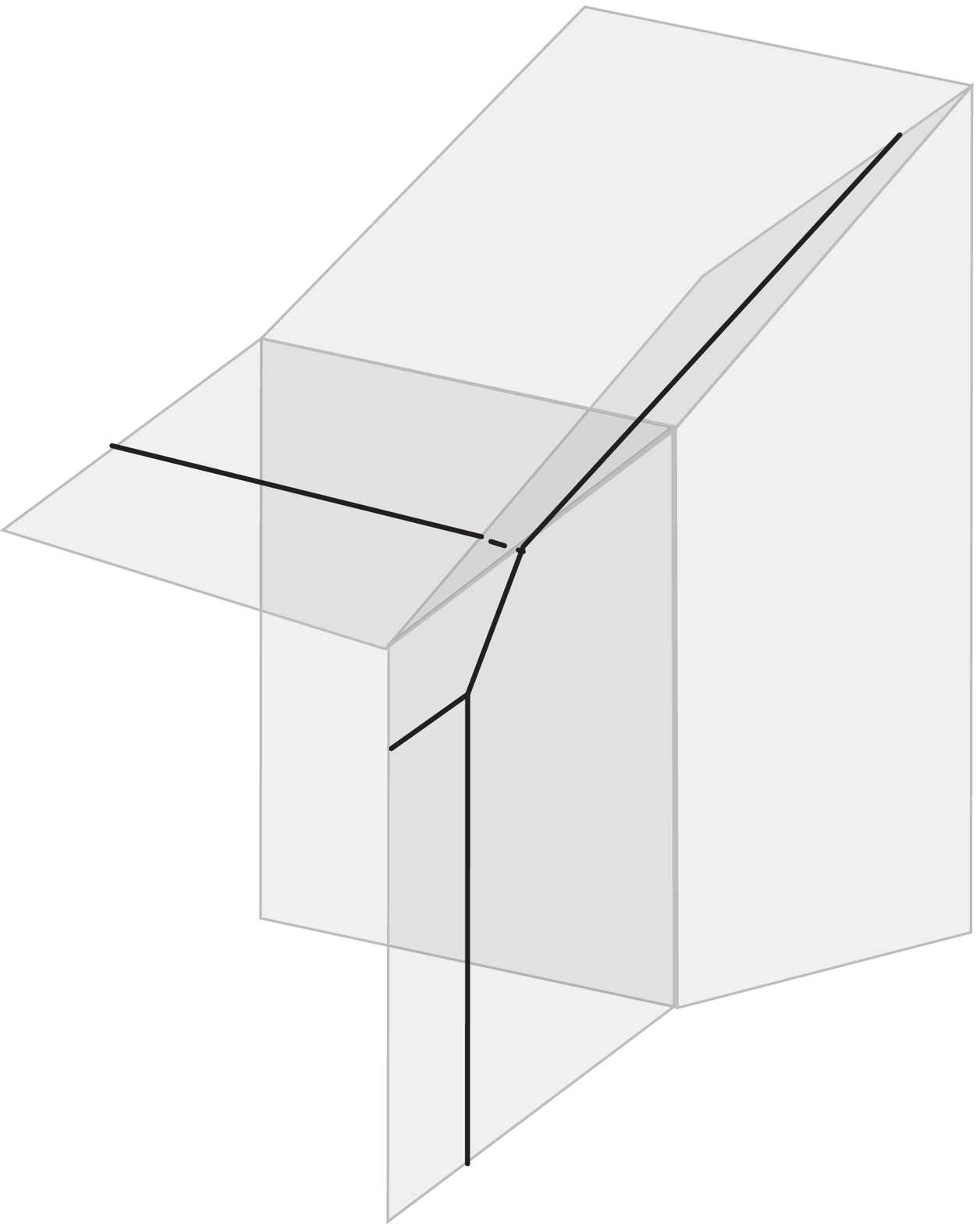}
\\ a) & b) &c)
\end{tabular}
\caption {Examples of tropical varieties. In these cases all the weights are equal to 1.}
\label{example}
\end{center}
\end{figure}

\begin{defi}
Let $V$ be a tropical variety in $\RR^n$, and $X$ be an algebraic
subvariety of $(\KK^*)^n$.  We say that $X$
 realizes $V$ if $V=Trop(X)$. If $X=V(P)$ for some polynomial $P(z)$,
 we say that $P(z)$ realizes $V$.
\end{defi}

Tropical varieties satisfy the so-called \textbf{balancing
  condition}. We give here this property only for tropical
curves, since this is anyway the only case we need in this paper and 
 makes the exposition easier. We refer to \cite{Mik3} for the general
 case.

Let $C\subset\RR^n$ be a tropical curve, and let $v$ be a vertex of
$C$.
Let $e_1,\ldots, e_l$ be the edges of $C$ adjacent to
$v$. Since $C$ is a rational graph, each edge $e_i$ has a primitive
integer direction. If in addition we ask that the orientation of $e_i$
defined by this vector points away from $v$, then this primitive
integer vector is unique. Let us denote by $u_{v,e_i}$ this vector.

\begin{prop}[Balancing condition]
For any vertex $v$, one has
 $$\sum_{i=1}^l w(e_i)u_{v,e_i}=0.$$ 
\end{prop}

If $C$ is a tropical curve in $\RR^n$, any of its bounded edge $e$ has a
\textbf{length} $l(e)$ defined as follows:
$$l(e)=\frac{||v_1v_2||}{w(e)||u_{v_1,e}||} $$
where $v_1$ and $v_2$ are its adjacent vertices, and $||v_1v_2||$
(resp. $||u_{v_1,e}||$) denotes the Euclidean length 
of the vector $v_1v_2$ (resp. $u_{v_1,e}$).

\subsection{Tropical hypersurfaces}

Let us now study closer tropical hypersurfaces, i.e. 
tropical varieties in $\RR^n$ of pure dimension $n-1$. These particular tropical varieties
can easily be described as algebraic varieties over the
\textbf{tropical semi-field}  $(\TT,\tg+\td,\tg\times\td)$.
Recall that $\TT=\RR\cup\{-\infty\}$ 
and that for any two elements $a$ and $b$ in
$\TT$, one has 
\begin{center}
$\tg a+b \td=\max(a,b)$ and $\tg a\times b \td =a+b.$
\end{center}
As usual, we abbreviate $a\times b$ in $ab$, and 
$(\TT,\tg+\td,\tg\times\td)$ in $\TT$, and we use the convention that
$\max(-\infty,a)=a$ and $-\infty + a=-\infty$.
Note that $\TT^*=\RR$.

Since  $\TT$ is a semi-field,
we have a natural notion of tropical polynomials, i.e. polynomials over
$\TT$. Such a polynomial $P(x)=\tg\sum a_ix^i \td$ induces a function
 $$\begin{array}{cccc}
P:&\TT^n&\longrightarrow& \TT
\\ & x&\longmapsto & \max (\left\langle x,i\right\rangle + a_i)
 \end{array}$$
where $x=(x_1,\ldots,x_n)\in\TT^n$, $i=(i_1,\ldots,i_n)\in\NN^n$,
$x^i=\tg x_1^{i_1}\ldots, x_n^{i_n}\td$, and
$\left\langle\ ,\ \right\rangle$ denotes
 the standard
Euclidean product on $\RR^n$.

We denote by $\overset{\circ}V(P)$ the set of points $x$ in  $\RR^n$ for 
which the value
of $P(x)$ is given by at least 2 monomials. 
This is a finite rational polyhedral complex, which induces a
subdivision $\Theta$ of $\RR^n$.
Given $F$  a face of $\Theta$ and $x$  a point in the relative
interior of $F$, the set $\{i\in\Delta(P)\ | \ P(x)=\tg a_ix^i\td\}$
does not depend on $x$. We denote its convex hull by $\Delta_F$. 
All together,  the polyhedrons  $\Delta_F$ form a subdivision of $\Delta(P)$,
called the \textbf{dual subdivision} of $P(x)$. The polyhedron
$\Delta_F$ is called the \textbf{dual cell} of $F$, and $\dim
\Delta_F=n-\dim F$. 
In particular,
if $F$ is a facet of 
$\overset{\circ}V(P)$ then $\Delta_F$ is a segment, and we define
the \textbf{weight} of $F$ by $w(F)=Card(\Delta_F\cap\ZZ^n)-1$. We
denote by $V(P)$ the polyhedral complex $\overset{\circ}V(P)$ equipped
with the map $w$ on its facets. $V(P)$ is called
the \textbf{tropical hypersurface} defined by $P(x)$. 

The  
  Newton polygon of $P(x)$ and its dual subdivision 
 are entirely determined, up to translation, by
$V(P)$.
A tropical hypersurface is said to be \textbf{non-singular} if all the maximal
 cells of its dual
 subdivision are primitive simplices. 
In particular, any facet of a non-singular tropical hypersurface has
weight 1.

Note that we have used the same notations  
as in section
\ref{non-arch}. This is justified by the following fundamental Theorem, due to Kapranov.
\begin{thm}[Kapranov \cite{Kap1}]
Let $P(z)=\sum a_iz^i $ 
be a polynomial over $\KK$. If we define $P_{trop}(x)=\tg\sum
val(a_i)x^i \td $, 
then we have
$$Trop(V(P))=V(P_{trop}). $$ 
\end{thm}

\begin{exa}
The tropical planar curve and the tropical plane in Figure
\ref{example}a and \ref{example}b, are given respectively by the
tropical polynomials 
 
$P(x,y)=\tg x^2+y^2+2x+2y+3xy+3\td$ 
and $Q(x,y,z)=\tg x+y+z+1\td$.
\end{exa}

Let $P(z)$ be a polynomial over $\KK$ realizing a tropical hypersurface $V$ in
$\RR^n$. 
To
each  face $F$ of $V$ dual to the polyhedron $\Delta_F$, 
we associate below a complex polynomial
$P_{\CC,F}(z)$. 
Let $i_1$, $\ldots$,
$i_l$ be the vertices of $\Delta_F$.

Let us first suppose that $\Delta_F$ has dimension $n$.
In this case, 
the points $(i_1,-val(a_{i_1}))$, $\ldots,$ $(i_l,-val(a_{i_l}))$ lie on the same hyperplane in
$\RR^n\times \RR$. Hence we have $-val(a_{i_j})=\lambda_{\Delta_F}(i_j)$ where
$\lambda_{\Delta_F}:\RR^n \to\RR$ is a linear-affine map. The maps 
$\lambda_{\Delta_F}$ glue along faces of codimension 1 to produce a
convex piecewise-affine map $\lambda:\Delta(P)\to\RR$. Note that the
cells of dimension $n$ of the dual subdivision of $V$ correspond
exactly to the domains of linearity of $\lambda$, and that 
$-val(a_i)\ge \lambda(i)$ for
any $i\in\Delta(P)\cap \ZZ^n$.

Let us go back to the case when
 $\Delta_F$ may have any dimension between 0 and $n$. According to the
preceding paragraph, there exists a linear-affine
function $\lambda_{\Delta_F}:\RR^n\to\RR $ such that 
$-val(a_{i_j})= \lambda_{\Delta_F}(i_j)$ for all $j=1\ldots l$, and  
$-val(a_{i})> \lambda_{\Delta_F}(i)$  for
any  $i$
not in $\Delta_F$. If
$\Delta_F$ has dimension $n$, then $\lambda_{\Delta_F}$ is unique and is
precisely the map we defined above.
Write $\lambda_{\Delta_F}(i)=\sum \gamma_ji_j + \alpha$, and
define
$\widetilde P(z)=t^{\alpha}P(t^{\gamma_1}z_1,\ldots,
t^{\gamma_n}z_n)$. If we write $\widetilde P(z)=\sum\widetilde a_iz^i$, then
$-val(\widetilde a_i)\ge 0$ for $i\in\Delta_F$, and $-val(\widetilde
a_i)> 0$ for $i\notin\Delta_F$. Hence, if we plug $t=0$ in $\widetilde
P(z)$, we obtain a well defined 
complex polynomial $P_{\CC,F}(z)$ with Newton
polygon $\Delta_F$.
Note that if $P(z)$ is defined over $\RR\KK$, then all the polynomials 
$P_{\CC,F}(z)$ are real.

\subsection{Tropical intersection}\label{sub trop int}

Let $P_1(x,y)$ and $P_2(x,y)$ be two tropical polynomials defining
respectively the tropical curves $C_1$ and $C_2$ in $\RR^2$.  Then, the
polynomial $P_3(x,y)=\tg P_1(x,y)P_2(x,y) \td$ defines a tropical
curve $C_3$, whose underlying set is the union of $C_2\cup C_3$. A
vertex of $C_3$ which is in the set-theoretic intersection $C_1\cap C_2$ is
called a \textbf{tropical intersection point}\footnote{Such points are also called \textbf{stable intersection
  points} in the literature, we refer to \cite{St2} for a justification
of this terminology.} of $C_1$ and $C_2$.
The set of tropical intersection points of
$C_1$ and $C_2$ is denoted by $C_1\cap_\TT C_2$.

Two tropical curves might have an infinite set-theoretic intersection,
however they always have a finite number of tropical intersection
points. 
Now we assign a multiplicity $(C_1\circ_\TT C_2)_v$ 
to each tropical intersection point $v$ of
$C_1$ and $C_2$ as follows
$$(C_1\circ_\TT C_2)_v=\frac{1}{2}\left(Area(\Delta_{v})-\delta_v \right)$$
where
\begin{itemize}
\item $\delta_v=0$ if $v$ is
an isolated intersection point of two edges of $C_1$ and $C_2$; 
\item $\delta_v=Area(\Delta_{v'})$ if $v$ is 
a vertex $v'$ of $C_1$ (resp. $C_2$)  but not of $C_2$
  (resp. $C_1$); 
\item $\delta_v=Area(\Delta_{v'})+ Area(\Delta_{v''})$ if $v$ is 
a vertex $v'$ of $C_1$, but  also a vertex $v''$ of $C_2$; 
\end{itemize}

Note that $(C_1\circ_\TT C_2)_v$ only depends on $C_1$ and $C_2$, and
neither on $P_1(x,y)$ nor $P_2(x,y)$.

A \textbf{component} of $C_1\cap C_2$ is a
connected component of this set. Such a component $E$ has a
multiplicity defined as
$$(C_1\circ_\TT C_2)_E=\sum_{v\in C_1\cap_{\TT,E} C_2} (C_1\circ_\TT
C_2)_v$$
where $C_1\cap_{\TT,E} C_2$ is the set of tropical intersection points
of $C_1$ and $C_2$ contained in $E$.
\begin{exa}
A transverse and a non-transverse intersection of planar tropical
lines
 are depicted respectively in Figures \ref{trans}a and \ref{trans}b.
\end{exa}

\begin{figure} [htbp]

\psfrag{C2}{$C_2$}\psfrag{C1}{$C_1$}\psfrag{p}{$p$}\psfrag{E}{$E$}
\begin{center}
\begin{tabular}{cc}
\includegraphics[height=3cm, angle=0]{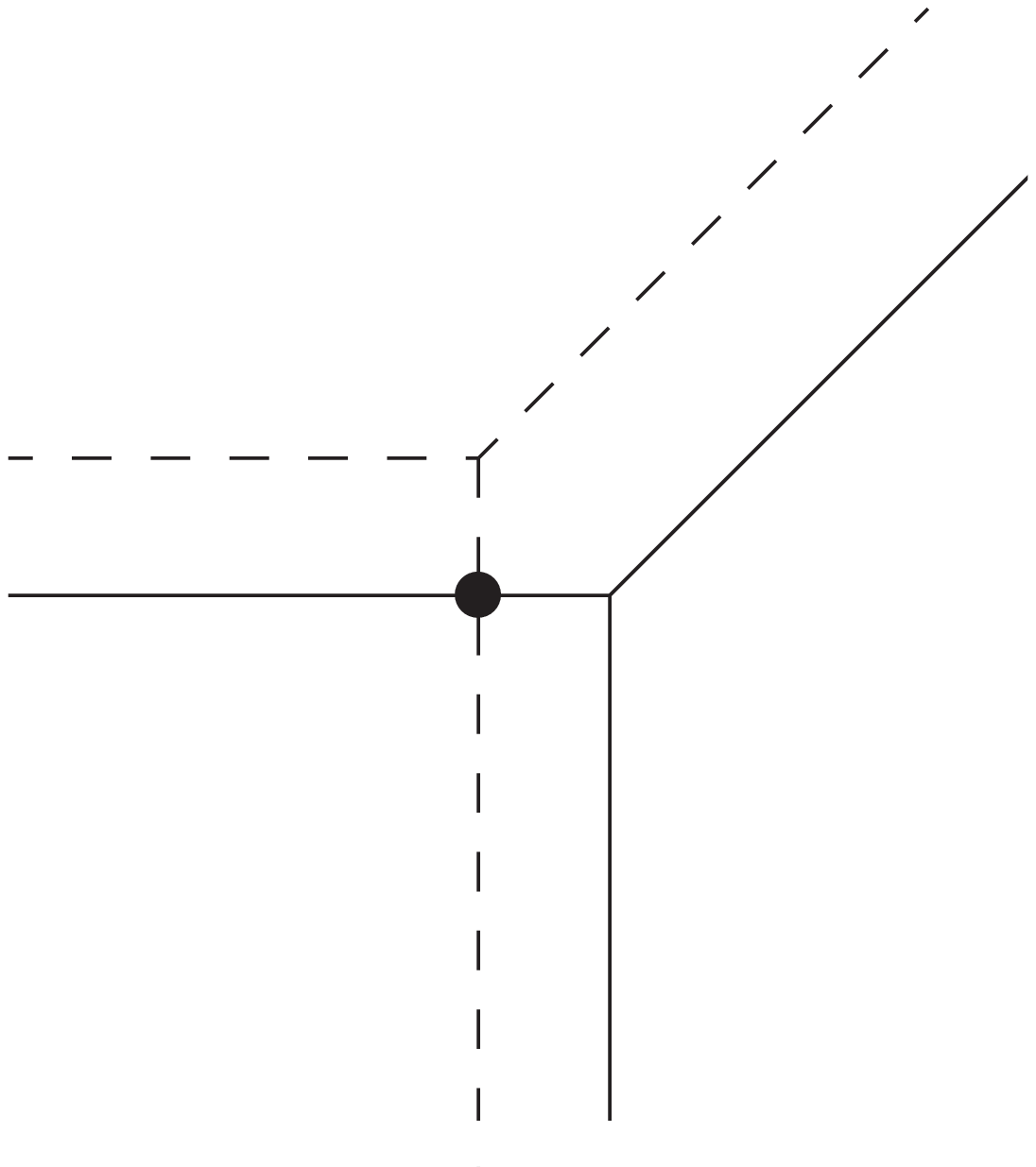}&
\includegraphics[height=3cm, angle=0]{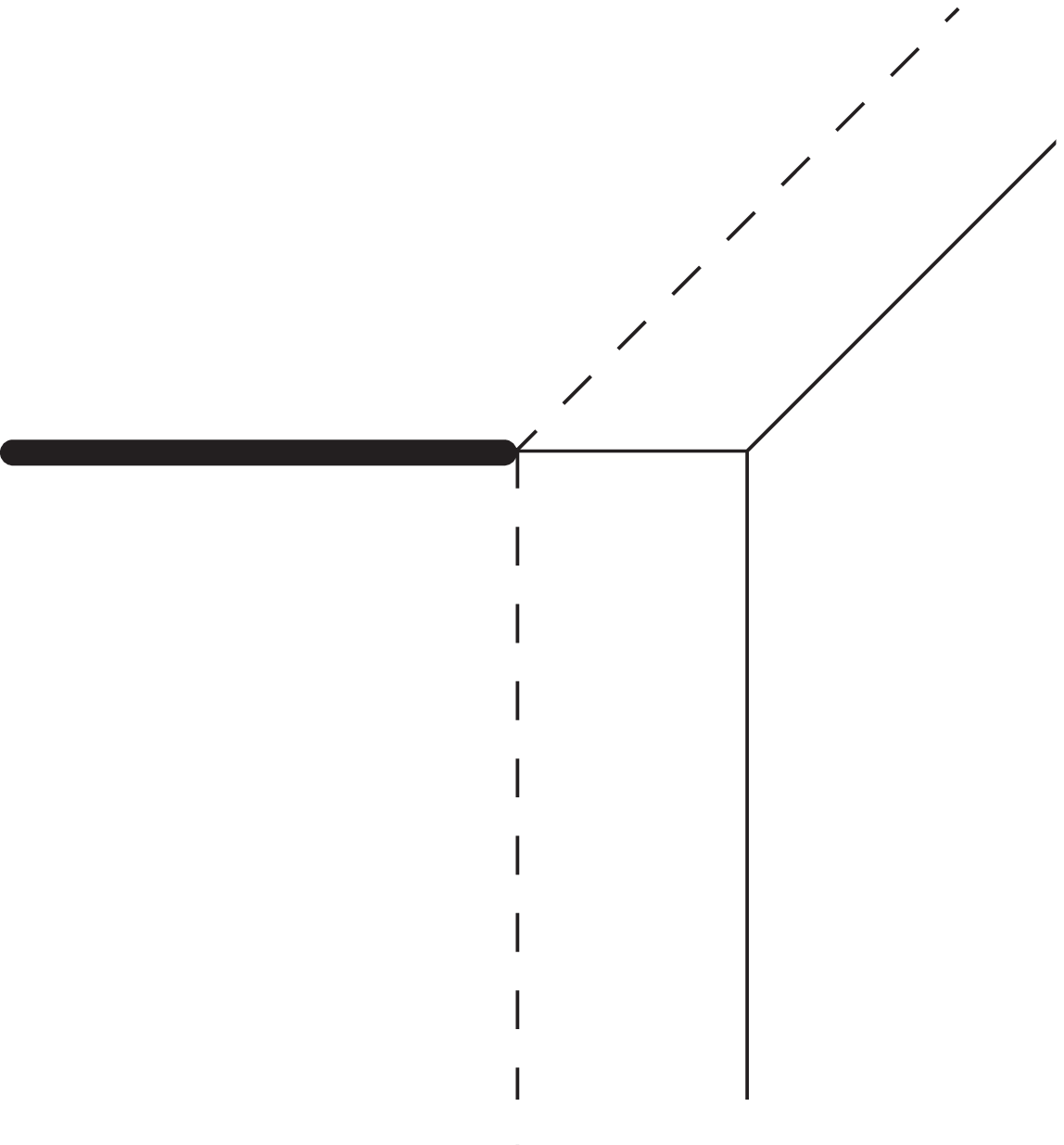}
\\ a) & b)
\end{tabular}
\caption {Transverse and non-transverse intersections. In both cases,
  $\mu(E)=1$.} 
\label{trans}
\end{center}
\end{figure}

\begin{exa}
Let $C_1$ and $C_2$ be two non-singular tropical curves in $\RR^2$
with a component 
$E$  of $C_1\cap C_2$ not reduced to a
point. Suppose that $E$ contains a boundary point $p$ which is not a
vertex of both $C_1$ and $C_2$ (see Figure \ref{2}). Then
$(C_1\circ_\TT C_2)_p=1$.  
\begin{figure} [htbp]

\psfrag{E}{$E$}\psfrag{e1}{$e_1$}\psfrag{e2}{$e_2$}\psfrag{C2}{$C_2$}\psfrag{C1}{$C_1$}\psfrag{p}{$p$}

\includegraphics[scale=0.30]{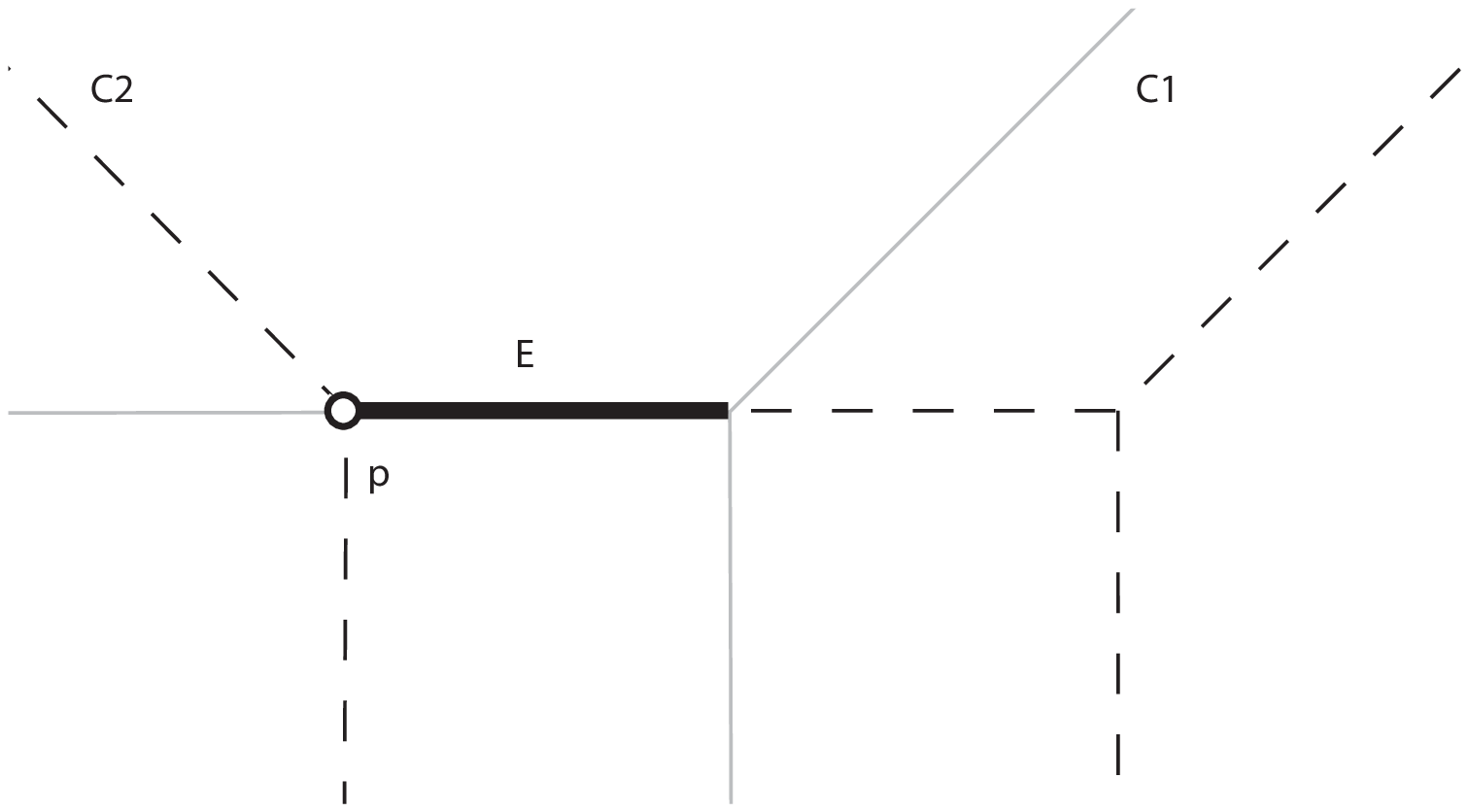}

\begin{center}

\caption {$(C_1\circ_\TT C_2)_p=1$  and $(C_1\circ_\TT C_2)_E=2$.}
\label{2}
\end{center}

\end{figure}

\end{exa}

Intersection in $(\KK^*)^2$ and tropical intersection are related by
the following Proposition.
When  the set-theoretic intersection is infinite, 
we use tropical modifications to reduce
the problem to local computations. Hence
we postpone the proof of Proposition \ref{trop int} to section \ref{section:
  modifications} (see Proposition \ref{set intersection}). Note that
Rabinoff also gave in \cite{Rab1} a proof of  Proposition \ref{trop int} 
 using Berkovich spaces.
\begin{prop}\label{trop int}
Let $X_1$ and $X_2$ be two algebraic curves in $(\KK^*)^2$ intersecting
in a finite number of points, and let
$E$ be a component of the intersection of $C_1=Trop(X_1)$ and
$C_2=Trop(X_2)$. Then, the number of intersection point (counted with
multiplicity) of $X_1$ and $X_2$ with valuation in $E$ is at most 
$(C_1\circ_\TT C_2)_E$, with equality if $E$ is compact.
\end{prop}

Next we prove some easy lemmas we will use later in this paper. 
Lemma \ref{few} is probably already known, however we couldn't find it
explicitely in the litterature.

\begin{lemma}\label{few}
A  polynomial in one variable with $l$ monomials cannot have a root of
order $l$ other than 0.
\end{lemma}
\begin{proof}
We prove the Lemma by induction on $l$.
The Lemma is obviously true if $l=1$. Suppose now that the Lemma is
true for some $l\ge 1$, and let $P(z)$ be a polynomial in one variable with $l+1$
monomials.
Since we are looking at roots $z\ne 0$, we may suppose that 
the constant term  of $P(z)$ is non null. In particular, the
derivative $P'(z)$,
 of $P(z)$  has
$l$ monomials. So $P(z)$ cannot have a root $z\ne 0$ of order bigger than $l+1$
since otherwise it would be a root of order bigger than $l$ of $P'(z)$.
\end{proof}

\begin{lemma}\label{no tang}
Let $X_1$ and $X_2$ be two algebraic curves in $(\KK^*)^2$, and
suppose that there exists a tropical intersection point $p$ of 
$C_1=Trop(X_1)$ and
$C_2=Trop(X_2)$ which is 
the  isolated intersection of an edge $e_1$ of $C_1$ and an
edge $e_2$ of $C_2$ (see Figure \ref{4}a). Suppose in addition 
that $p$ is neither a vertex of $C_1$ nor 
of $C_2$, and that $w(e_1)=w(e_2)=1$.
Then, any intersection point  of $X_1$ and $X_2$ with valuation $p$ is transverse.
\end{lemma}
\begin{figure} [htbp]

\psfrag{C2}{$C_2$}\psfrag{C1}{$C_1$}\psfrag{p}{$p$}
\begin{center}
\begin{tabular}{cc}
\includegraphics[height=3cm, angle=0]{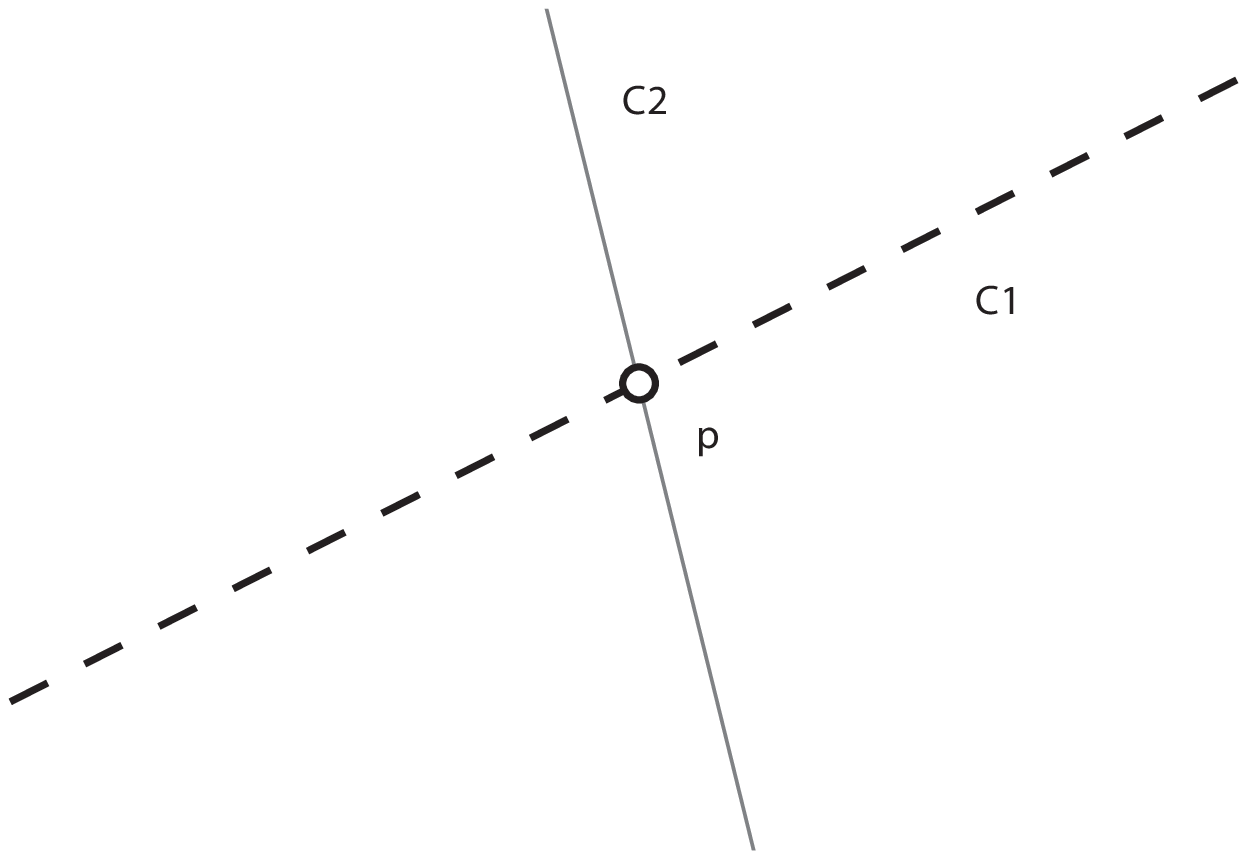}&
\includegraphics[height=3cm, angle=0]{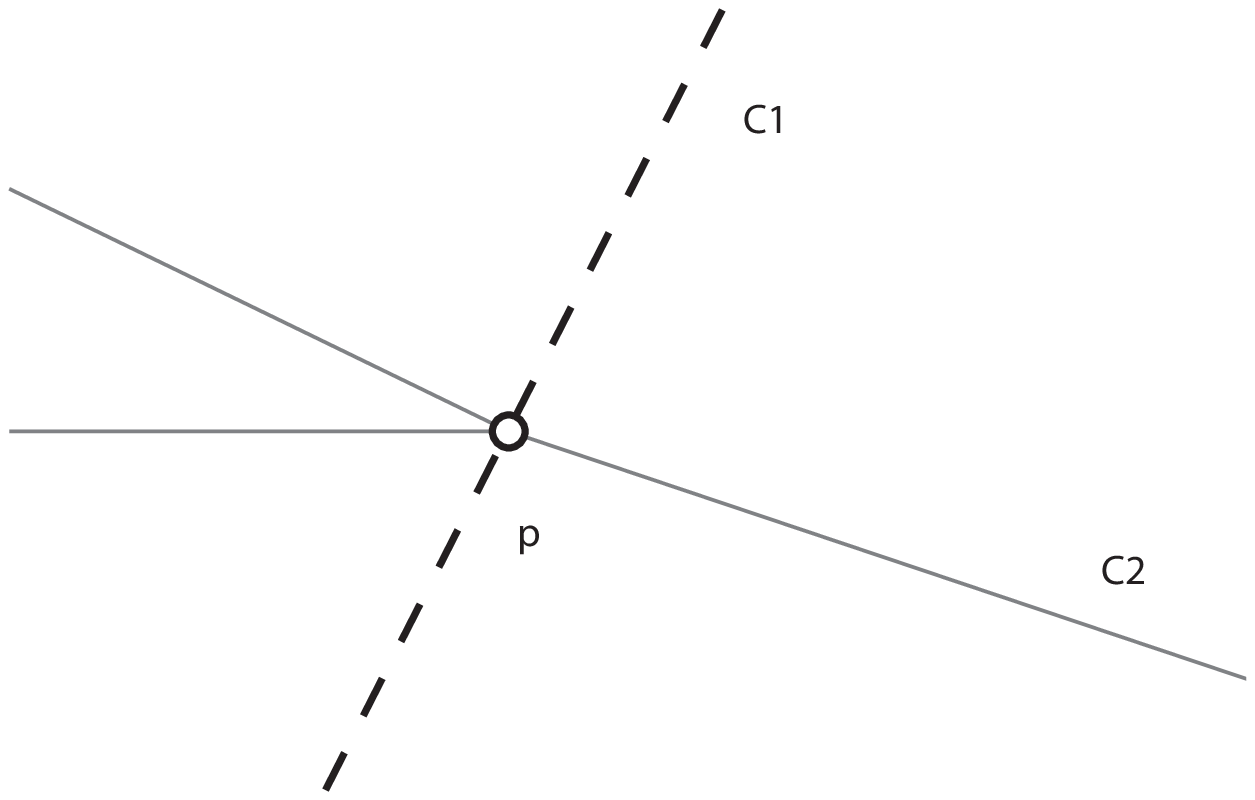}
\\ a) & b)
\end{tabular}
\caption {In this two cases, if $Trop(X_i)=C_i$ and $C_i$ is non-singular, no intersection point of $X_1$ and $X_2$ with valuation $p$ has multiplicity bigger that 2.}
\label{4}
\end{center}
\end{figure}

\begin{proof}
Suppose that $X_i$ is defined by a polynomial $P_i(z,w)$, and suppose
that there exists an intersection point  of $X_1$ and $X_2$ with
valuation $p$ with multiplicity at least $2$.
Without loss of generality, we may suppose that
$P_1^{\Delta_{e_1}}(z,w)=z-1$, and 
that the two coefficients of
$P_2(z,w)$ corresponding to  $\Delta_{e_2}$ have valuation 0. In
particular, $p=(0,0)$.
 Then, the algebraic varieties $X_1(t)$ and $X_2(t)$ have an
intersection point of multiplicity at least 2 which converges in
 $(\CC^*)^2$ when $t\to 0$. This implies that the two curves $V(z-1)$ and
$V(P_{2,\CC,e_2})$
 have an
intersection point of multiplicity at least 2 in
 $(\CC^*)^2$. Since these intersection points are solution  in $\CC^*$
of the
equation $P_{2,\CC,e_2}(1,w)=0$
which is a binomial equation, this is
impossible by Lemma \ref{few}.
\end{proof}

\begin{lemma}\label{no infl}
Let $X_1$ and $X_2$ be two algebraic curves in $(\KK^*)^2$, and
suppose that there exists a tropical intersection point $p$ of 
$C_1=Trop(X_1)$ and
$C_2=Trop(X_2)$ such that $p$ is 
a vertex $v$ of $C_2$ but not of $C_1$
 (see Figure \ref{4}b). 
 Suppose in addition that $\Delta_v$ is primitive, and that $w(e)=1$
 where $e$ is the edge of $C_1$ containing $p$.
Then,  any intersection point  of $X_1$ and $X_2$ with valuation $p$
is of multiplicity at most 2.
\end{lemma}

\begin{proof}
As in the proof Lemma \ref{no tang}, we may suppose that $X_1$ is
the line with equation $z=1$ and that $P_2(z,w)$ is a trinomial. Once
again, the result follows from Lemma \ref{few}.
\end{proof}

For a deeper study of simple tropical tangencies, we refer the
interested reader to the
forthcoming papers \cite{Br9} and \cite{Br12}.

\begin{lemma}\label{propE1} 
Let $l$ be a positive integer, let $X_1$ be an algebraic curve
in $(\KK^*)^2$ with Newton polygon the triangle with vertices $(0,0)$,
$(0,1)$ and $(1,l)$, and let $X_2$ be a line in  $(\KK^*)^2$. Suppose
that $v$ is a vertex of both $Trop(X_1)$ and $Trop(X_2)$ (see Figure \ref{6}).
Then,
counting with multiplicity, at least $l-1$ intersection
points of $X_1$ and $X_2$ have valuation $v$ (note that $(Trop(X_1)\circ_\TT Trop(X_2))_v=l+1$). 
\end{lemma}

\begin{figure} [htbp]

\psfrag{e}{$e$}\psfrag{C2}{$Trop(X_2)$}\psfrag{C1}{$Trop(X_1)$}\psfrag{v}{$v$}

\includegraphics[scale=0.30]{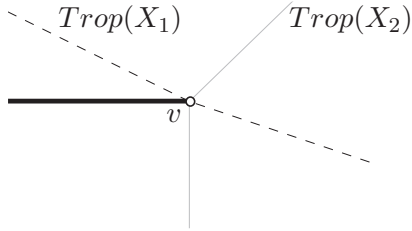}

\begin{center}

\caption {At least $l-1$ intersection points of $X_1$ and $X_2$ have valuation $v$.}
\label{6}
\end{center}
\end{figure}
\begin{proof} 
Without lost of generality, we may suppose that $X_1$ is defined by
the polynomial $P(z,w)=1+w+zw^l$,
and $X_2$ by  $Q(z,w)=a+bz-w$ with $val(a)=val(b)=0$. In
particular, $v=(0,0)$.
The intersection points of $X_1$ and $X_2$
 are the points $(z,a +bz)$ where
 $z$ is a root of the polynomial $\widetilde{P}(z)=P(z,a+bz)$. We have 
$$\begin{array}{lll}
\widetilde{P}(z)&=&(1+a)+(b+a^l)z+\sum_{j=1}^l
\begin{pmatrix}
l\\
j\\
\end{pmatrix}a^{l-j}b^{j}z^{j+1} 
\\&=&\sum_{j=0}^{l+1}c_jz^j.
\end{array}$$

Since
 $val(a)=val(b)=0$, we have $val(c_0)\leq 0$, $val(c_1)\leq 0$ and 
$val(c_j)=0$ for $j\ge 2$. Hence, 0 is a tropical root of order at
least $l-1$ of $\widetilde{P}_{trop}$.
\end{proof}

\section{Tropical modifications}\label{section: modifications}

The tropicalization of an
algebraic variety $X$ in $(\KK^*)^n$ defined by an ideal $I$ 
only depends on the first order term
of  elements of $I$. For hypersurfaces this follows immediately from
Kapranov's Theorem; in the general situation one can refer to \cite{St6}
or \cite{Ale1}.
As rough as it may seem, the tropicalization process keeps track of a
non-negligible amount of information about  original algebraic
varieties, e.g. intersection multiplicities. However, some information
depending on more than just first order terms
might be lost when passing from $X$ to $Trop(X)$. Tropical
modifications, introduced by Mikhalkin in \cite{Mik3},
 can be seen as a refinement of the tropicalization
process, and allows one to recover some information about $X$ sensitive
to higher order terms.

\subsection{Example}
Let us start with a simple example illustrating our approach. Consider
the two 
lines $X_1$ and $X_2$ in $(\KK^*)^2$ with equation
$$X_1:\ P_1(z,w)=(1+t^2)+z+w=0\quad  \text{and}\quad  X_2:\ (1+t) + z + t^{-1}w=0.$$
It is not hard to compute that these two lines intersect at the point 
$p=( -1,-t^2)$ which has
valuation $(0,-2)$.
Suppose now that we want to compute the valuation of $p$
just using tropical geometry, i.e. looking at $Trop(X_1)$ and
$Trop(X_2)$. As depicted in Figure \ref{7}a,
 the set $Trop(X_1)\cap Trop(X_2)$ is infinite, and it is not
 clear at all which point on  $Trop(X_1)\cap Trop(X_2)$ corresponds to
 $Val(p)$. Proposition \ref{trop int} and
the stable intersection point $(0,-1)$ of
 $Trop(X_1)$ and $Trop(X_2)$ tell us that $X_1$ and $X_2$ intersect in
at most 1 point, but turn out to be useless in the exact determination of
$Val(p)$.

\begin{figure} [htbp]

\psfrag{C2}{$Trop(X_2)$}\psfrag{C1}{$Trop(X_1)$}\psfrag{v}{$v$}\psfrag{T}{$C'$}\psfrag{t}{$C''$}\psfrag{A}{$A$}\psfrag{B}{$B$}
\psfrag{C}{$C$}
\begin{center}
\begin{tabular}{ccc}
\includegraphics[height=5cm, angle=0]{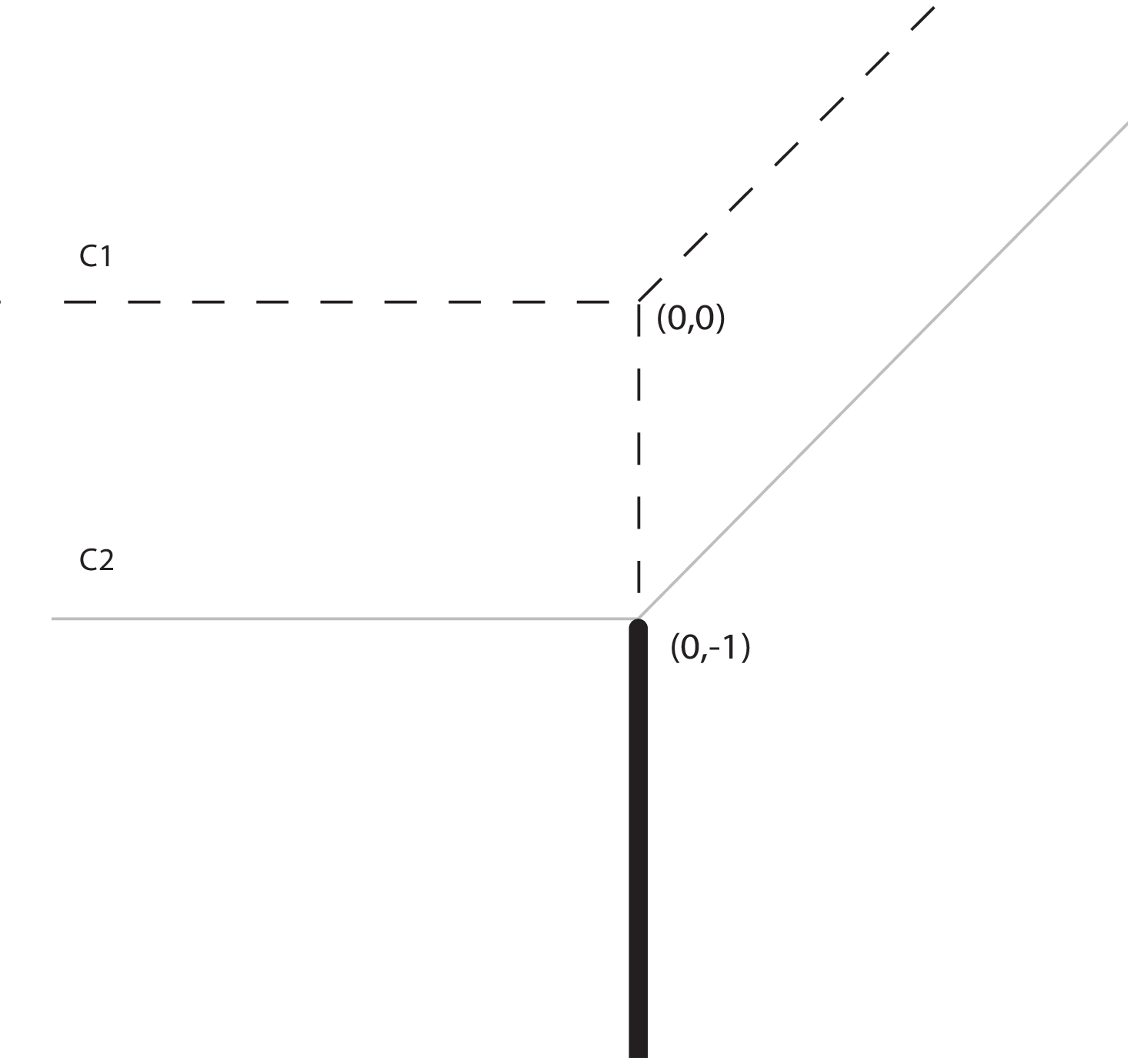}&
\includegraphics[height=5cm, angle=0]{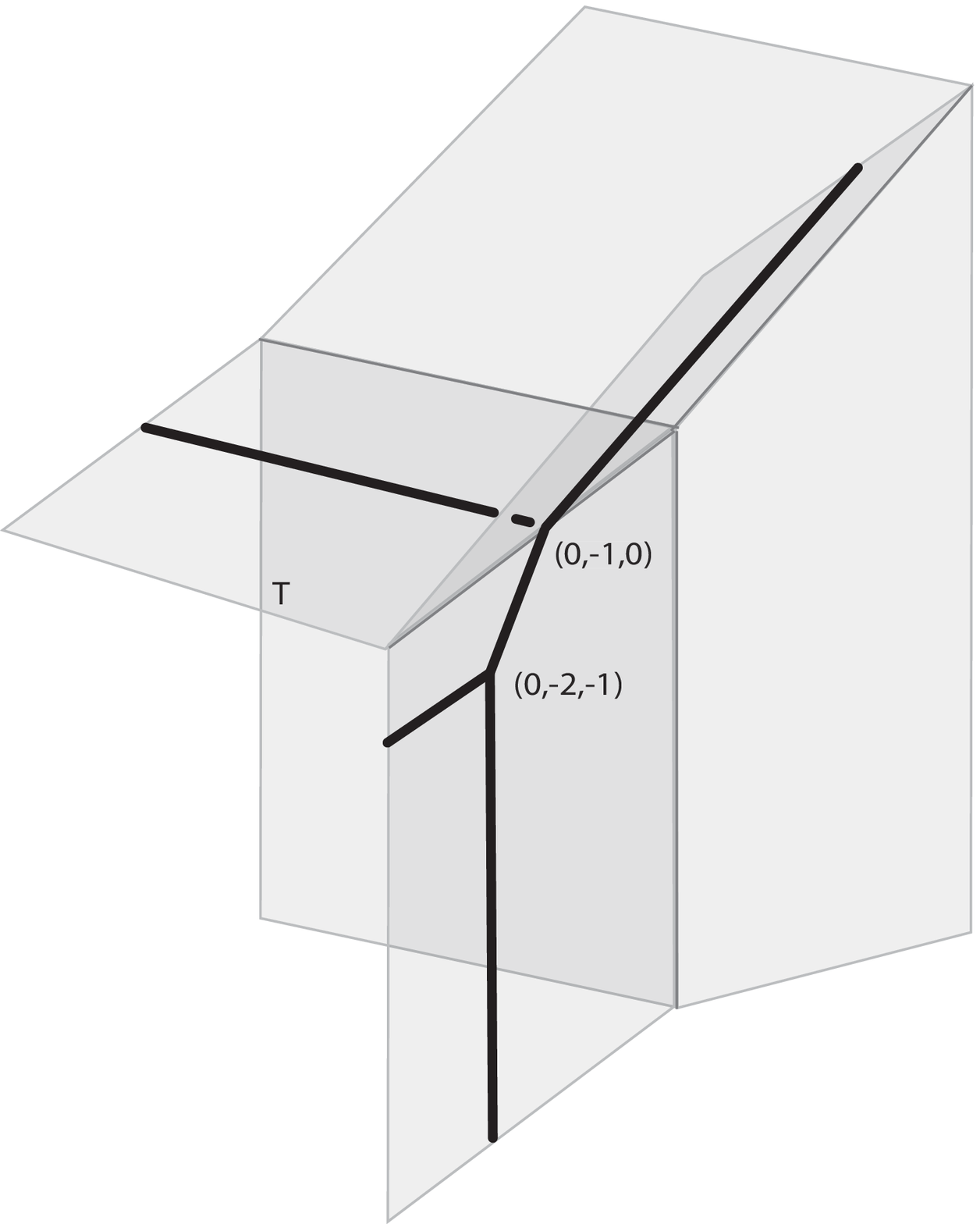}&
\includegraphics[height=5cm, angle=0]{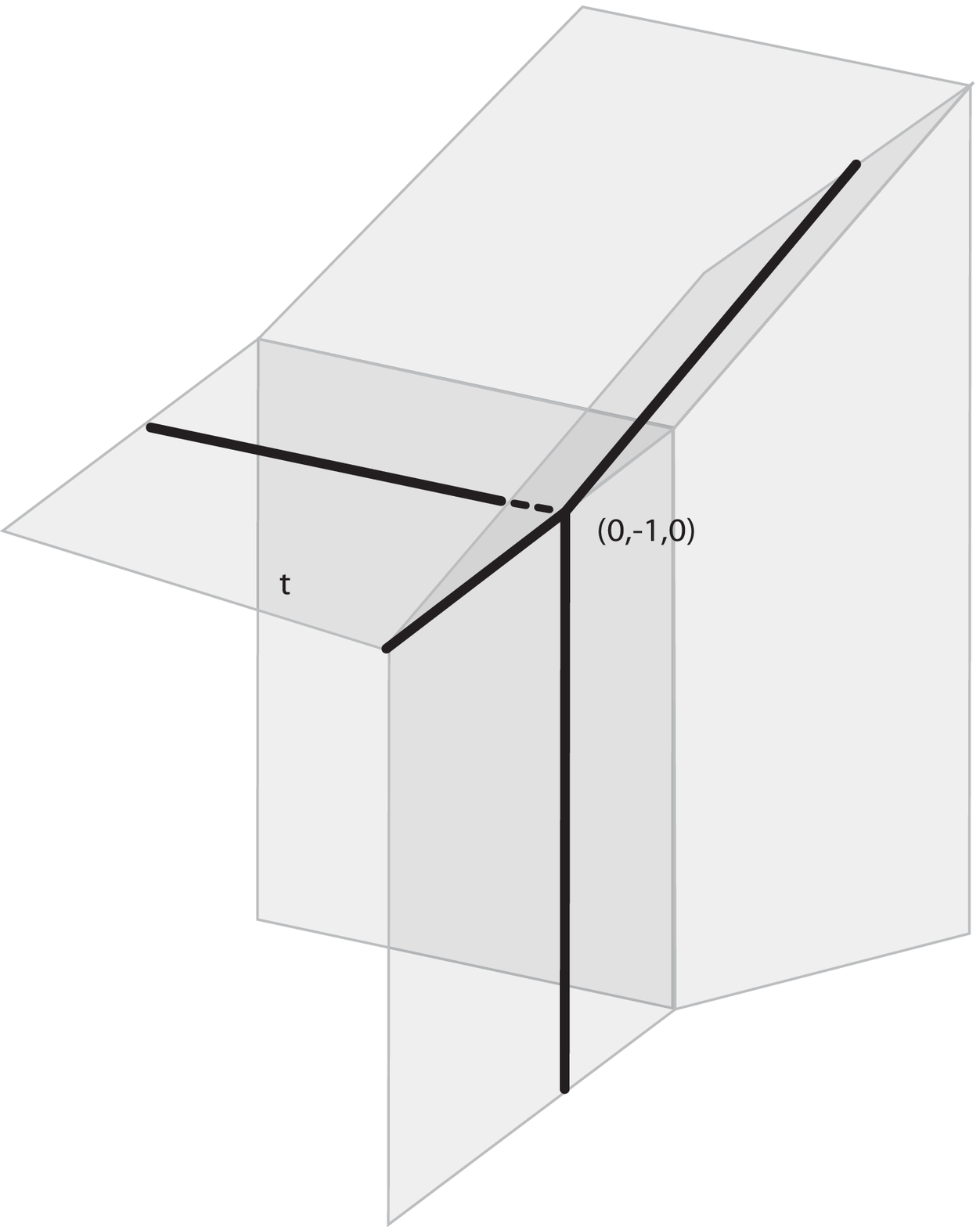}
\\ a) & b) & c)
\end{tabular}

\caption {Two tropical modifications of $Trop(X_2)$}
\label{7}
\end{center}
\end{figure}

To resolve the infinite set-theoretic intersection
$Trop(X_1)\cap Trop(X_2)$, we
 use one of the two lines, say $X_1$, to embed our
plane in $(\KK^*)^3$.
 
Let us denote
$Y=(\KK^*)^2\setminus X_1$ and consider the following map
$$\begin{array}{cccc}
\Phi:& Y&\longrightarrow &  (\KK^*)^3
\\ & (z,w)&\longmapsto & (z,w,\ P_1(z,w)).
\end{array}$$
The map $\Phi$ restricts to $X_2\setminus\{p\}$, and
$Trop(\Phi(X_2\setminus\{p\}))$ is a tropical curve in $\RR^3$ with an
edge starting at $(0,-2,-1)$ and unbounded in the direction
$(0,0,-1)$ (see Figure \ref{7}b). 
Clearly, this edge corresponds to $p$, and tells us that
$Val(p)=(0,-2)$.

\vspace{1ex}
Next section is devoted to
generalizing the method  used in 
this example.

\subsection{General method}\label{general}

Let $P(z)$ be a
polynomial in $n$ variables over $\KK$
and
denote $Y=(\KK^*)^n\setminus V(P)$. 
As in the
preceding section, this polynomial defines the following embedding of $Y$ to 
$(\KK^*)^{n+1}$ 
$$\begin{array}{cccc}
\Phi:& Y&\longrightarrow &  (\KK^*)^{n+1}
\\ & z&\longmapsto & (z,P(z)).
\end{array}$$

The tropical variety $W=Trop(\Phi(Y))$ is called the \textbf{tropical
  modification} of $\RR^n$ defined by $P(z)$. 
Since $\Phi(Y)$ has equation $z_{n+1}-P(z_1,\ldots,z_n)=0$, it follows
from Kapranov's Theorem that $W$ is given by the tropical polynomial
$\tg x_{n+1}+P_{trop}(x_1,\ldots,x_n)\td$. 
If $\pi^\KK_{n+1}:(\KK^*)^{n+1}\to (\KK^*)^{n}$ (resp. 
$\pi_{n+1}:\RR^{n+1}\to \RR^{n}$) denotes the projection forgetting the
last coordinate,  we obviously 
have $Trop\circ \pi^\KK_{n+1}=\pi_{n+1}\circ Trop$.

Since $W$ is a tropical hypersurface, many combinatorial properties of
 $W$ are straightforward:
the map $\pi_{n+1}$ restrict to a surjective map 
$\pi_{W}:W\to \RR^n$, one-to-one above $\RR^n\setminus V(P_{trop})$;
if $p\in V(P_{trop})$, then $\pi_{W}^{-1}(p)$ is a half ray, unbounded
in the direction $(0,\ldots,0,-1)$; the weight of a facet $F'$ of $W$
is equal to $w(F)$ if $\pi_{W}(F')=F$ is a facet of  $V(P_{trop})$, and is
equal to 1 otherwise.

\begin{exa}\label{exa trop modif}
The tropical plane in Figure \ref{example}b is the tropical
modification of $\RR^2$ defined by a polynomial of degree 1.
\end{exa}

More generally, if $X$ is
 an algebraic variety in $(\KK^*)^n$  with no component contained in $V(P)$, 
the polynomial $P(z)$
 defines a divisor $D(P)$ on $X$,
and
the map $\Phi$ defines an
 embedding of $X'=X\setminus V(P)$ in $(\KK^*)^{n+1}$. The
 tropical variety 
$V'=Trop(X')$ is called the tropical modification of $V=Trop(X)$
 defined by $P(z)$.

Unlike
in the case of a tropical modification of $\RR^n$, the tropical
variety
$V'$ does not depend only on
first order terms of $X$ and $P(z)$. 
The map $\pi_{n+1}$ still restricts to a surjective map $\pi_{V'}:V'\to V$,
but this map is 
 one-to-one only above $V\setminus V(P_{trop})$, i.e. $\pi_{V'}$ could be
 not injective not only on $\pi_{V'}^{-1}(Trop(V(P)\cap X))$ but also on
the (potentially strictly) bigger set $\pi_{V'}^{-1}(V(P_{trop})\cap V)$.
Hence
very few combinatorial
properties of $ \pi_{V'}^{-1}(V(P_{trop}))$ 
can be deduced in general only from those of $V$ and
$V(P_{trop})$:
 the set 
$\pi_{V'}^{-1}(p)$ is a bounded set if $p\in V\setminus Trop(V(P)\cap X)$,
and unbounded in the direction $(0,\ldots,0,-1)$ otherwise;
the weight of a facet $F'$ of $V'$, unbounded in the direction
$(0,\ldots,0,-1)$ and 
 such that  $\pi_{V'}(F')=F$ is a
facet of $Trop(D(P))$, is equal to 
$w(F)$ (recall that by definition, each
  facet of $Trop(D(P))$ has  a weight);
the weight of a facet $F'$ of $V'$ not contained in 
$\pi_{V'}^{-1}(V(P_{trop})\cap V))$ is equal to the weight of the
facet of $V$ containing  $\pi_{V'}(F')$.

\begin{exa}
Let us illustrate the dependency of
tropical modifications  on higher
order terms by going on with the example of preceding section. Recall
that $X_1$ and $X_2$ are the two lines in $\RR^2$ given by
$$X_1:\ (1+t^2)+z+w=0\quad  \text{and}\quad  X_2:\ (1+t) + z + t^{-1}w=0.$$
We
have already seen that the tropical modification of $Trop(X_2)$ along
$X_1$ is a tropical line in $\RR^3$ with two 3-valent vertices. Consider
now the line $X_3$ defined by the equation $2+z+w=0$. Note that
$Trop(X_1)=Trop(X_3)$. 
Since $X_2$ and
$X_3$ 
intersect at $(-2-t,t)$ which has valuation $(0,-1)$, the  tropical
modification of $Trop(X_2)$ along 
$X_3$ is a tropical line in $\RR^3$ with one 4-valent vertex (see
Figure \ref{7}c).
\end{exa}

\begin{exa}
Consider the curve $X$ in $(\KK^*)^2$ given by the equation
$$Q(z,w)=(t^{-9}+t^{-5}+1) + (2t^{-4}+1)z + tz^2 + (1+ t^{-5})w + 2zw
+
t^5z^2w +t^3w^2.$$
The tropical curve $C=Trop(X)$ 
and its tropical modification $C'$
given by the polynomial 
$P(z,w)=z+t^{-5}$ 
are depicted in Figure
\ref{modif non inj}. In
particular $C$ is a singular tropical curve, and the map
 $\pi_{C'}$ is
 not injective  on a strictly bigger set than
$\pi_{C'}^{-1}(Trop(V(P)\cap X))$.
To see that $C'$ is as depicted in Figure \ref{modif non inj},
 just notice that the
projection $\RR^3\to \RR^2$ forgetting the first coordinate sends
$C'$ to $C''=Trop(V(Q''))$ where $Q''(z,w)=Q(z-t^{-5},w)$, and compute easily
$$Q''(z,w)= 1 + z+ w + t z^2 + t^3 w^2 + t^5 wz^2.$$
\end{exa}
\begin{figure}[h]
\psfrag{2}{2}\psfrag{C}{C}\psfrag{C'}{C'}\psfrag{C''}{C''}
\centering
\begin{tabular}{c}
\includegraphics[height=6cm, angle=0]{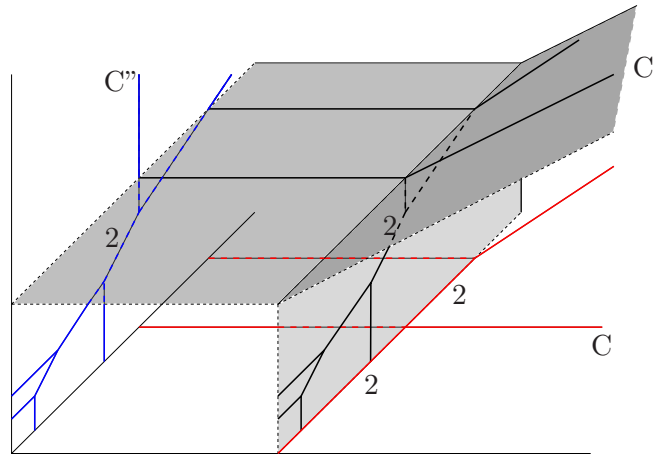}
\end{tabular}
\caption{The tropical curves $C$ and $V(P_{trop})$ do not determine $C'$.}
\label{modif non inj}
\end{figure}

\subsection{The case of plane curves}
In the case of plane curves, the situation is much simpler than the
general case discussed above. In particular, we  will see 
that the tropical modification of a non-singular plane tropical curve
depends on very few combinatorial data. An edge unbounded in the
direction $(0,0,-1)$ of a tropical curve  $C$ will be called a
\textbf{vertical end}  of $C$.

Let $P_1(z,w)$ and $P_2(z,w)$ be two polynomials 
  defining respectively the curves $X_1$ and $X_2$ in  $(\KK^*)^2$,
  such that $X_1$ and $X_2$ have no irreducible component in common.
We denote $C_i=Trop(X_i)$, and by $C'_1$ the tropical modification of
  $C_1$ given by $P_2(z,w)$.

Next Lemma is a restatement  in the particular case of plane curves 
of the 
material we discussed in  section
\ref{general}.
\begin{lemma}\label{intersection}
If $e$ is a vertical end of $C'_1$, then
$\pi_{C_1'}(e)\in Trop(X_1\cap X_2)$. 

Conversely, if $p \in Trop(X_1\cap X_2)$, then $\pi_{C_1'}^{-1}(p)$ contains 
a vertical end $e$ of $C'_1$, and
$$w(e)=\sum_{q\in X_1\cap X_2\cap Val^{-1}(p)}(X_1\circ X_2)_q. $$
\end{lemma} 

Tropical modifications allow us to relate easily intersection in
$(\KK^*)^2$ and tropical intersection.
\begin{prop}\label{set intersection}
Let $E$ be a component of $C_1\cap C_2$, and let $m$ be the sum of the
weight of all vertical ends in $\pi_{C'_1}^{-1}(E)$. Then
$$m\le  (C_1\circ_\TT C_2)_E$$
and equality holds if $E$ is compact. 
\end{prop}
\begin{proof}
Let $e_1,\ldots,e_r$ (resp. $\widetilde e_1,\ldots,\widetilde e_s$) 
be the edges of $C'_1$ which are not contained in
$\pi_{C'_1}^{-1}(E)$ but adjacent to a vertex $v_i$ in
$\pi_{C'_1}^{-1}(E)$ (resp. which are unbounded but not vertical,
and contained in $\pi_{C'_1}^{-1}(E)$). See Figure \ref{Edges}.

\begin{figure} [htbp]
\begin{center}
\psfrag{v}{$v$}\psfrag{1}{$e_1$}\psfrag{2}{$e_2$}\psfrag{3}{$e_3$}\psfrag{E}{$\pi^{-1}_W(E)$}
\psfrag{4}{$e_4$}\psfrag{e}{$\widetilde e_1$}\psfrag{5}{$v_1$}\psfrag{6}{$v_2$}
\begin{tabular}{cc}
\includegraphics[height=5cm, angle=0]{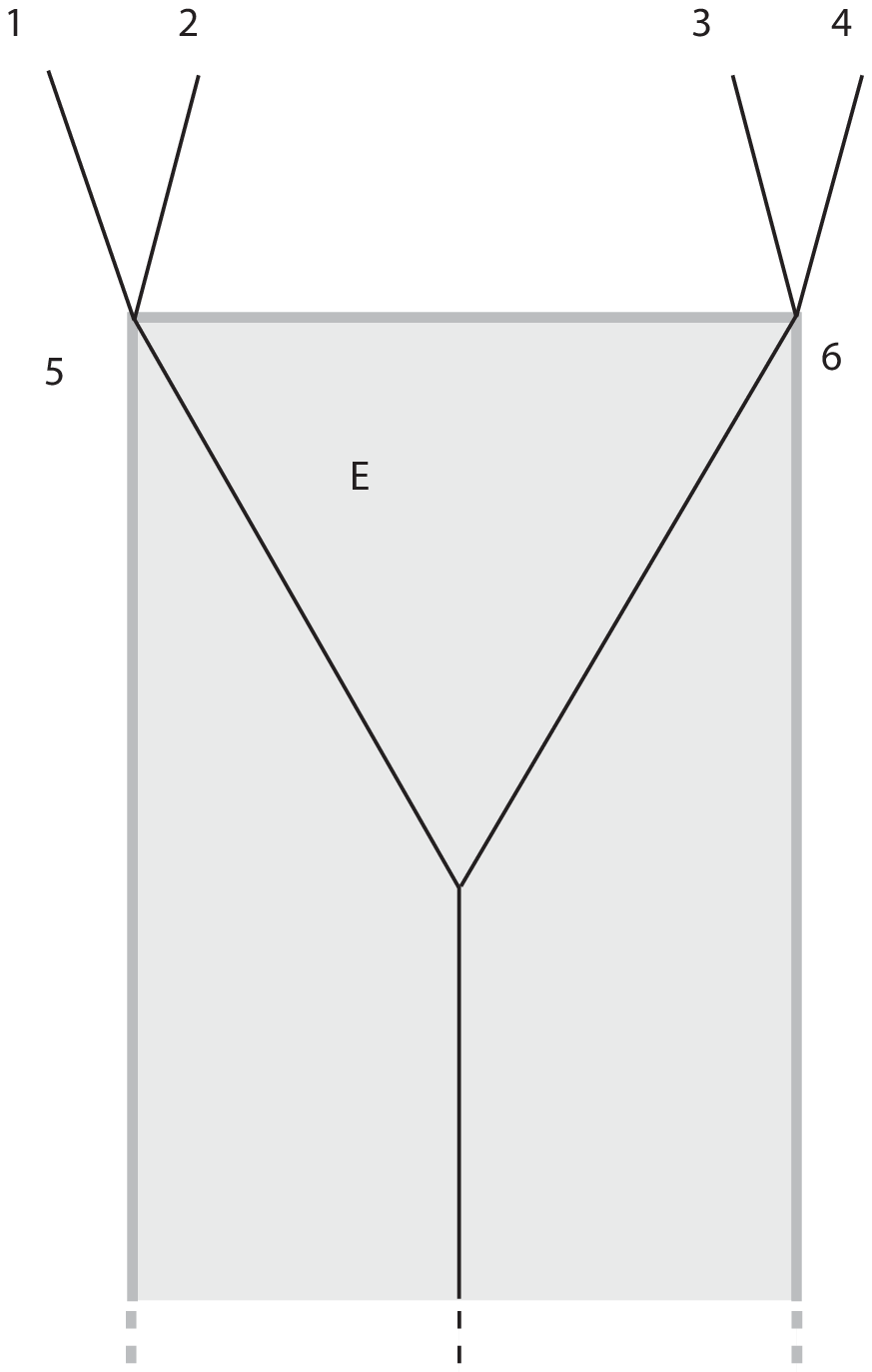}
&
\includegraphics[height=5cm, angle=0]{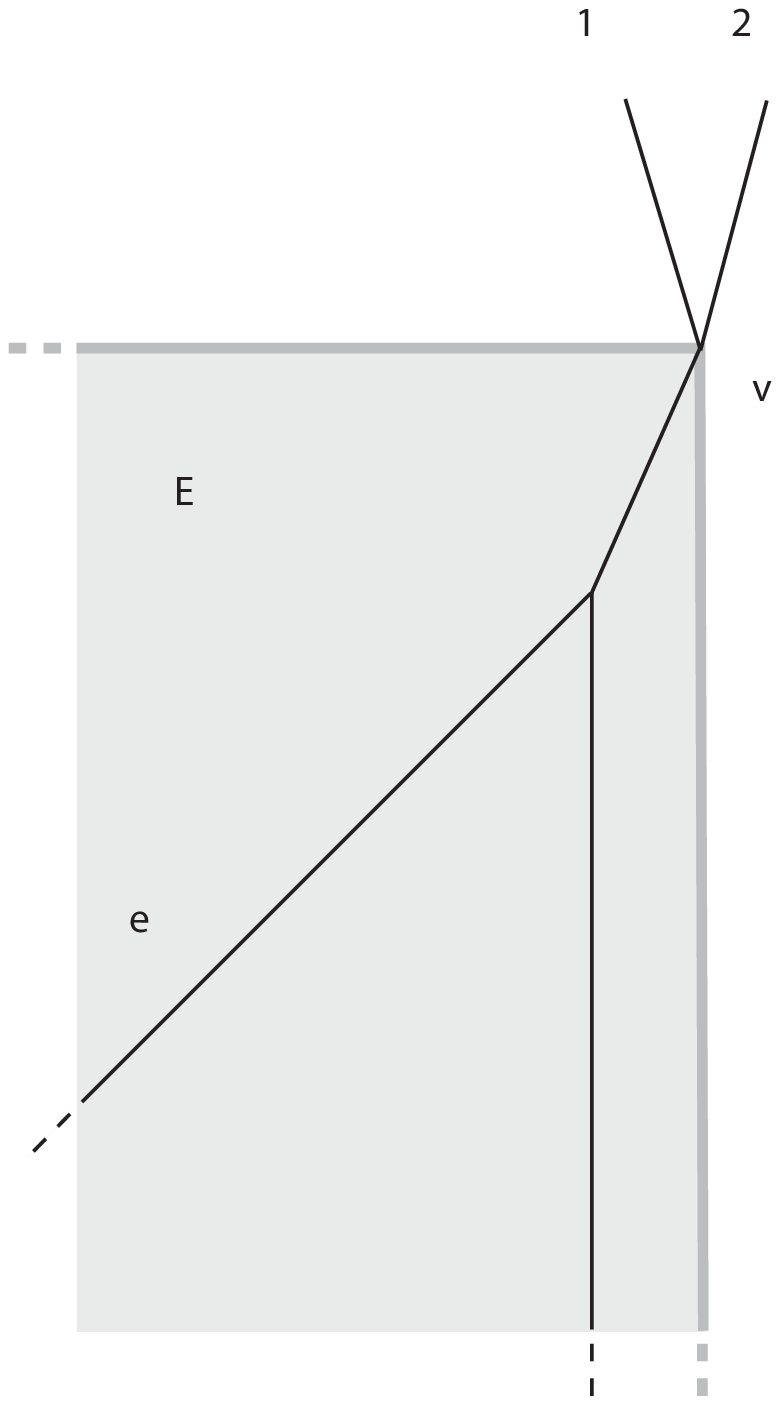}
\\ a) The compact case ($s=0$). & b) The non compact case ($s>0$).
\end{tabular}

\caption {}
\label{Edges}
\end{center}
\end{figure}

Let $(x_i,y_i,z_i)$ (resp. $(\widetilde x_i,\widetilde
y_i,\widetilde z_i)$)
be the primitive
integer direction of $e_i$ (resp. $\widetilde e_i$) 
pointing away from $v_i$ (resp. pointing to
infinity).
 Then, it follows
from
 the
balancing condition that
$$m=\sum_{i=1}^r w(e_i)z_i + \sum_{i=1}^s w(\widetilde e_i)\widetilde z_i.$$
Note that the balancing condition implies that
if $E$ is compact (i.e. when $s=0$), then 
 the integer $m$ depends only on $C_1$ and $C_2$. 
In the case 
where $E$ is not compact, we define an integer $m'$ as follows. 

We denote by $W$ the tropical modification of $\RR^2$ given by
$P_2(z,w)$. 

For any $1\le i \le s$, we denote by $\widehat e_i$ the non vertical
edge of $W$ such that $\pi_W(\widehat e_i)\cap \pi_W(\widetilde
e_i)\neq \emptyset$, and by
$(\widehat x_i, \widehat y_i, \widehat
z_i)$ the primitive vector of $\widehat e_i$ pointing to
infinity. Then, $(\widetilde x_i,\widetilde y_i)=\lambda (\widehat
x_i,\widehat y_i)$ with $\lambda$ a positive rational number. 
Since  $\widetilde e_i$ is contained in $\pi_W^{-1}(E)$, the slope of $\widetilde e_i$ is bounded by the slope of $\widehat e_i$. Hence, we necessarily have $\widetilde z_i\le \lambda \widehat z_i$ with equality if and only if $\widetilde e_i$ and $\widehat e_i$ are parallel (see Figure \ref{Slopes}).

Let us define 
$$m'=m+ \sum_{i=1}^s  w(\widetilde e_i)\left( {\lambda}\widehat z_i -
\widetilde z_i\right).$$ 
Hence we have $m\le m'$. In particular, $m=m'$ if and only if all the unbounded edges of $\pi_{C'_1}^{-1}(E)$ are vertical ends or parallels to the corresponding edge of $W$.

Now the balancing condition implies that
the integer $m'$ only depends on $C_1$
and $C_2$, and not any more on $X_1$ and $X_2$.

\begin{figure} [htbp]
\begin{center}
\psfrag{b}{$(\widetilde x_i, \widetilde y_i,\widetilde{z_i})$}\psfrag{a}{$(\widehat x_i,\widehat y_i, \widehat z_i)$}\psfrag{e}{$\widetilde e_i$}\psfrag{f}{$\widehat e_i$}
\begin{tabular}{cc}
\includegraphics[height=5cm, angle=0]{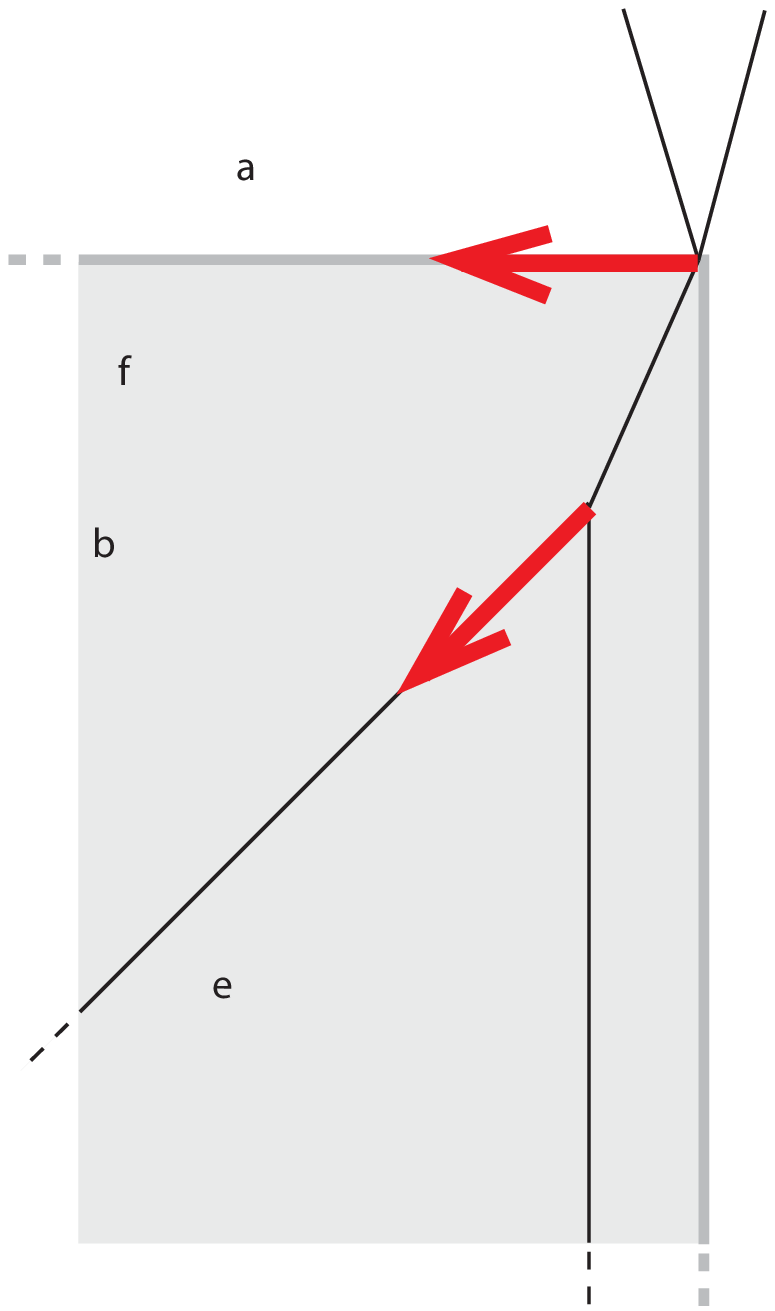}
&
\includegraphics[height=5cm, angle=0]{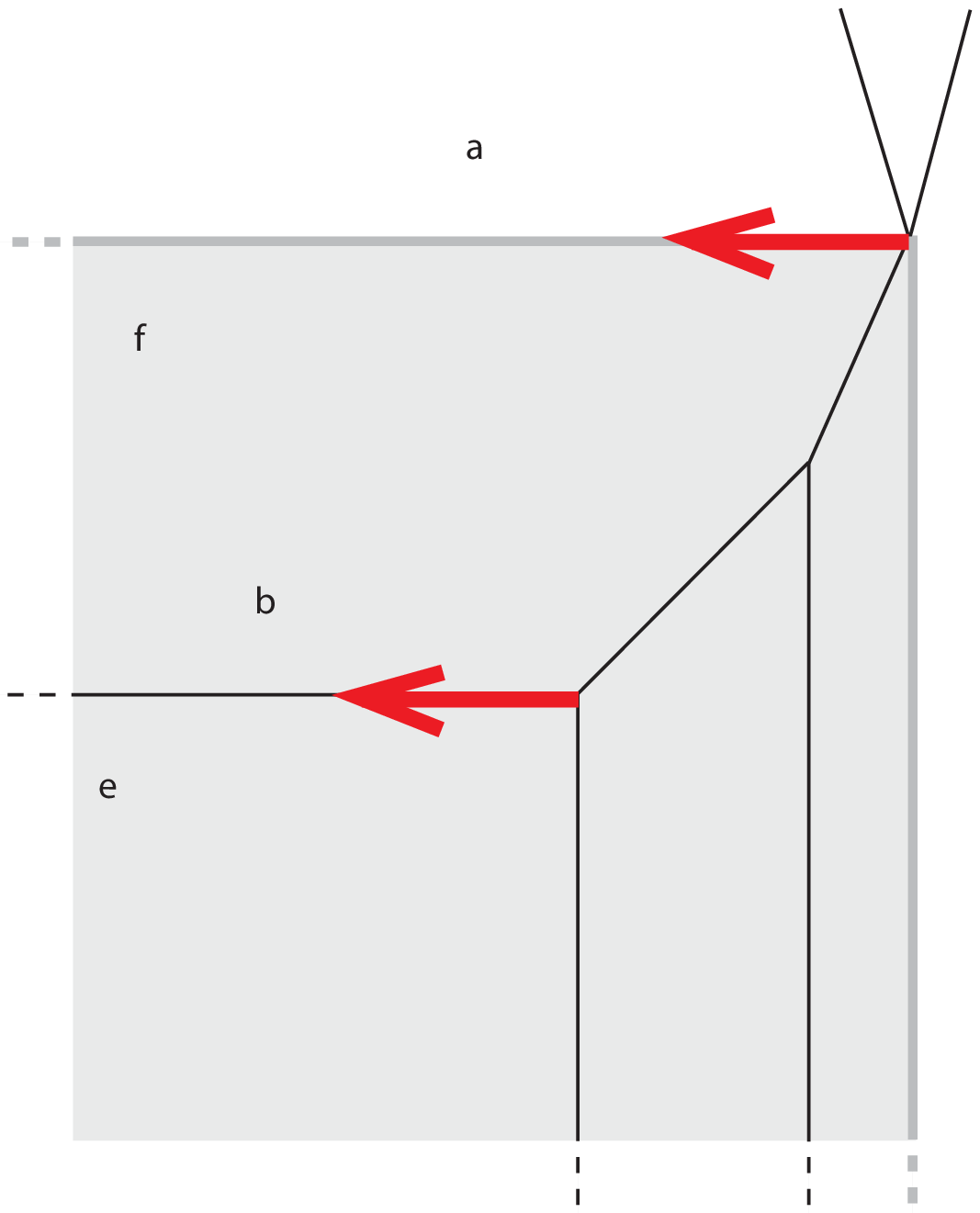}
\\ a) $\widetilde z_i<\lambda \widehat z_i$. & b) $\widetilde z_i= \lambda \widehat z_i$.
\end{tabular}

\caption {The slope of $\widetilde e_i$ is bounded by the slope of $\widehat e_i$.}
\label{Slopes}
\end{center}
\end{figure}

Hence to conclude, it remains to prove that $m'=(C_1\circ_\TT C_2)_E$. To do it, it is sufficient to prove that there exists a
balanced graph $\Gamma$ with rational slopes in $W$ (i.e. a tropical
1-cycle in the terminology of \cite{Mik3}) such that
$\pi(\Gamma)=C_1$, and such that $\pi_{\Gamma}^{-1}(p)$ is a vertical end of
weight $\frac{1}{2}(Area(\Delta_p)-\delta_p)$ if $p\in
C_1\circ_{\TT,E} C_2$,
and
$\pi_{\Gamma}^{-1}(p)$ is a point otherwise.
The existence of $\Gamma$ is clear if $E$ is the isolated 
intersection of two
edges of $C_1$ and $C_2$. The general case reduces to the latter case
via stable intersections (see \cite{St2}).
\end{proof}

Note that Proposition \ref{set intersection} and its proof do not
depend of the algebraically closed ground field $\KK$, and generalize
to intersections of tropical varieties of higher dimensions. The
existence of $\Gamma$ can also be established using the fact that
tropical intersections are tropical varieties (see \cite{Pay1}). 

Using
a more involved tropical intersection theory (see for example
\cite{Mik3}, \cite{Kris1}, or \cite{AllSmooth}), the proof of
Proposition \ref{set intersection} 
should generalize easily when replacing 
the ambient space $\RR^n$ by any smooth realizable tropical variety.

\begin{lemma}\label{weight modif}
Let $p$ be a point on an edge $e$ of $C_1$ 
such that $\pi_{C_1'}^{-1}(p)$ does not contain any vertex of $C'_1$,
 and denote by $e_1,\ldots, e_l$ the edges of $C'_1$ containing
a point of $\pi_{C_1'}^{-1}(p)$. Then
$$\sum_{i=1}^l w(e_i)\le w(e).$$
\end{lemma}
\begin{proof}
By assumption, the set $\pi_{C_1'}^{-1}(p)$ is
finite.
Without loss of generality, we may assume that there exists a
tropical line $L$ given by $\tg y+a \td$, with $a\in\TT$, having an
isolated intersection point with $e$ at $p$. 
Let $C_3$ be the non-singular 
tropical curve in $\RR^2$ containing $e$ and defined
by a binomial polynomial (in particular $C_3$ is a classical line).
If we denote by
$m_p$ the
number of intersection points of $X_1$ with the line of equation
$w+t^{-a}=0$, then we have
$$m_p=(C_3\circ_\TT L)_pw(e).$$
Now the result follows from the fact that for each $i$, at
least $(C_3\circ_\TT L)_pw(e_i)$ 
points in $X'_1\cap\{w+t^{-a}=0\}\subset (\KK^*)^3$ have valuation
contained in $e_i$.
\end{proof}

\begin{cor}\label{nonsingular}
If $C_1$ is a non-singular tropical curve 
with no component
contained in $C_2$, then $C'_1$
is entirely determined by $C_1$, $C_2$, and $Trop(X_1\cap X_2)$. More
precisely, we have

\begin{itemize}
\item $\pi_{C'_1}$ is one-to-one above $C_1\setminus Trop(X_1\cap X_2)$;

\item for any $p\in  Trop(X_1\cap X_2)$, the set $\pi_{C_1'}^{-1}(p)$
  is a vertical end of weight $w(p)$ (recall that by definition, each
  point in $Trop(X_1\cap X_2)$ comes with a multiplicity);

\item any edge $e$ of $C'_1$ which is not a vertical end is of weight 1.
\end{itemize}
\end{cor}
\begin{proof}
Let us denote by $W$ the tropical modification of $\RR^2$ given by $P_2(z,w)$.
According to Lemmas \ref{set intersection} and
 \ref{weight modif}, $C'_1$ is entirely determined by
the knowledge of $Trop(X_1\cap X_2)$, the direction of
one edge of $C_1'$ and one point of $C'_1$.
By hypothesis,  there exists
a point $p$ in $C_1\setminus C_2$ on an $e$ edge of $C_1$. Since
$\pi_{W}$ is one to one over $\RR^2\setminus C_2$, the point
$\pi_{C_1'}^{-1}(p)=\pi_{W}^{-1}(p)\in W$ is fixed, as well as the
direction of the edges of $C'_1$ passing through $\pi_{C_1'}^{-1}(p)$.
\end{proof}

\section{Tropicalization of  inflection points}
Now we come to the core of this paper. Namely, given a non-singular
tropical curve $C$, we study the possible
tropicalizations for the inflection points of a realization of $C$. 
Our main result is that for almost all tropical curves, there exists a finite
number of such points $p$ on $C$ and that the number of inflection points
which tropicalize to $p$ only depends on $C$, and not on
 the chosen realization.

Before going into the details, let us give an outline of our strategy.
Let $X$ be a realization of $C$, and $T$ be a tangent line to $X$ at
an inflection point $p$. 
First of all we prove in Proposition \ref{tangents}
 that  the vertex of $L=Trop(T)$ has
to be a vertex of $C$, which leaves only finitely many possibilities
for $L$.
In a second step, we refine in Proposition \ref{inclusion}
the possible locations of $Val(p)$ by
studying the tropical modifications of $C$ and $\RR^2$ defined by $T$.
In particular we identify finitely many subsets of $C$,
independant of
$X$ and called
\textbf{inflection components} of $C$, which may possibly contain $Val(p)$. Note
that these inflection components are 
often reduced to a point.
Finally, we prove in Theorem \ref{main} 
that the number of inflection points of $X$ with
valuation in a given inflection component $\E$  of $C$ only depends
on $\E$. We call this number the \textbf{multiplicity} of $\E$. 
The proof of Theorem \ref{main}  is postponed to section \ref{proof
  main}, and goes by the
study of 
the tropical  modification of $C$ and $\RR^2$ defined  by $Hess_X$.

\subsection{Inflection points of curves in  $(k^*)^2$}\label{sec:
  infl}
Let $k$ be any field of characteristic 0.
Given a (non-necessarily homogeneous) polynomial in two variables
$P(z,w)$, we denote by $P^{hom}(z,w,u)$ its homegeneization.
 Inflection points   of the curve $V(P)$   (recall that by
  definition $V(P)\subset (k^*)^2$)
 are defined as  the inflection points in 
$(k^*)^2= \{[z:w:u]\in kP^2\ | \ zwu\ne 0\}$ of the projective
curve defined by  $P^{hom}(z,w,u)$. 
Note that  inflection points of $V(P)$
 are invariant under the transformations $(i,j)\in\NN^2\mapsto
z^iw^jP(z,w)$ but are \textbf{not} invariant in general under 
invertible monomial transformations of $(k^*)^2$ (i.e.  automorphisms of
$(k^*)^2$).

Since the torus $(k^*)^2$ is not compact, the number of inflection
points  of $V(P)$
may depend 
on the coefficients of $P(z,w)$. Indeed, some inflection points could
escape from $(k^*)^2$ for some specific values of these coefficients.
However, we will see in Proposition \ref{max infl points} 
that this can happen only if $\Delta(P)$ contains
edges parallel to an edge of the simplex $T_1$.

\subsection{Location}\label{locate}
In the whole section, $X$ is an algebraic curve
in $(\KK^*)^2$ which is not a line, and
whose tropicalization is a non-singular tropical curve
$C$.
As we will see with Lemma \ref{subd} ,  it is hopeless to locate
the tropicalization of inflection points of $X$ looking at
$Trop(X)\cap Trop(Hess_X)$, this intersection being highly
non-transverse.
However, an inflection point $q$ of 
 $X$ comes together with its
tangent $T$, which has intersection at least $3$ with $X$ at $q$. 
It turns out that 
the determination of all possible $Trop(T)$ is a
 much 
 easier task.

Hence let  $q$ be an inflection point of $X$,
with tangent line $T$.
We denote by $L$ the tropical line
$Trop(T)$, 
by $p$ the point $Val(q)$, by $v$ the vertex of $L$, and by $E$ the
component  of $C\cap L$ containing $p$.

\begin{prop}\label{tangents}
The vertex $v$ is a common vertex of $C$ and $L$,
and 
 is contained in  $E$. Moreover, $E$
 is one of the following (see Figure \ref{Possibles}):
\begin{itemize}
\item reduced to $v$;
\item an  edge of $C$;
\item three adjacent edges of $C$, at least 2 of them being bounded.
\end{itemize}

\begin{figure} [htbp]
\begin{center}
\begin{tabular}{cccc}
\includegraphics[height=2cm, angle=0]{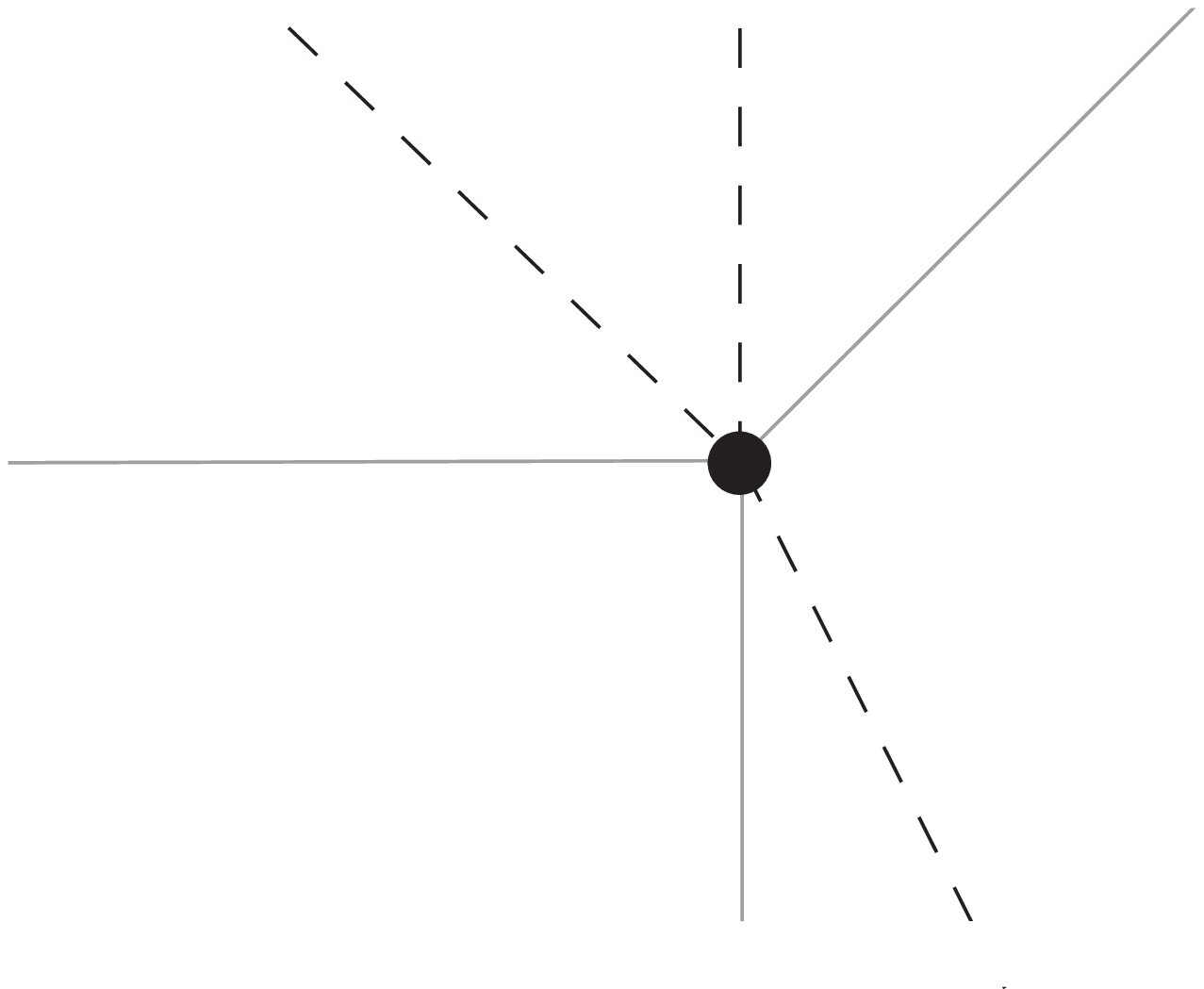}&
\includegraphics[height=2cm, angle=0]{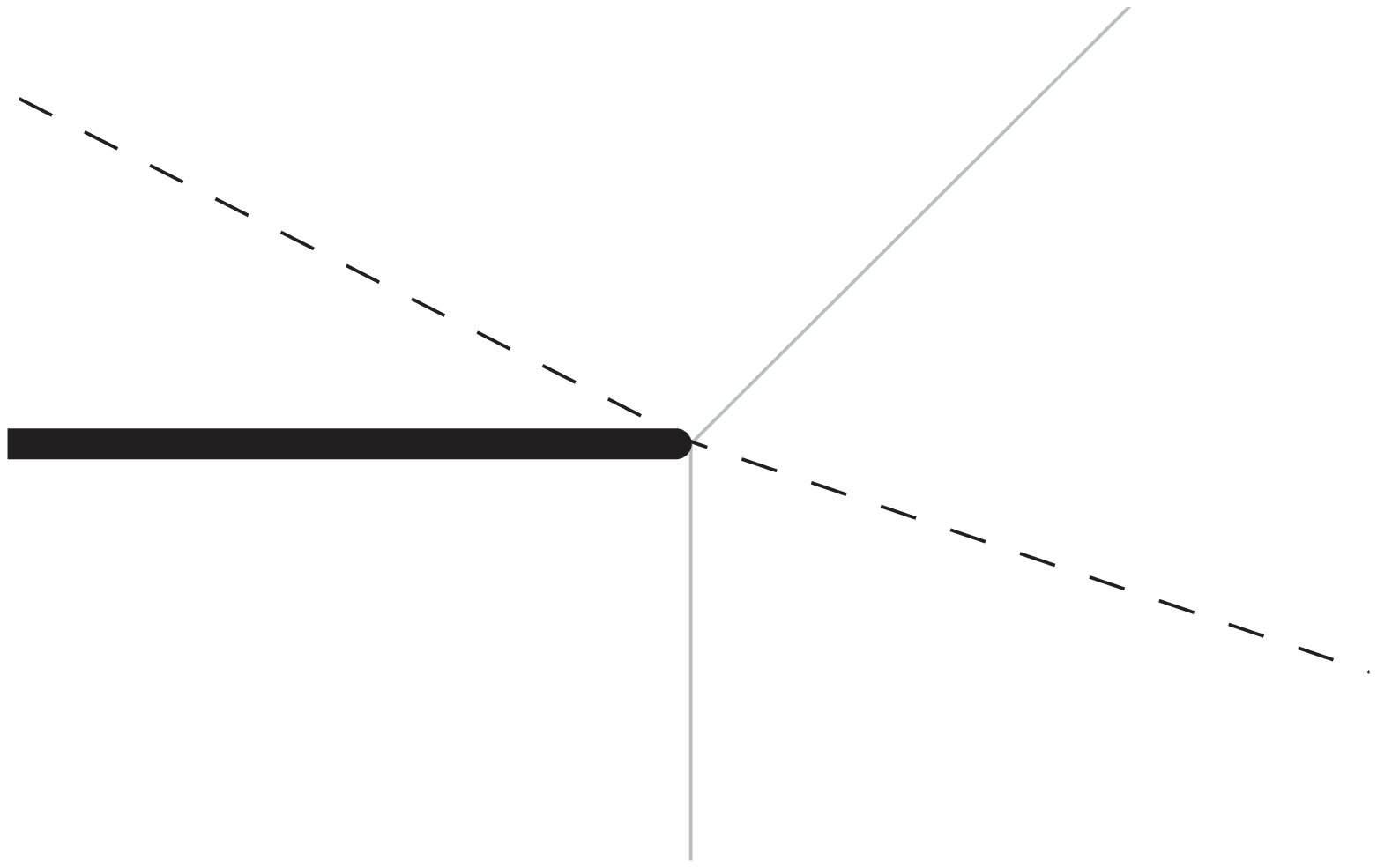}&
\includegraphics[height=2cm, angle=0]{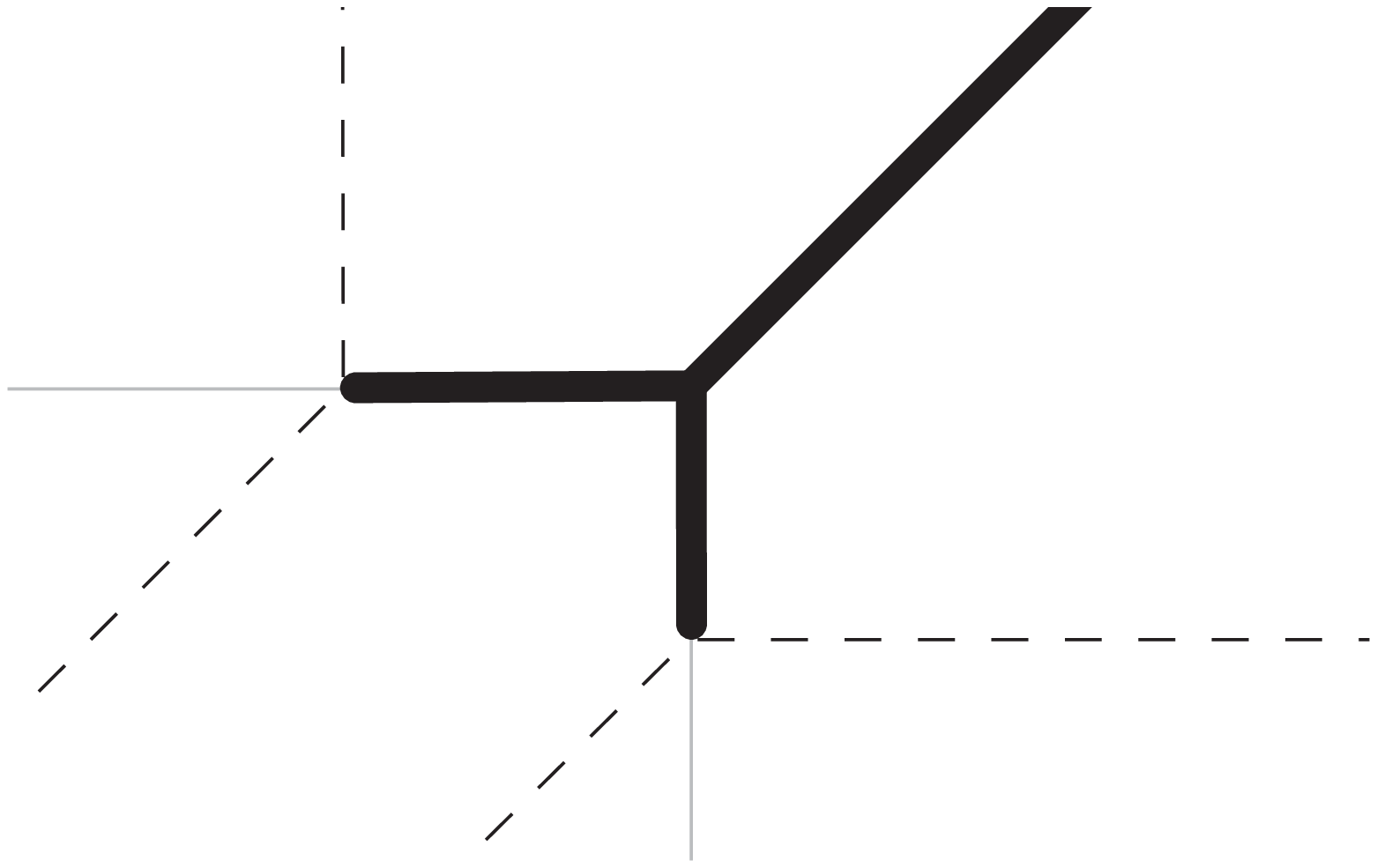}&
\includegraphics[height=2cm, angle=0]{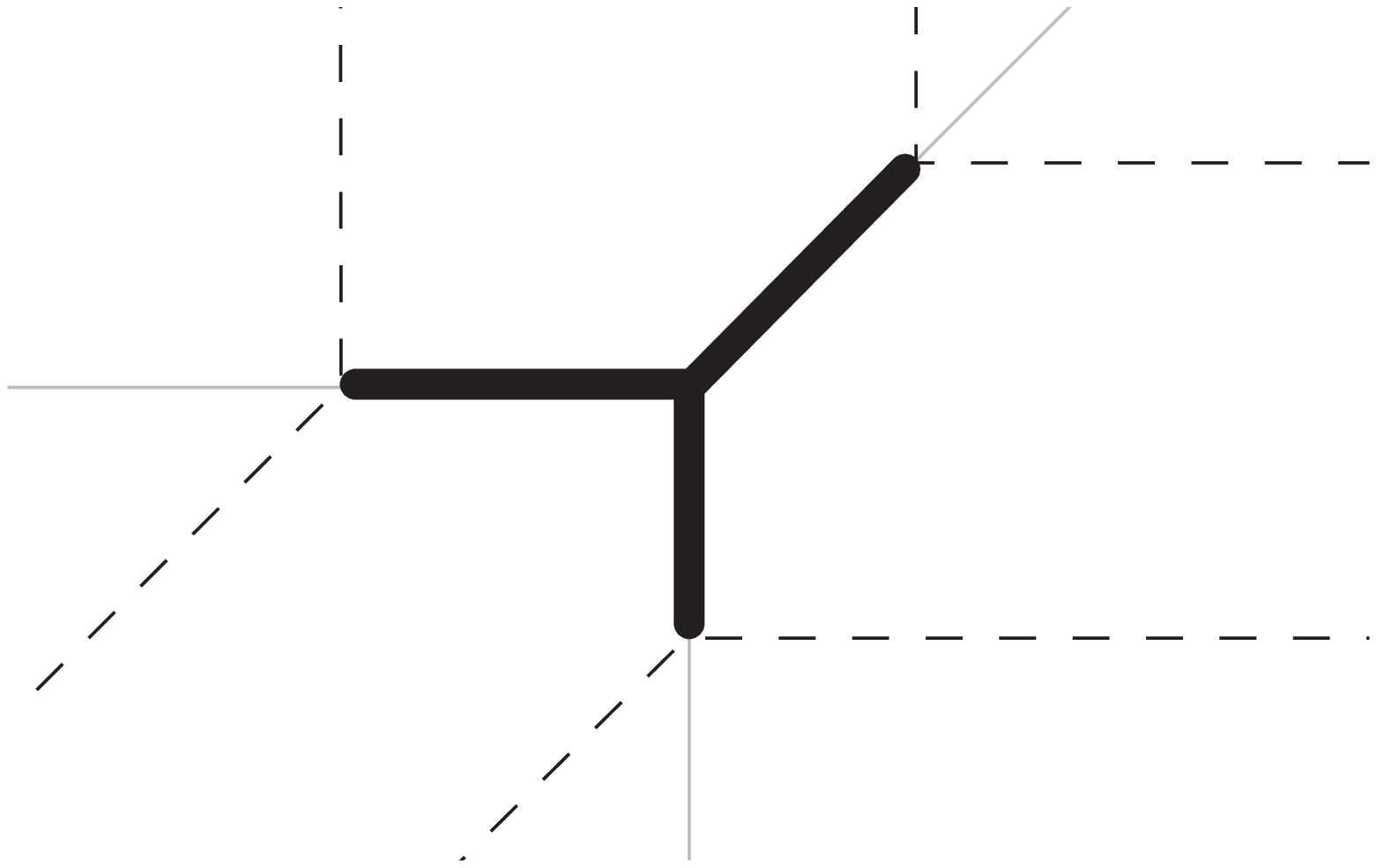}

\end{tabular}
\caption {Possible tropicalizations of third-order tangent
  lines. The intersection of $L$ and $C$ is in bold.}
\label{Possibles}
\end{center}
\end{figure}

\end{prop}
\begin{proof}
According to  
Proposition \ref{trop int}, if $E$ contains a tropical inflection point, then $(C\cap_\TT L)_E\geq 3$. Since
$C$ and $L$ are both non-singular, if $E$
does not fulfill the conclusion of the Proposition, then $E$ is one of
the following (see Figure \ref{3}):
\begin{itemize}
\item reduced to a point which is not a common vertex of both $C$ and
  $L$;
\item an unbounded edge of $C$ or $L$ but not of both;
\item a bounded segment which is not an edge of $C$ neither of
  $L$; 
\item three adjacent edges of $C$ with only one of them being
  bounded;
\end{itemize}
The conclusion in the first case follows from Lemmas \ref{no tang} and
\ref{no infl}. In the last three cases we have
$$(C\cap_\TT L)_E< 3.$$
Hence we excluded all cases not listed in the proposition.
\end{proof}

\begin{figure} [htbp]
\begin{center}

\begin{tabular}{cccc}
\includegraphics[height=2cm, angle=0]{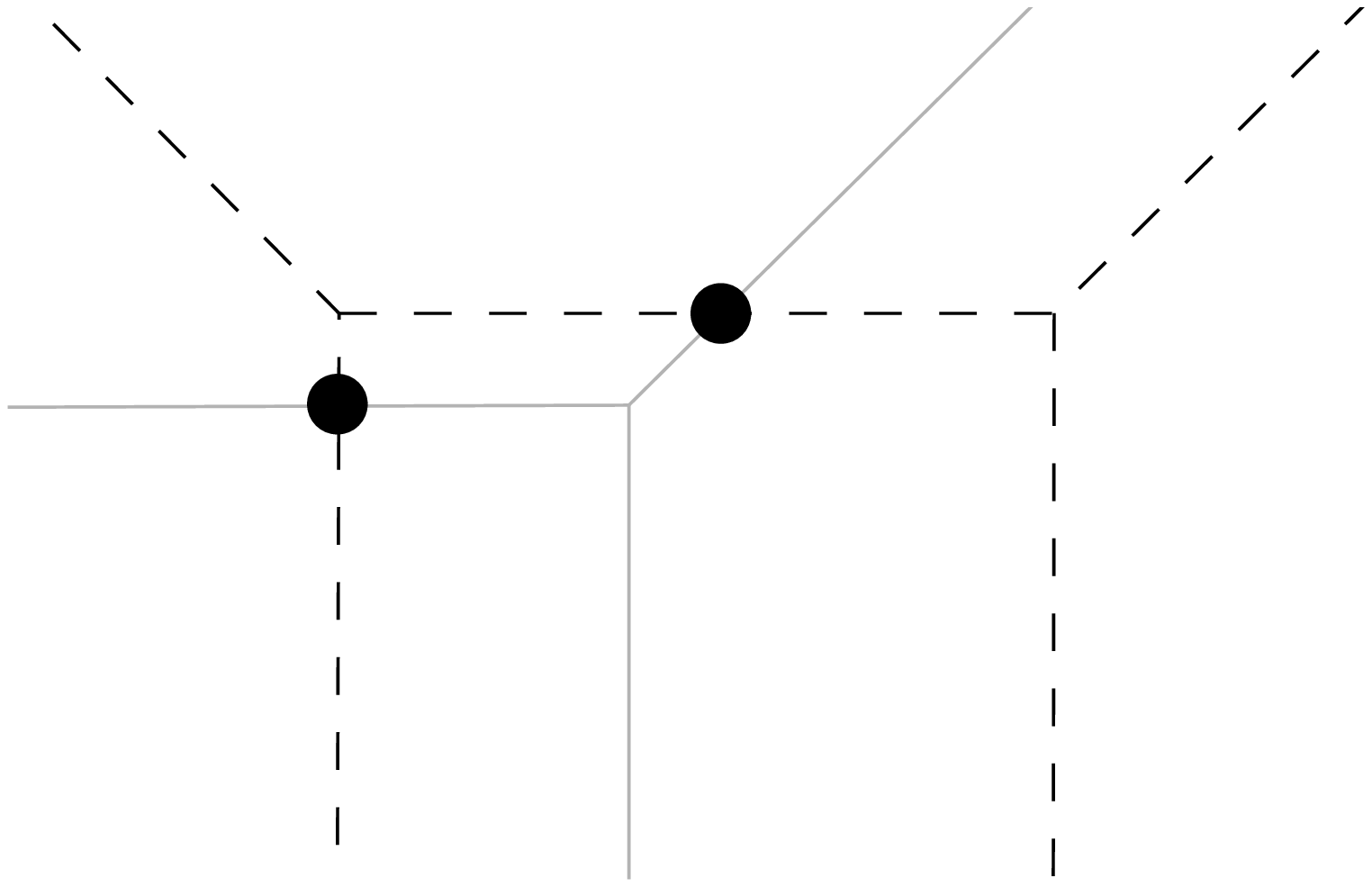}&
\includegraphics[height=2cm, angle=0]{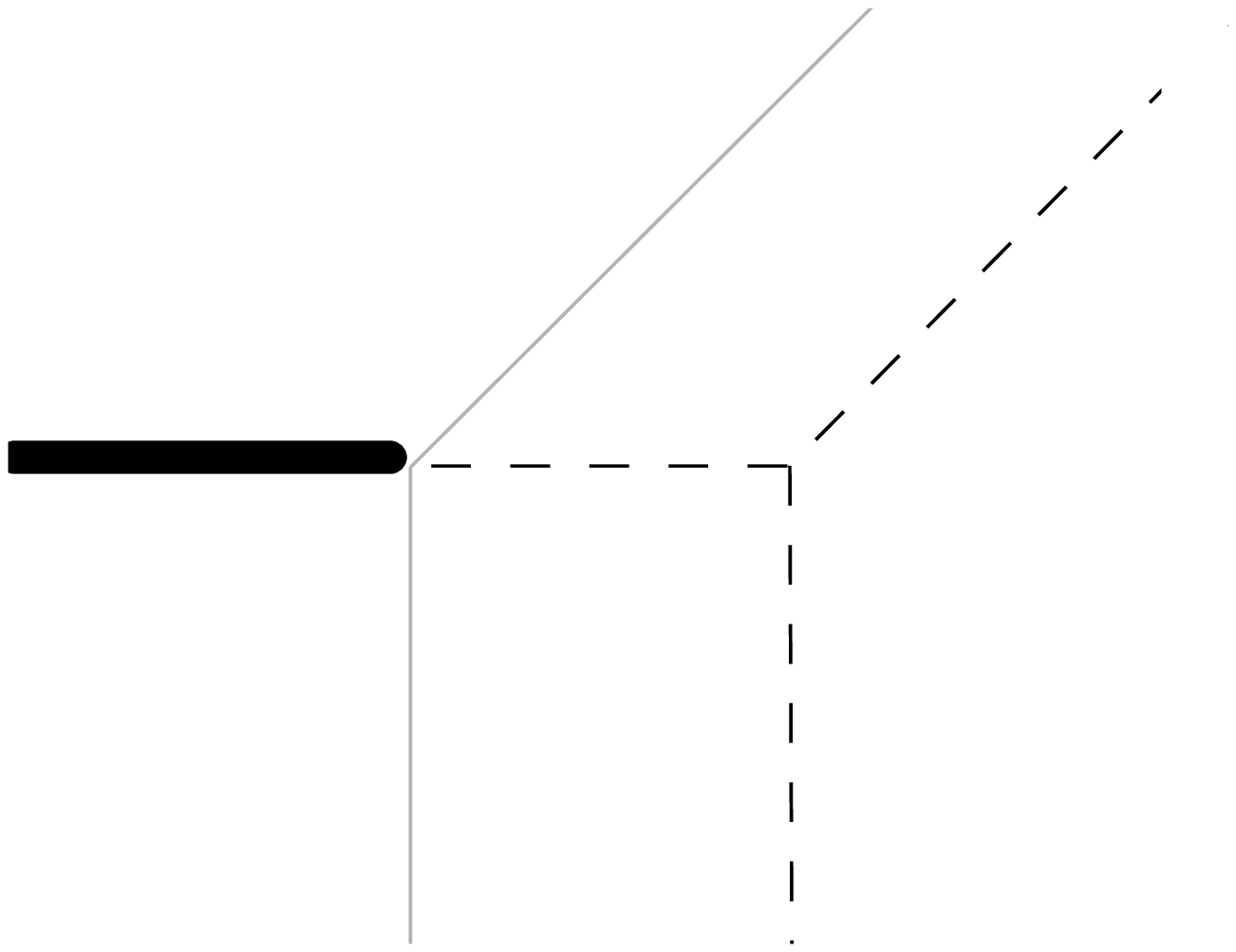}&
\includegraphics[height=2cm, angle=0]{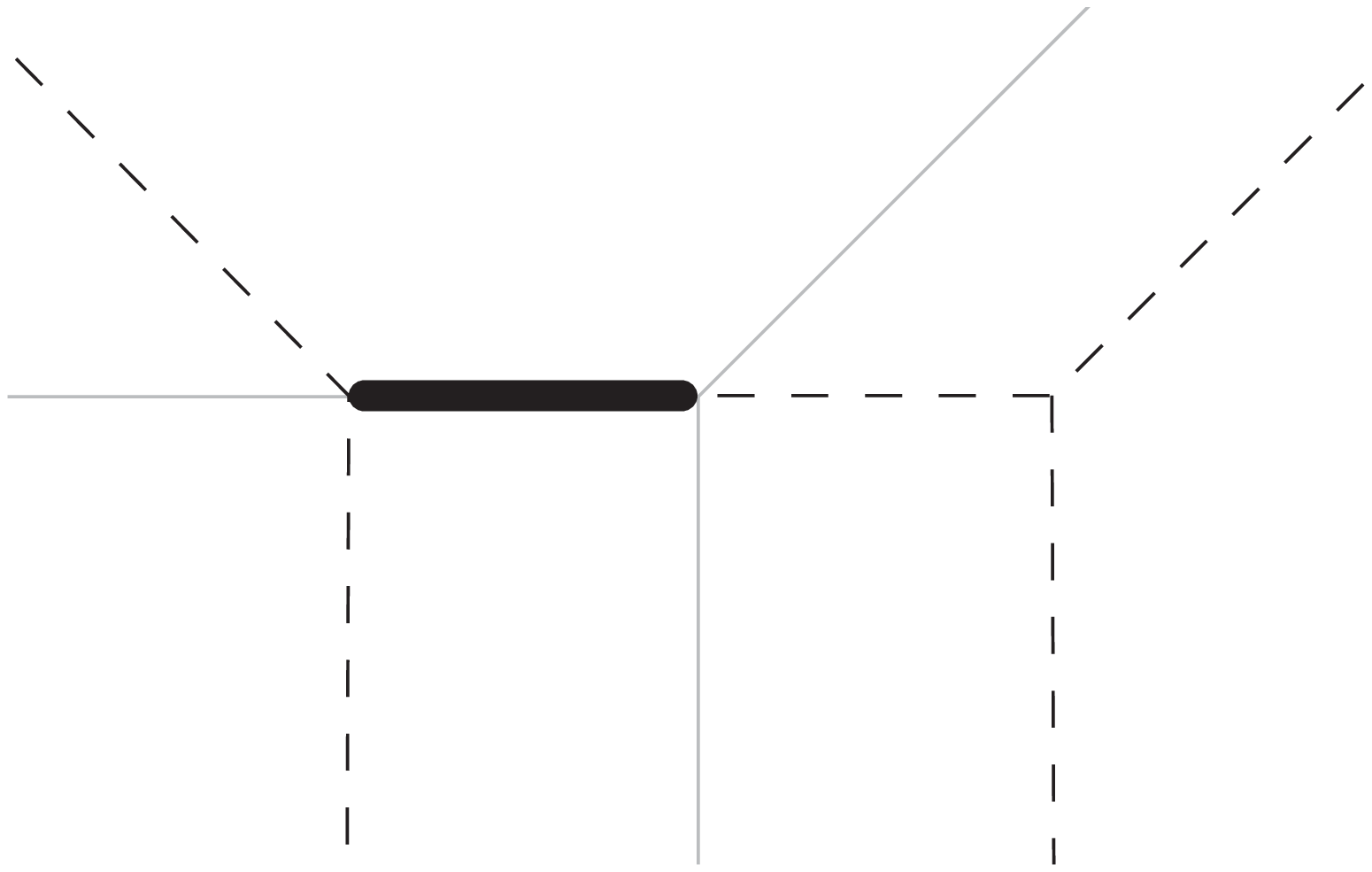}&
\includegraphics[height=2cm, angle=0]{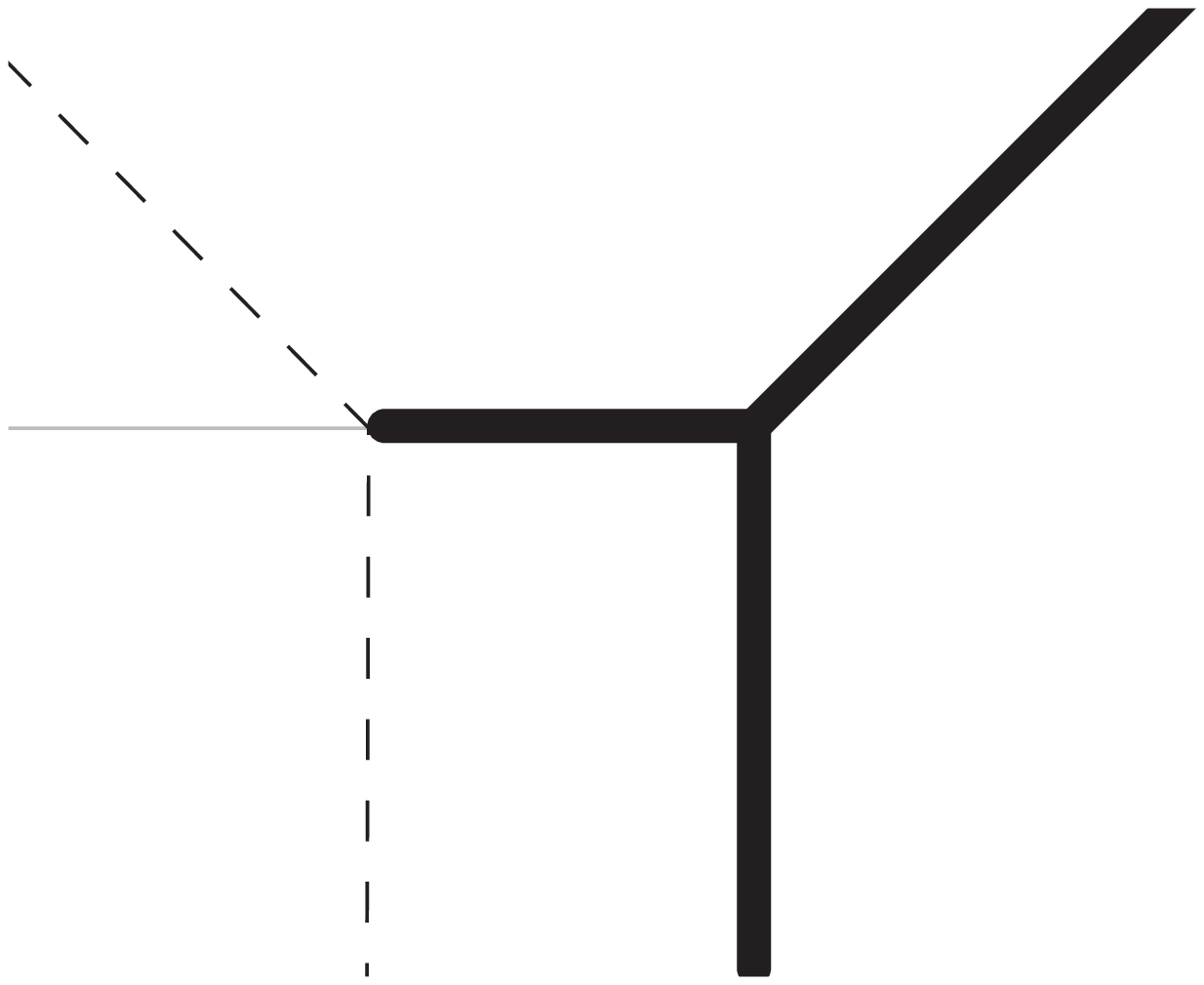}

\end{tabular}

\caption {Impossible tropicalizations of third-order tangent lines.}
\label{3}
\end{center}
\end{figure}

An immediate and important consequence of Proposition \ref{tangents} is that
 since the vertex of $L$ must be a vertex of $C$, there 
are only finitely many possibilities for $L$. 
For each of
these possibilities, 
we use  tropical modifications to get a
refinement on the
possible locations of $p$. If $E$ contains a bounded edge $e$, then 
we denote  by $v_e$ the other vertex of $C$ 
adjacent to $e$.  Recall that $l(e)$ is the length of $e$ as an edge of $C$.
We define the subset $\I_L$ of
$E$ as follows :
\begin{itemize}
\item if $E=\{v\}$ or $E$ is an unbounded edge of $L$, then
  $\I_L=\{v\}$ (see Figure \ref{11}a and \ref{11}b);

\item if $E$ is a bounded edge $e$ of $C$,  then
  $\I_L=\{v,p_e\}$ where  $p_e$ is 
the point on $e$ at distance $\frac{l(e)}{3}$ from $v$ (see Figure \ref{11}c);

\item if $E$ is the union of 2 bounded edges $e_1$, $e_2$, and one
  unbounded edge $e_3$, then
 \begin{itemize}
\item if   $l(e_1)> l(e_2)$, then $\I_L=\{p_{e_1}\}$ where  $p_{e_1}$ is 
the point on $e_1$ at distance $\frac{l(e_1)-l(e_2)}{3}$ from $v$ (see Figure \ref{11}d);
 \item if  $l(e_1)=l(e_2)$, then $\I_L$ is the whole edge $e_3$ (see Figure \ref{11}e);
\end{itemize}

\item if $E$ is the union of 3 bounded edges $e_1$, $e_2$, and $e_3$, then
\begin{itemize}
\item if   $l(e_1)\ge l(e_2)> l(e_3)$, then
  $\I_L=\{p_{e_1},p_{e_2}\}$ where  $p_{e_i}$ is  
the point on $e_i$ at distance $\frac{l(e_i)-l(e_3)}{3}$ from $v$ (see Figure \ref{11}f);
 \item if  $l(e_1)>l(e_2)=l(e_3)$, then $\I_L$ is the whole segment
   $[v;p_{e_1}]$ 
where  $p_{e_1}$ is 
the point on $e_1$ at distance $\frac{l(e_1)-l(e_2)}{3}$ from $v$ (see Figure \ref{11}g);
\item if   $l(e_1)=l(e_2)=l(e_3)$, then $\I_L=\{v\}$ (see Figure \ref{11}h).
\end{itemize}

\begin{figure} [htbp]

\begin{center}

\begin{tabular}{cccc}
\includegraphics[height=2.5cm, angle=0]{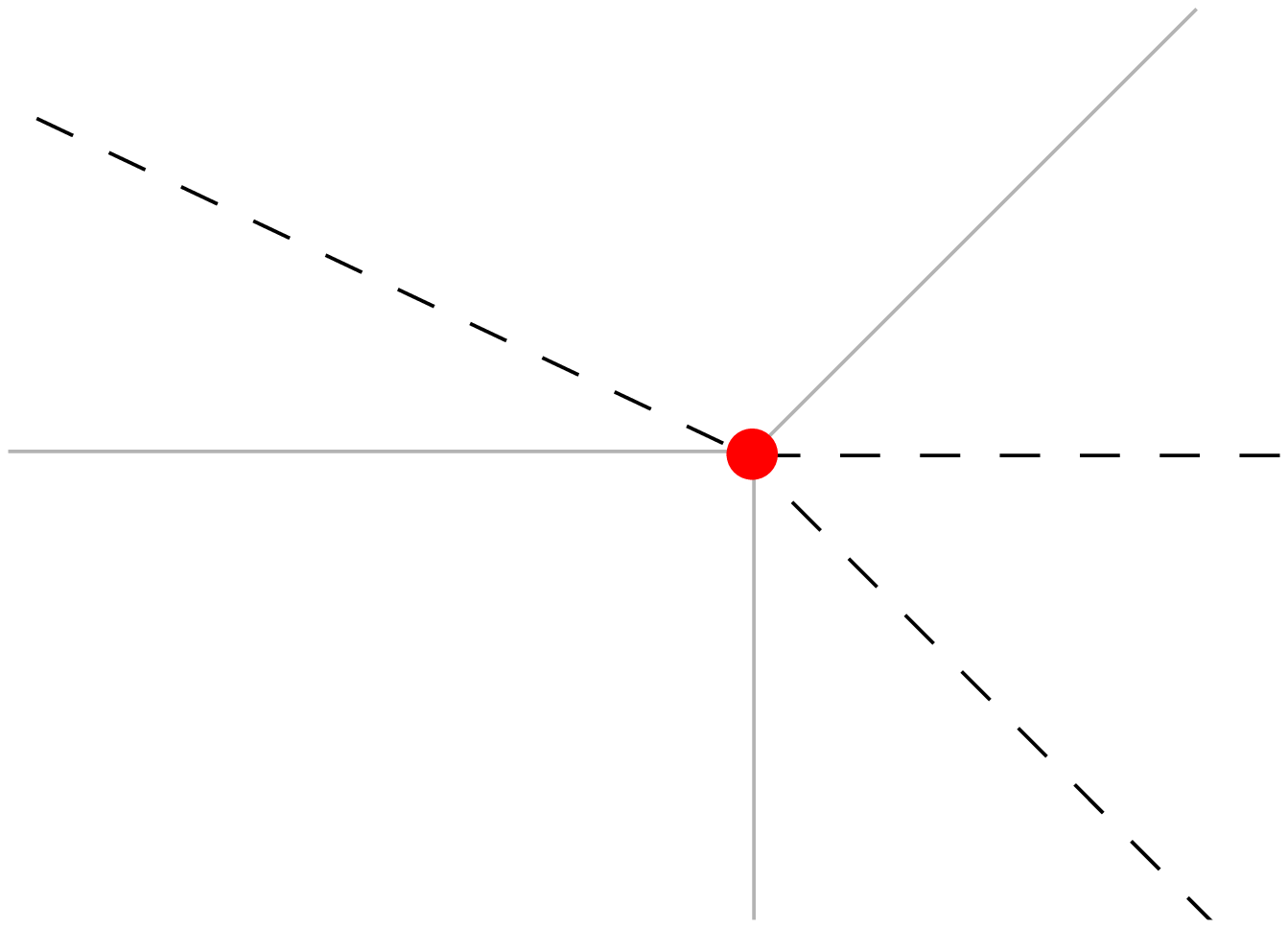}&
\includegraphics[height=2.5cm, angle=0]{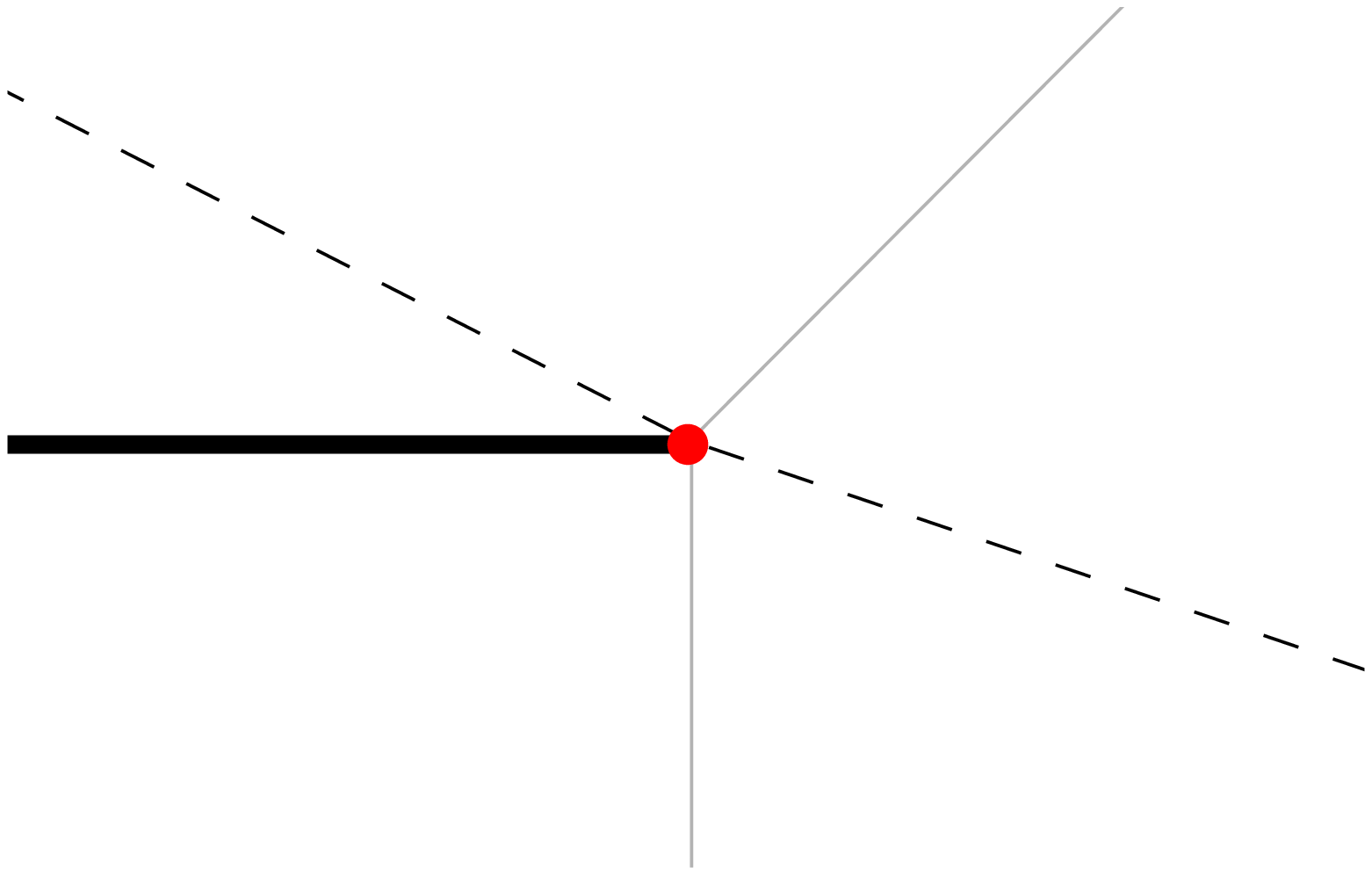}&
\includegraphics[height=2.5cm, angle=0]{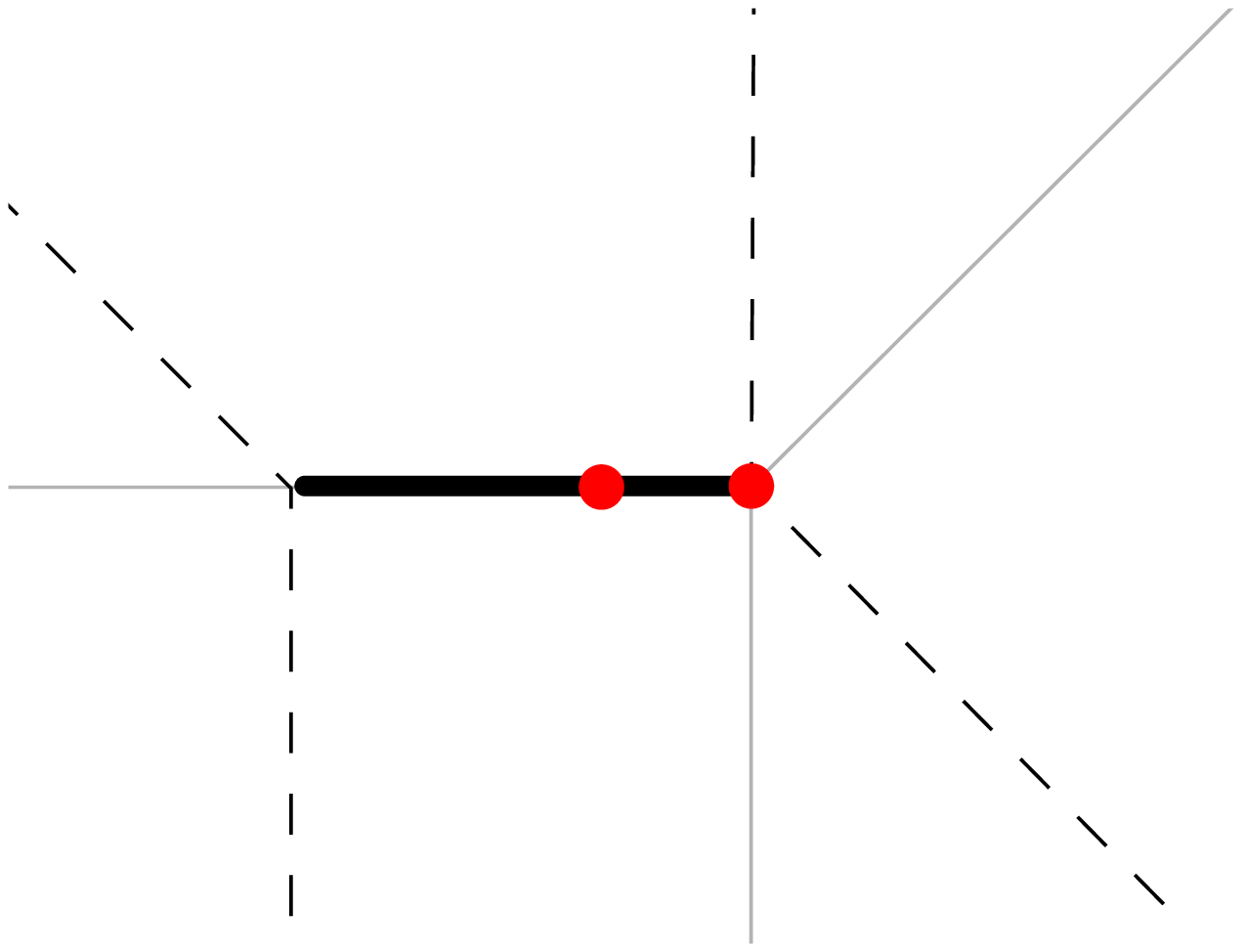}&
\includegraphics[height=2.5cm, angle=0]{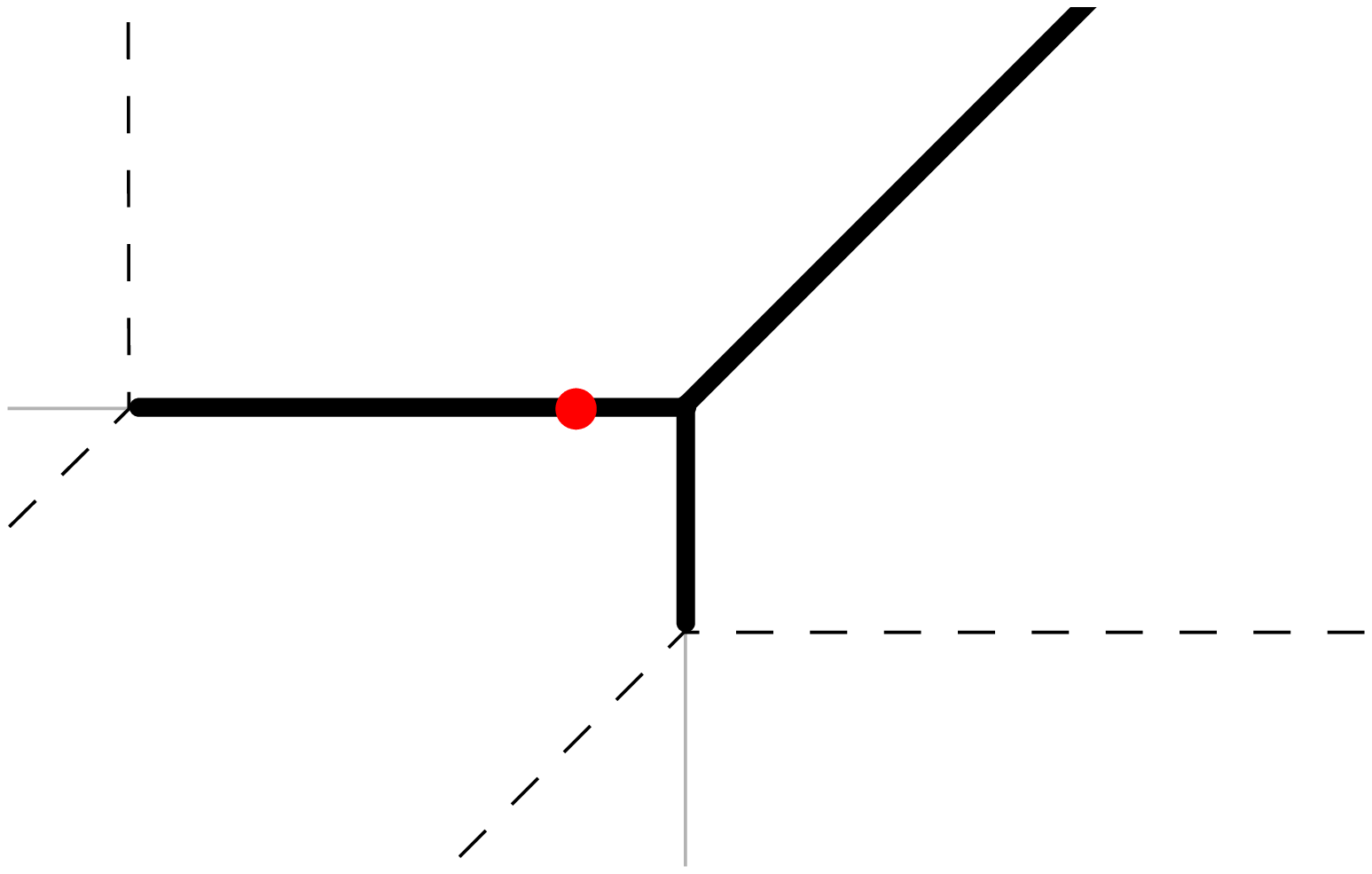}
\\ a) & b) & c) & d)
\\ \includegraphics[height=2.5cm, angle=0]{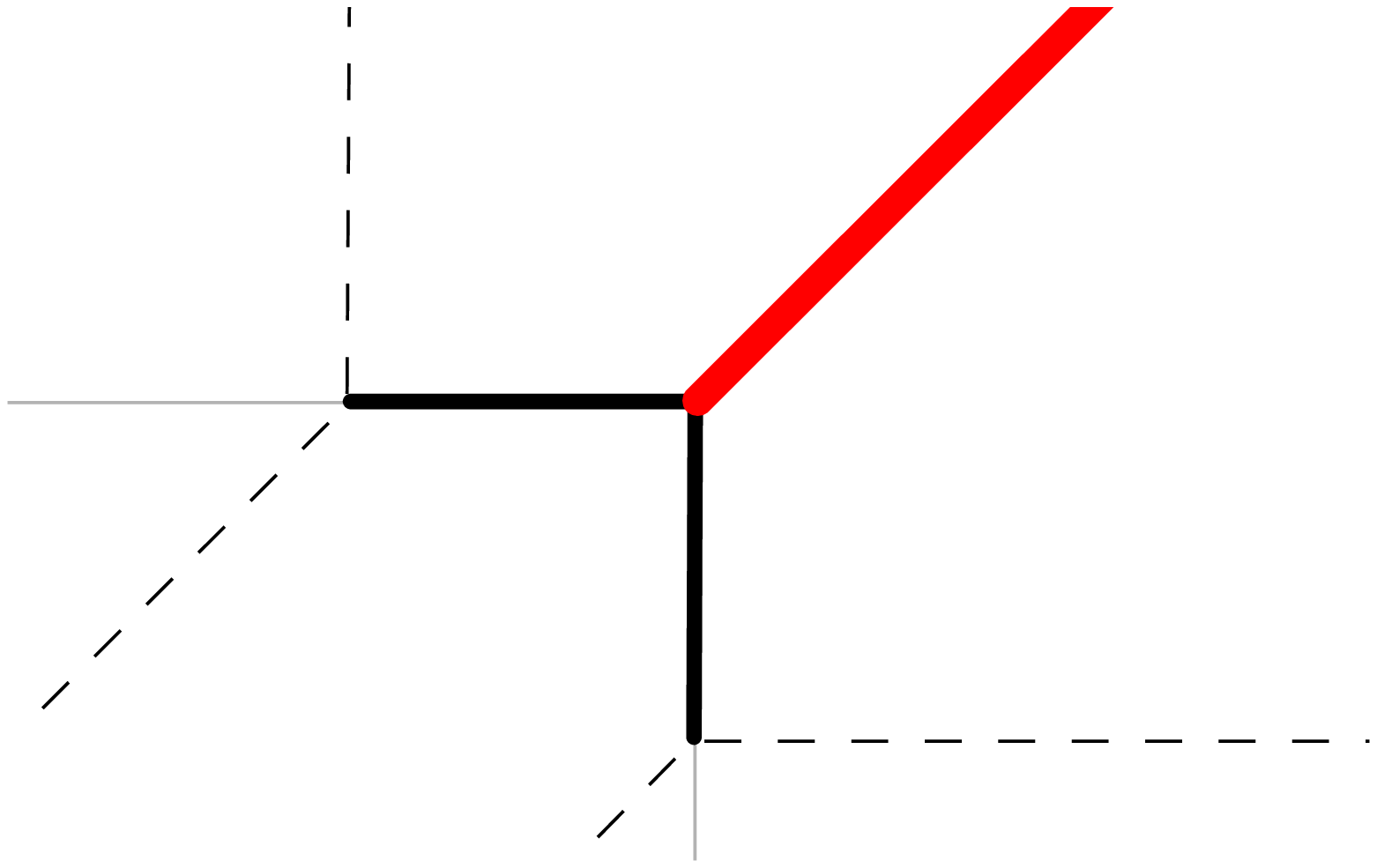}&
\includegraphics[height=2.5cm, angle=0]{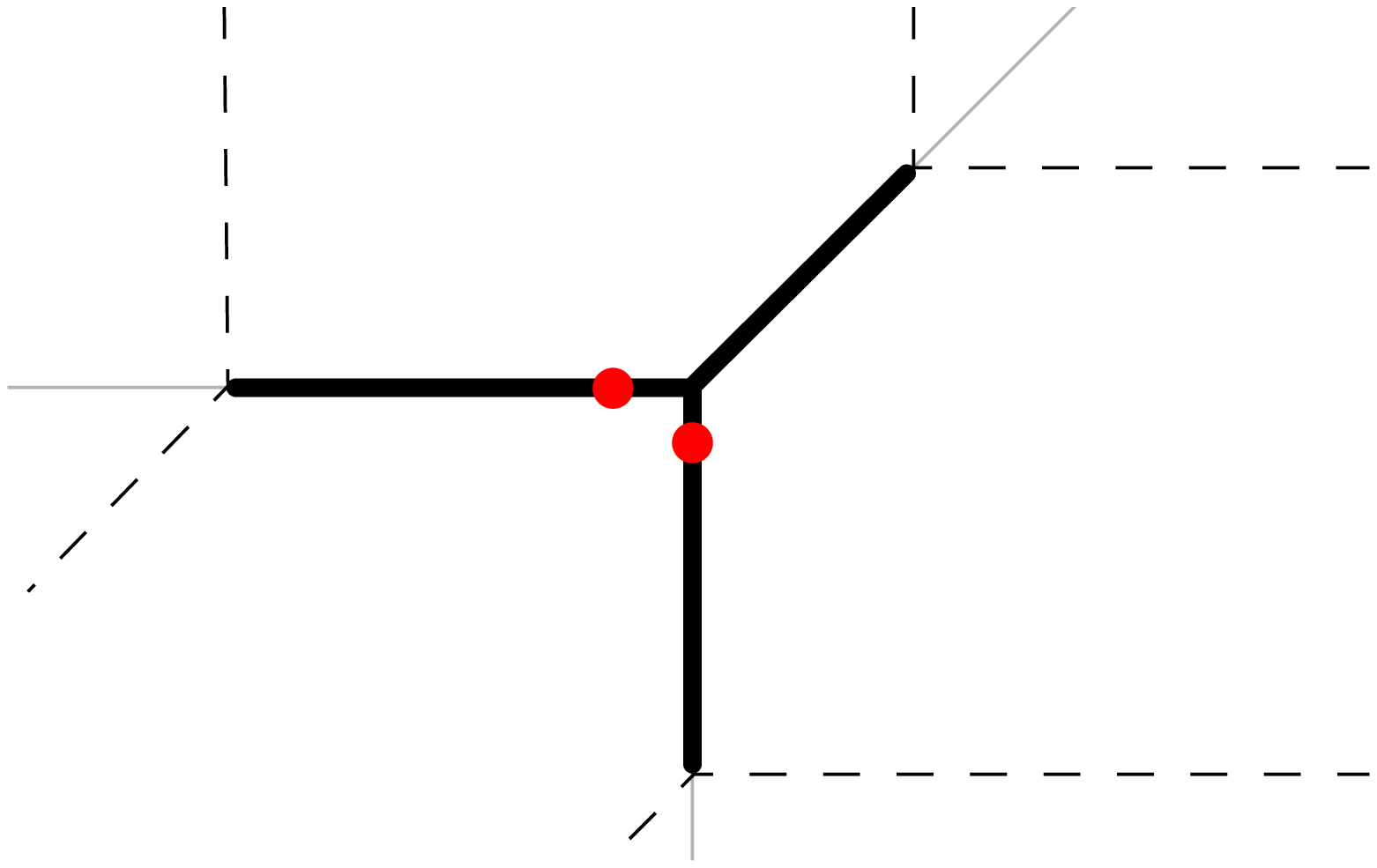}&
\includegraphics[height=2.5cm, angle=0]{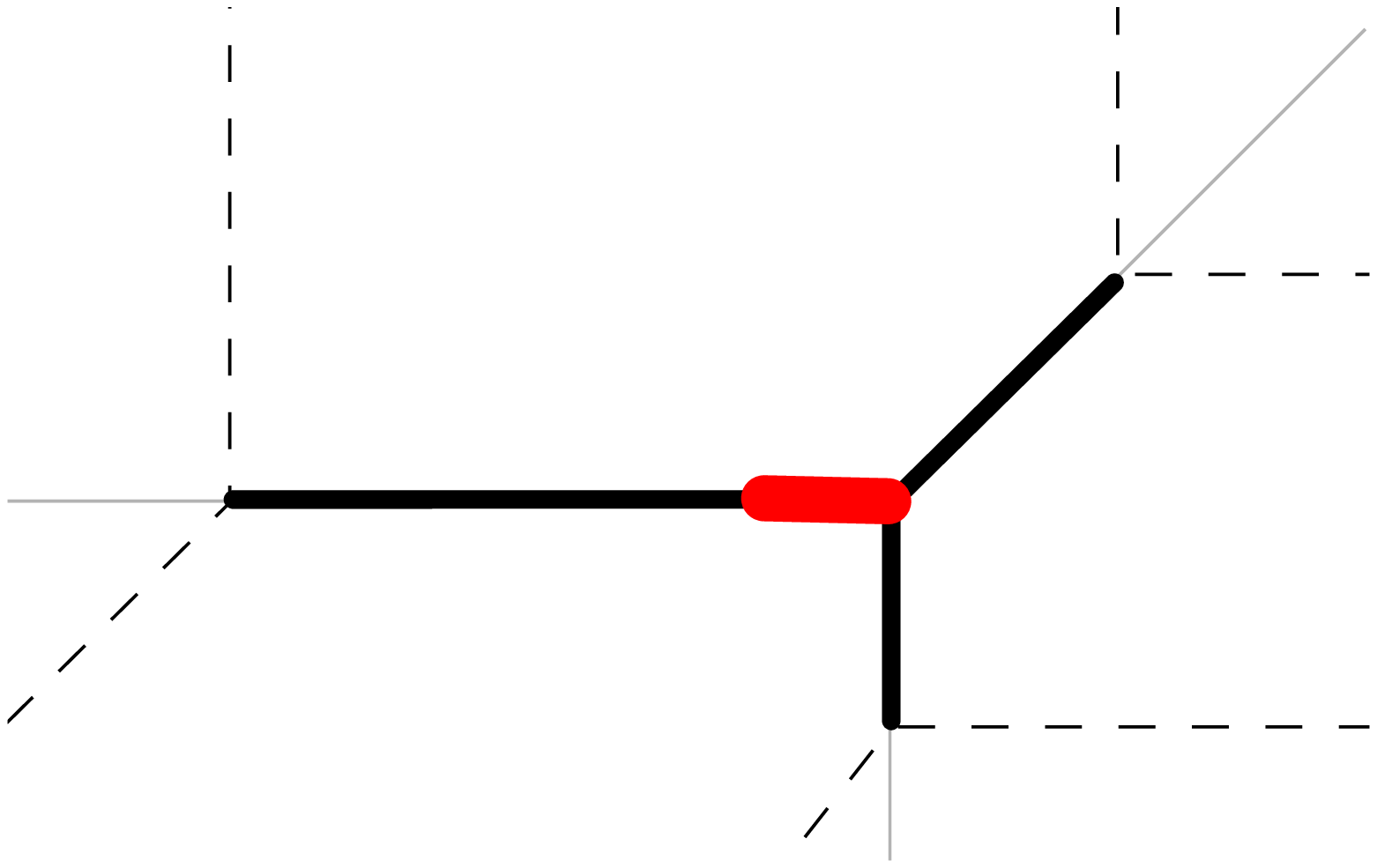}&
\includegraphics[height=2.5cm, angle=0]{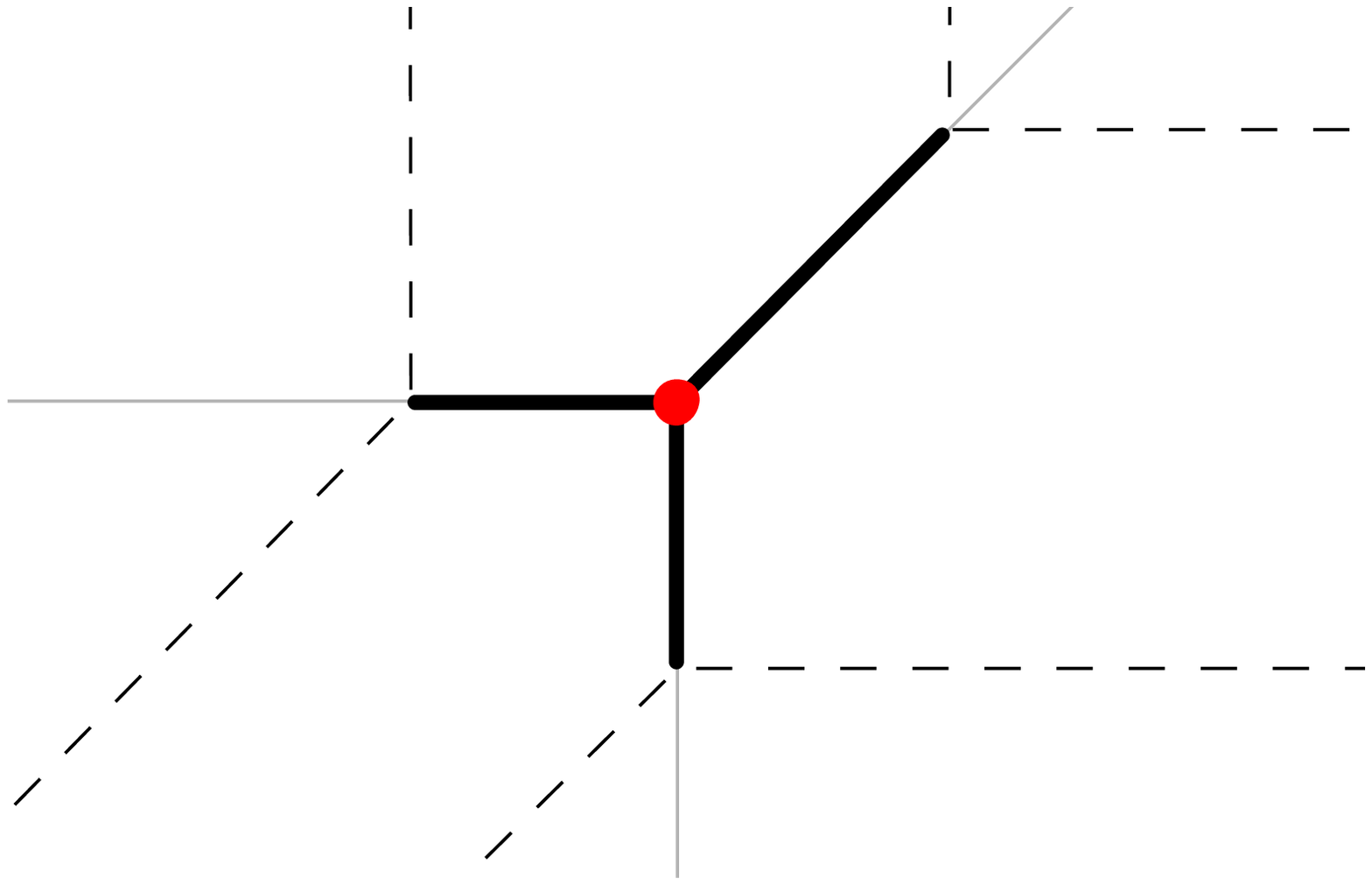}
\\ e) & f) & g) & h)
\end{tabular}

\caption {Description of $\I_L$ in all the possibles cases.}
\label{11}
\end{center}
\end{figure}

\end{itemize}
Note that except in two cases, the
  set $\I_L$ is  finite.

\begin{prop}\label{inclusion}
The  point $p$ is in $\I_L$.
\end{prop}
\begin{proof}
Without loss of generality, we may assume that $v=(0,0)$ and
that $T$ is given by the equation $1+z+w=0$. Let $W$ (resp. 
$C'$) be the tropical modification of $\RR^2$ (resp. $C$) given by
the polynomial $1+z+w$. Recall that the tropical hypersurface $W$ has been
described in Examples \ref{exa trop modif}.
We denote $p=(x_p,y_p)$.
Given a
vertex  $v_1$ of $C$, we denote by $v_1'$ the vertex of $C'$ such that
$\pi_{C'}(v_1')=v_1$.  
According to
Lemma \ref{intersection}, the tropical curve $C'$ must have a vertical
end $e_p$
 with $w(e_p)\ge 3$  such that
$\pi_{C'}(e_p)=p$. 
Let us prove the Proposition case by case.

\vspace{1ex}
\textbf{Case 1: $E=\{v\}$.}  This case is trivial. 

\vspace{1ex}
\textbf{Case 2: $E$ is an unbounded edge $e$ of $C$. } 
We may assume that  $E$ is a horizontal edge. Then the proposition follows
from  Lemma \ref{propE1}.

\vspace{1ex}
\textbf{Case 3: $E$ is a bounded edge $e$ of $C$. } We may assume that
$E$ is a horizontal edge.
Since $C'\subset W$, we have
$v'=(0,0,0)$ and 
$v_e'=(-l(e),0,0)$ (see Figure \ref{Case3}). 
If $p=v$
there is nothing to prove, so
suppose now that $p\ne v$. We have
$(C\cap_\TT L)_E=(C\cap_\TT L)_v +1$. Hence according to Corollary 
\ref{set intersection} and Lemma
\ref{propE1}, 
the edge $e_p$ has weight exactly 3 and is the only vertical end of $C'$
above $e\setminus \{v\}$. According to Corollary \ref{nonsingular} and
the balancing condition, $C'$ has exactly 3 edges above $e\setminus
\{v\}$: $e_p$, an edge with primitive integer
direction $(1,0,-1)$ adjacent to $v_1'$, and an edge with primitive integer
direction $(1,0,2)$ adjacent to $v'$. Hence, we have $(l(e)+x_p) +
2x_p=0$ which reduces to $x_p=-\frac{l(e)}{3}$.

\begin{figure} [htbp]

\psfrag{V}{$v'$}\psfrag{v1}{$v'_1$}\psfrag{p}{$p$}\psfrag{v2}{$v'_2$}\psfrag{3}{$3$}\psfrag{ep}{$e_p$}
\begin{center}

\includegraphics[height=5cm, angle=0]{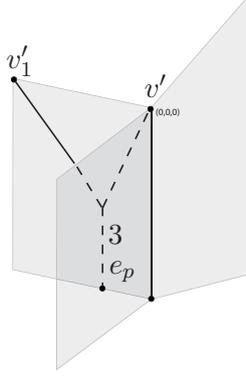}
\caption {Case 3: $E$ is a bounded edge.}
\label{Case3}
\end{center}
\end{figure}
\vspace{1ex}

\textbf{Case 4: $E$  is the union of 2 bounded edges $e_1$, $e_2$, and one
  unbounded edge $e_3$. }
We may assume that $e_1$ is horizontal, $e_2$ is vertical, and that
$l(e_1)\ge l(e_2)$. 
Since $C'\subset W$, we have

$$v_{e_1}'=(-l(e_1),0,0),\quad v_{e_2}'=(0,-l(e_2),0),\quad \text{and}\quad
v'=(0,0,-a) \ \text{with}\ a\ge 0.$$
In this case $(C\cap_\TT L)_E=3$. Then,
the edge $e_p$ has weight exactly 3 and is the only vertical end of $C'$
above $E$. Moreover, the curve $C'$ is completely determined once $p$
is known.

If $p\in e_1$ (i.e. $y_p=0$), 
then the fact that $v'$ is a vertex
of $C'$ gives us the equations $l(e_2)=a$ and $(l(e_1)+x_p) +
2x_p=a$ which reduces to $x_p=-\frac{l(e_1)-l(e_2)}{3}$(see Figure \ref{Case4}a). 

If $p\in e_2$, then $y_p=-\frac{l(e_1)-l(e_2)}{3}$. Since in this case
$y_p$ is
non-positive, this is possible only if $l(e_1)=l(e_2)$ and $y_p=0$ (see Figure \ref{Case4}b). 

If $p\in e_3$, then the vertex $v'$ imposes the condition
$l(e_1)=l(e_2)$. Hence, as soon as $l(e_1)=l(e_2)$, the point $p$ may
be anywhere on $e_3$ (see Figure \ref{Case4}c).

\begin{figure} [htbp]

\psfrag{V}{$v'$}\psfrag{v1}{$v'_{e_1}$}\psfrag{p}{$p$}\psfrag{v2}{$v'_{e_2}$}\psfrag{3}{$3$}\psfrag{ep}{$e_p$}
\begin{center}
\begin{tabular}{ccc}
\includegraphics[height=5.5cm, angle=0]{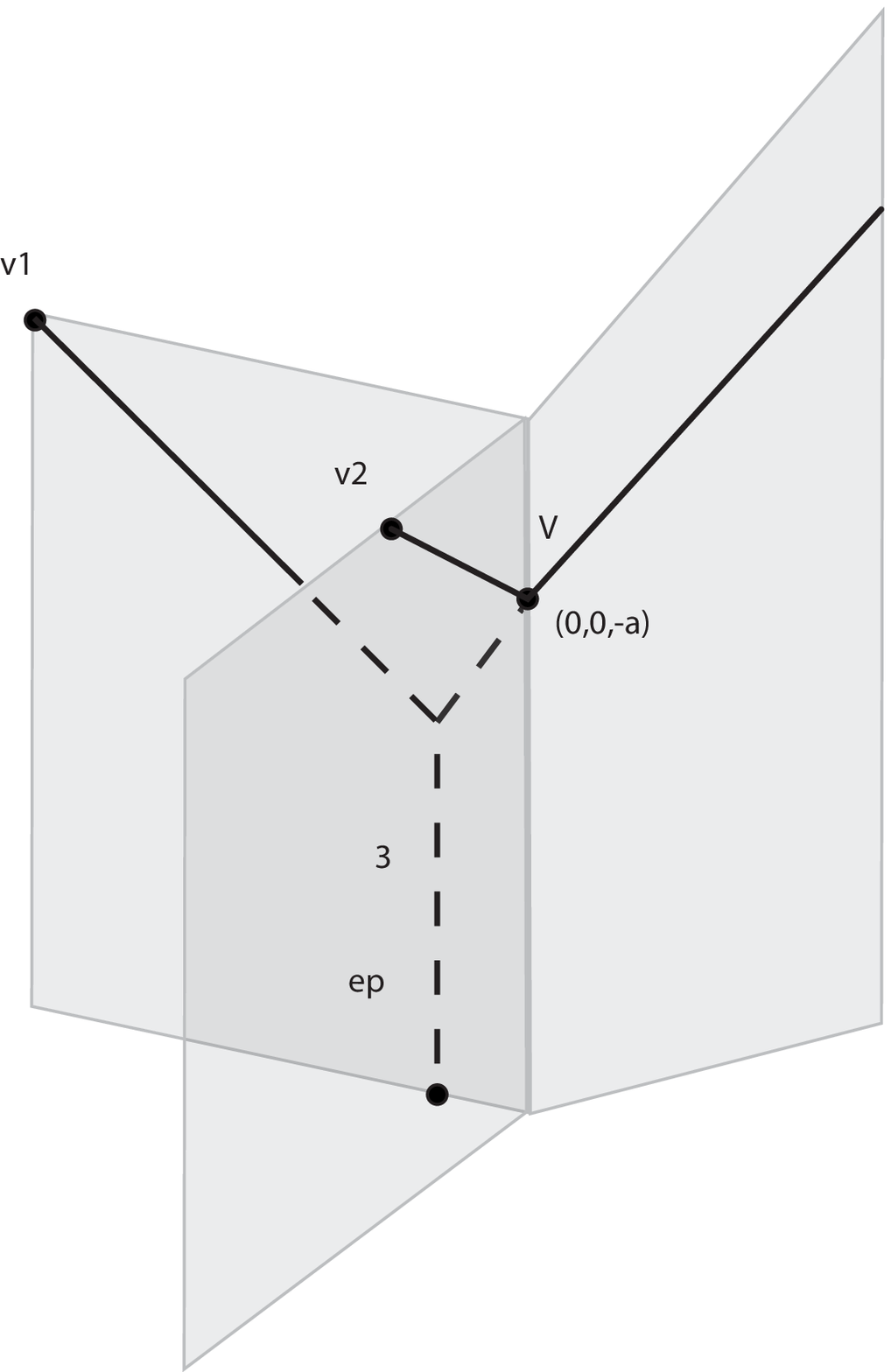}&
\includegraphics[height=5.5cm, angle=0]{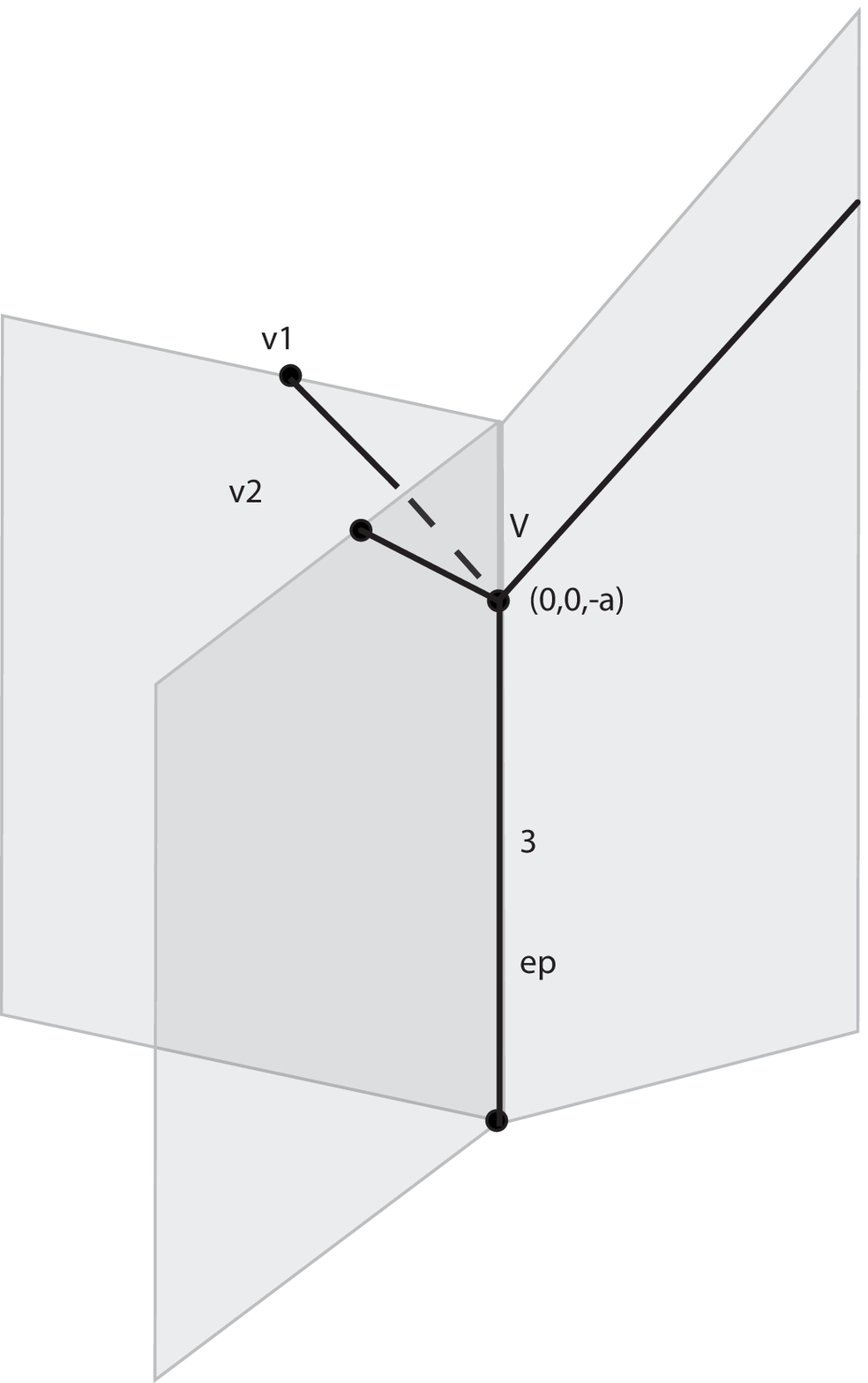}&
\includegraphics[height=5.5cm, angle=0]{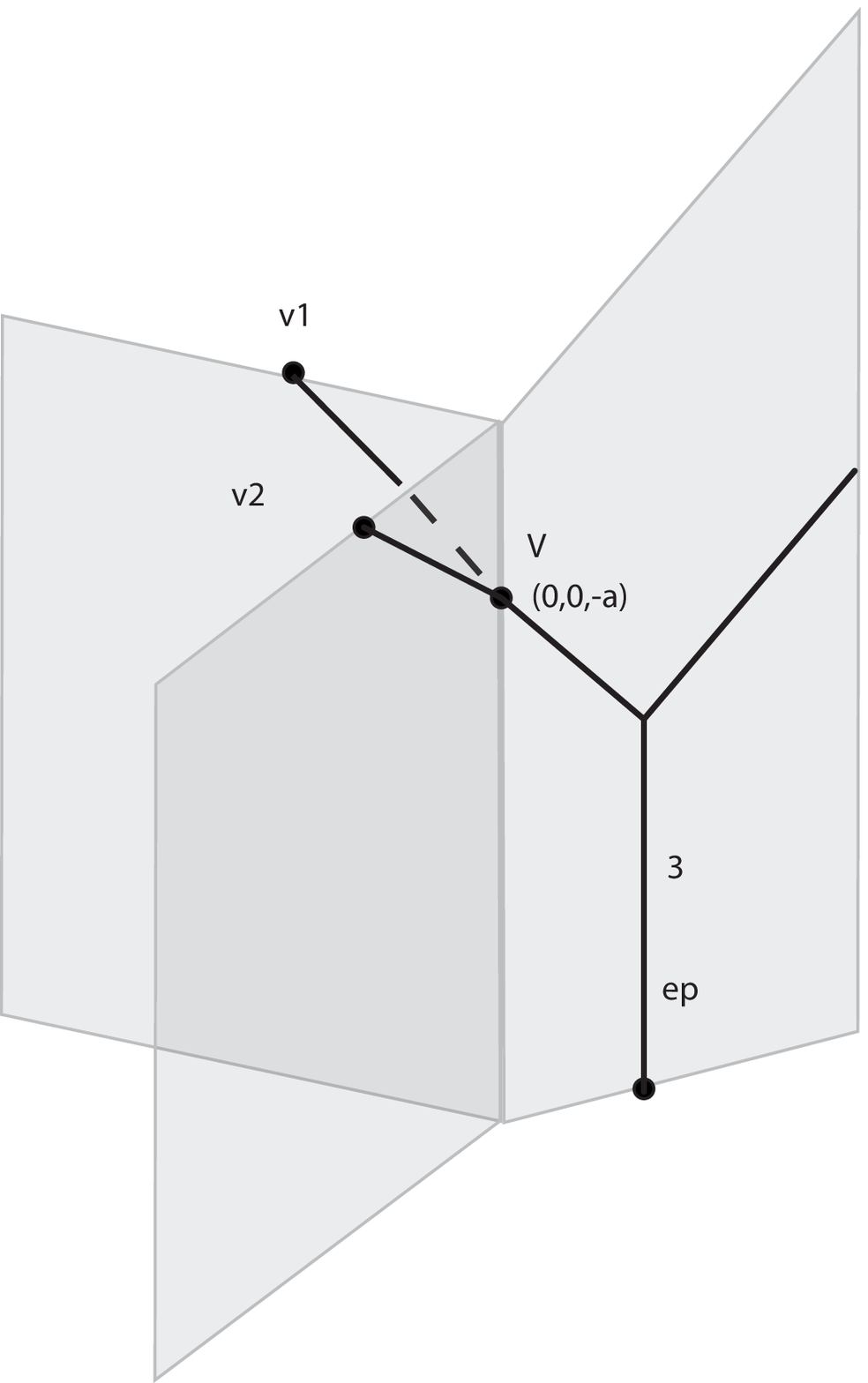}
\\ a) $l(e_1)>l(e_2)$. & b) $l(e_1)=l(e_2)$. & c) $l(e_1)=l(e_2)$.
\end{tabular}
\caption {Case 4: $E$  is the union of 3 edges, 2 of them bounded.}
\label{Case4}
\end{center}
\end{figure}

\vspace{1ex}
\textbf{Case 5: $E$  is the union of 3 bounded edges $e_1$, $e_2$, and  $e_3$. }
We may assume that $e_1$ is horizontal, $e_2$ vertical, and that
$l(e_1)\ge l(e_2)\ge l(e_3)$. 
Since $C'\subset W$, we have (see Figure \ref{Case5})
$$v_{e_1}'=(-l(e_1),0,0),\ v_{e_2}'=(0,-l(e_2),0),
\ v_{e_3}'=(l(e_3),l(e_3),l(e_3)),\
\ \text{and}\
v'=(0,0,-a) \ \text{with}\ a\ge 0.$$
We deduce from
$(C\cap_\TT L)_E=4$ that 
the edge $e_p$ may have weight 3 or 4. If $w(e_p)=4$, then $e_p$  
is the only vertical end of $C'$
above $E$.
If $w(e_p)=3$, there exist
exactly two vertical ends of $C'$
above $E$, $e_p$ and $e'$. Note that $w(e')=1$, and that we may 
assume that $w(e_p)=3$ since the case $w(e_p)=4$ corresponds to the
case
 $e_p=e'$. 
Moreover, the curve $C'$ is completely determined once $p$ and $p'=\pi_{C'}(e')=(x',y')$
are known. 

\begin{figure} [htbp]
\begin{center}

\psfrag{v1}{$v_{e_1}'$}\psfrag{V2}{$v_{e_2}'$}\psfrag{V3}{$v_{e_3}'$}\psfrag{V}{$v'$}\psfrag{V5}{$e'$}\psfrag{V6}{$e_p$}\psfrag{3}{$3$}\psfrag{4}{$4$}\psfrag{=}{$e'=e_p$}
\begin{tabular}{cc}
\includegraphics[height=5.5cm, angle=0]{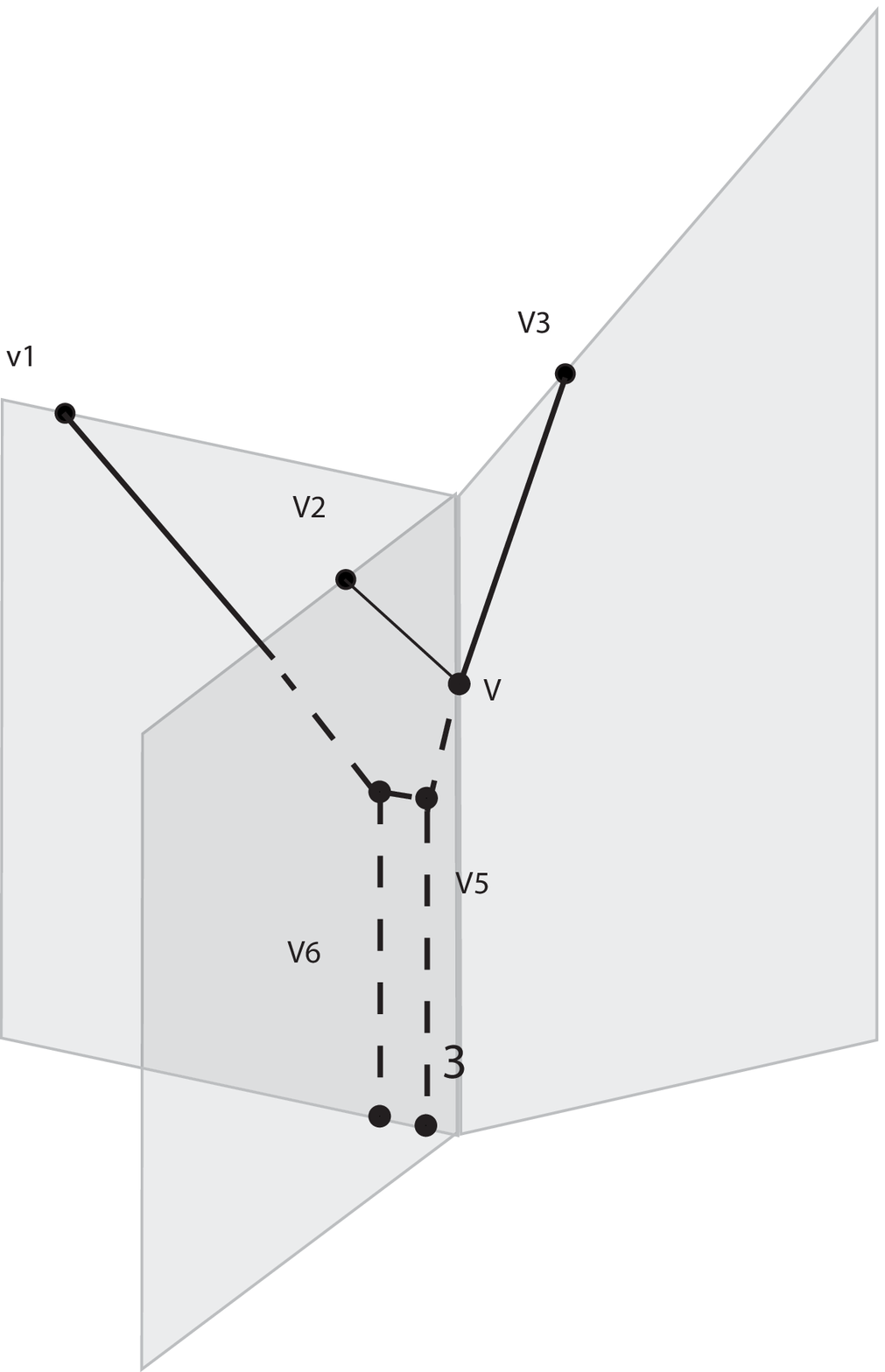}&
\includegraphics[height=5.5cm, angle=0]{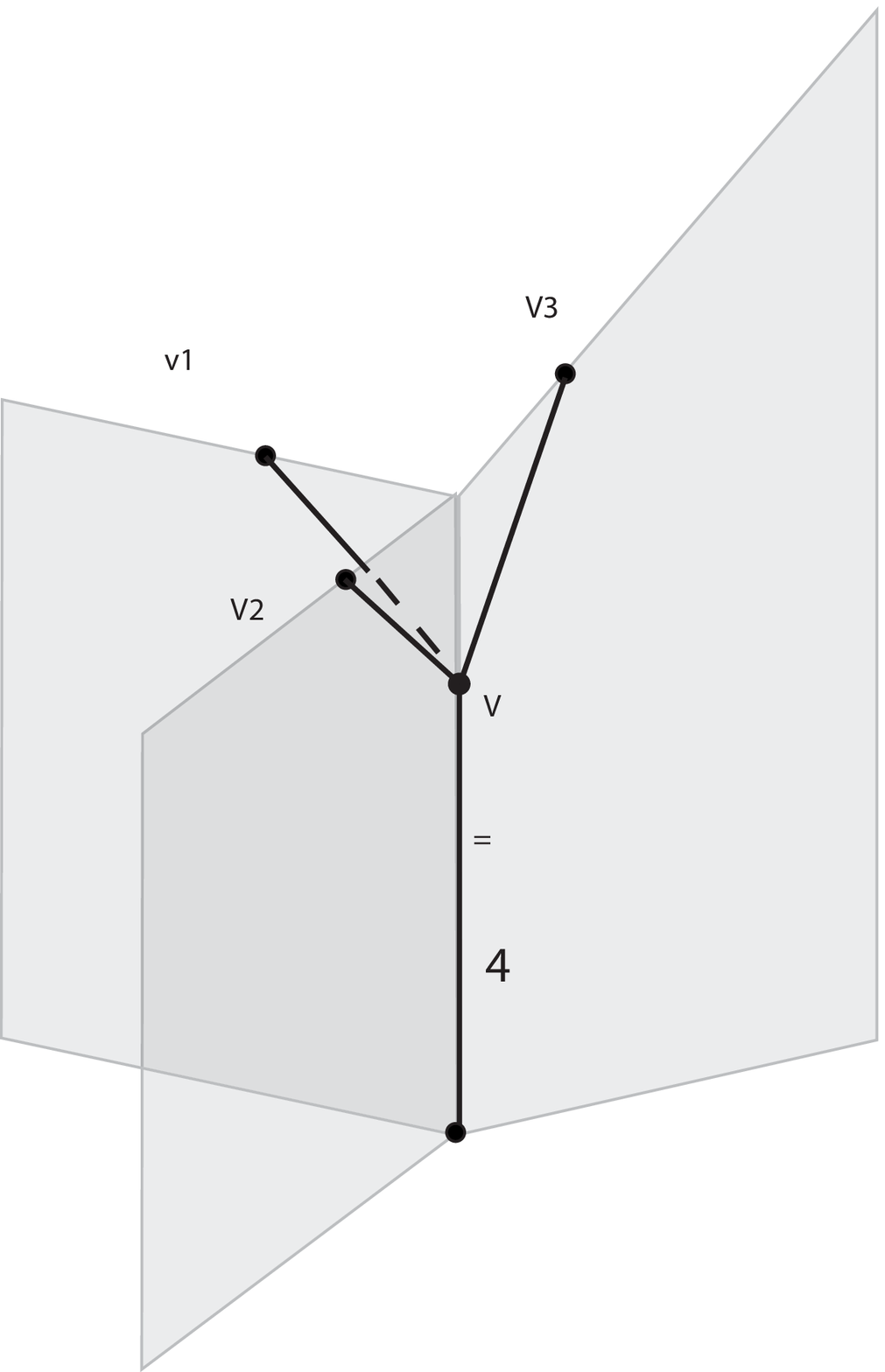}
\\ a) If $p,p'\in e_1$. & b) If $p,p'\in e_3$.
\\ \includegraphics[height=5.5cm, angle=0]{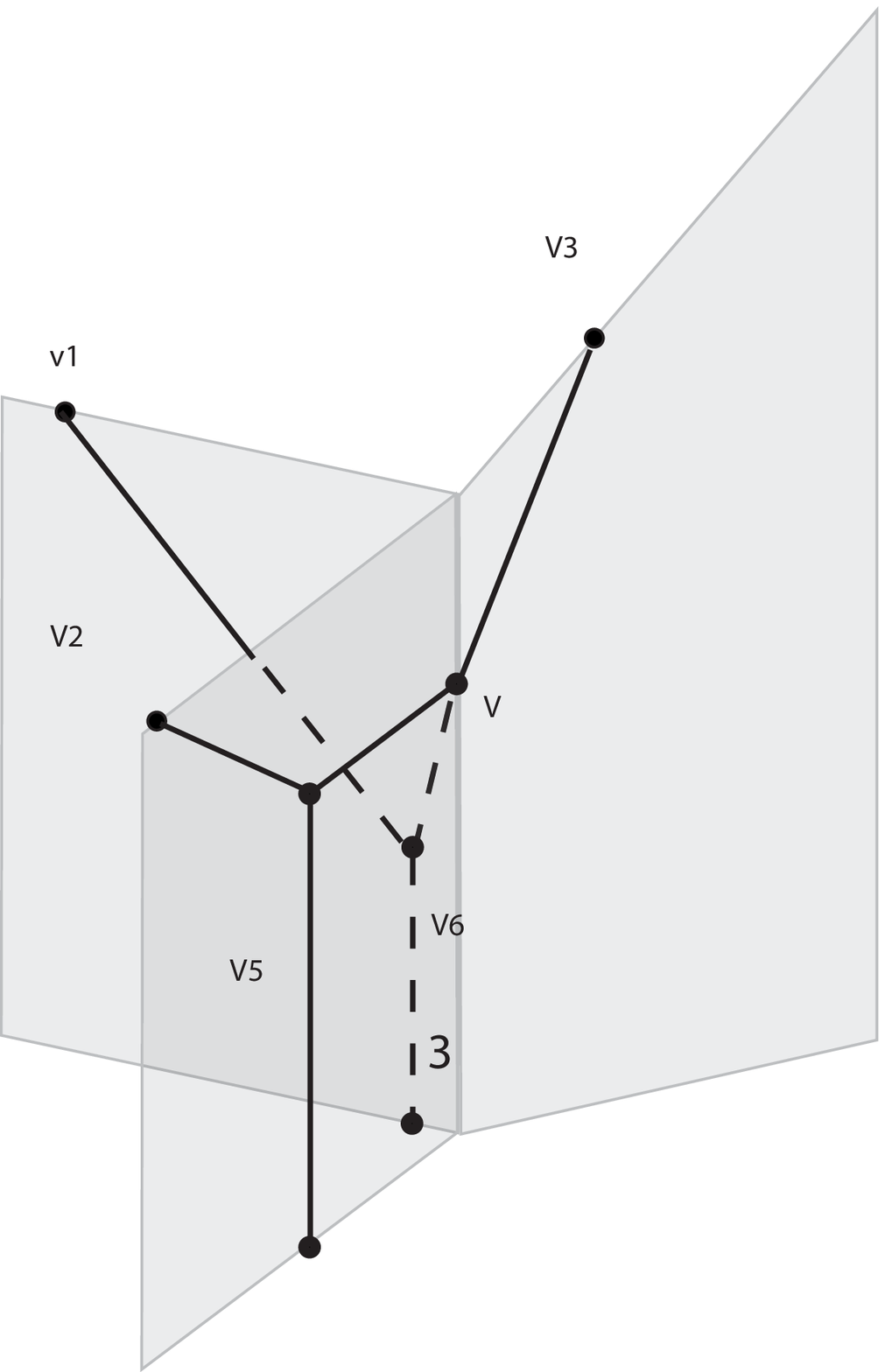}&
\includegraphics[height=5.5cm, angle=0]{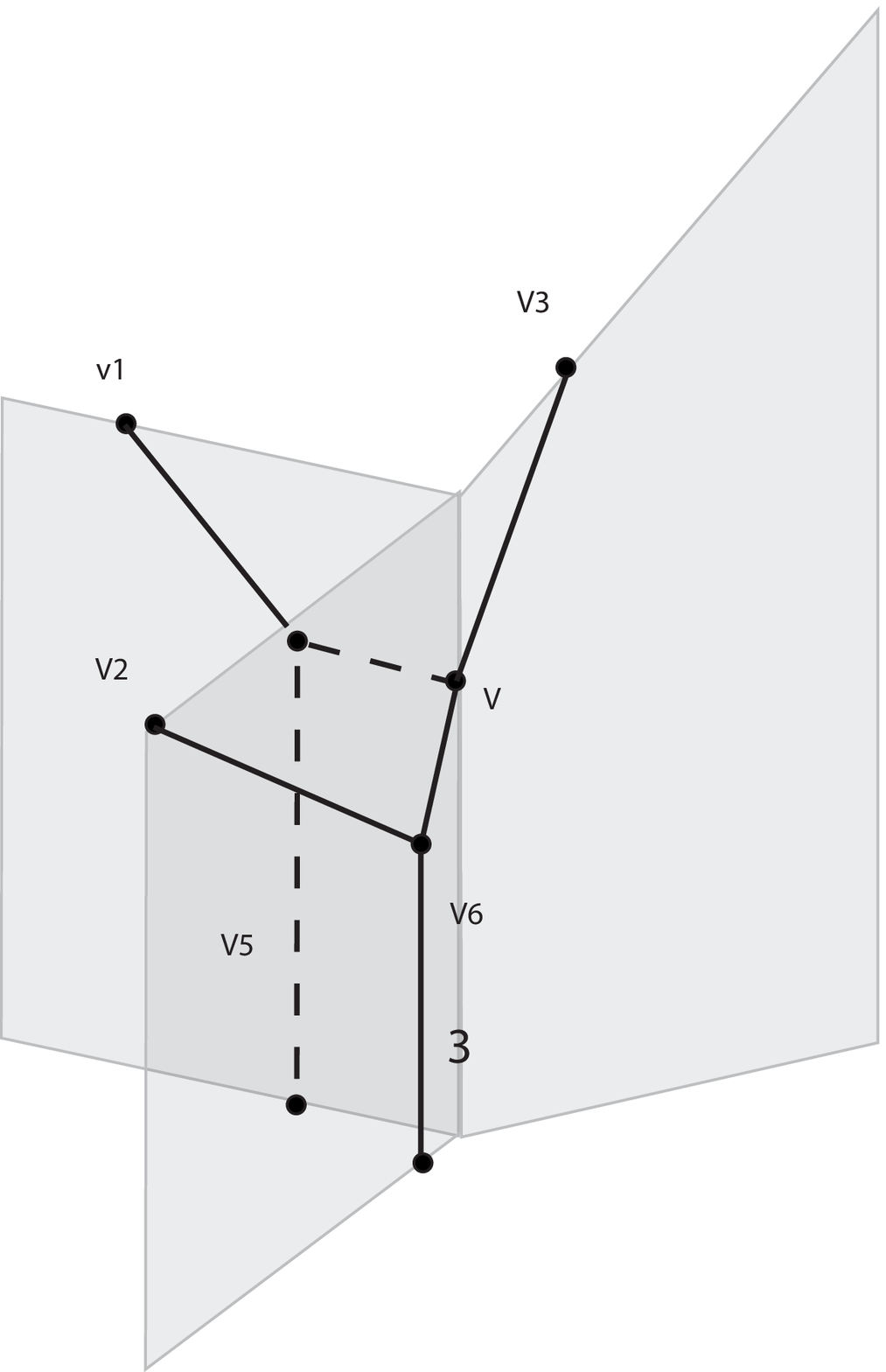}

\\c) If $p\in e_1$ and $p'\in e_2$. & d) If $p\in e_2$ and $p'\in e_1$.
\end{tabular}
\end{center}

\begin{center}

\caption {Case 5: $E$ is the union of 3 bounded edges.}
\label{Case5}
\end{center}

\end{figure}

If $p,p'\in e_1$, then the equations given by the vertex $v'$ of $C'$
reduce to
 $l(e_2)=l(e_3)=a$ and $x_p=-\frac{l(e_1)-l(e_2) - 
 x'}{3}$. 
So this is possible only if $l(e_2)=l(e_3)$,
and in this case the point $p$ may be anywhere on $e_1$ as long
as it is at distance at most $\frac{l(e_1)-l(e_2)}{3}$ from $v$ (See Figure \ref{Case5}a).

In the same way, the points $p$ and $p'$ can be both either on $e_2$
or on $e_3$
if and only if $l(e_1)=l(e_3)$, and in this case 
$p=p'=v$. (See Figure \ref{Case5}b).

If $p\in e_1$ and $p'\in e_2$, then we get
  $x_p=-\frac{l(e_1)-l(e_3)}{3}$ (See Figure \ref{Case5}c).

If $p\in e_2$ and $p'\in e_1$, then we get
  $y_p=-\frac{l(e_2)-l(e_3)}{3}$ (See Figure \ref{Case5}d).

If $p$ is in $e_3$, then we get $a=l(e_1)$ or $a=l(e_2)$, and
$a\le l(e_3) + 2x_p $ which is possible only if $x_p=0$.

In the same way, if $p'\in e_3$, then $p'=v$.
\end{proof}

\subsection{Multiplicities}\label{mult}
In the preceding section, we have seen that if $C$ is a non-singular
tropical curve in $\RR^2$, then the inflection points of any 
realization $X$ of $C$ tropicalize in a simple subset
$I_C$ 
of $C$, which
depends only on $C$. Namely, given a tropical line $L$ whose vertex
$v$ is
also a vertex of $C$ and such that $v$ is contained in a component of
$C\cap L$ of multiplicity at least 3, we  define the set $\I_L$ as
in  section \ref{locate}. 
Then we define
$I_C=\bigcup \I_L$ where $L$ ranges over all such tropical lines. We also
 define $\I_C$ as the set of all connected components of $I_C$.
In this section we prove that given
any element $\E$ of $\I_C$, the number of inflection
points of $X$ which tropicalize in $\E$ only depends on $\E$.

Let $\Delta\subset \RR^2$ be an integer  convex polygon, 
and let $\delta$ be  an edge of
$\Delta$. If $\delta$ is not parallel to any edge of $T_1$, 
then we set $r_\delta=0$; if $\delta$ is  supported on the
line with equation $i=a$ (resp. $j=a$, $i+j=a$) and $\Delta$ is
contained in the half-plane defined by $i\le a $ (resp. $j\le a $,
$i+j\ge a$), 
then we set $r_\delta=Card(\delta\cap\ZZ^2)-1$; otherwise we set  
$r_\delta=2(Card(\delta\cap\ZZ^2)-1)$; finally, we define
$$i_{\Delta}=3Area(\Delta) -\sum_{\delta \text{ edge of }\Delta}r_\delta.$$
Note that  $i_{\Delta}<0$ if and only if $\Delta$ is equal to $T_1$ or
one of its edges.

\begin{defi}
An element of $\I_C$ is called an inflection component of $C$.
The multiplicity
of an inflection component $\E$, denoted by $\mu_\E$, is defined as follows
\begin{itemize}
\item if $\E$ is a vertex of $C$ dual to the primitive triangle
  $\Delta\ne T_1$, then
$$\mu_\E=i_\Delta; $$

\item if $\E$ is bounded and
contains a vertex of $C$ dual to the primitive triangle
  $T_1$, then
$$\mu_\E=6; $$

\item in all other cases,
$$\mu_\E=3. $$
\end{itemize}

\end{defi}

\begin{exa}\label{trop curves with inflection}
We depicted in Figure \ref{honey compl} some 
honeycomb tropical curves together with their inflection
components. Each one of these components is a point  of multiplicity 3.
\end{exa}
\begin{figure}[h]
\centering
\begin{tabular}{ccc}
\includegraphics[height=5cm, angle=0]{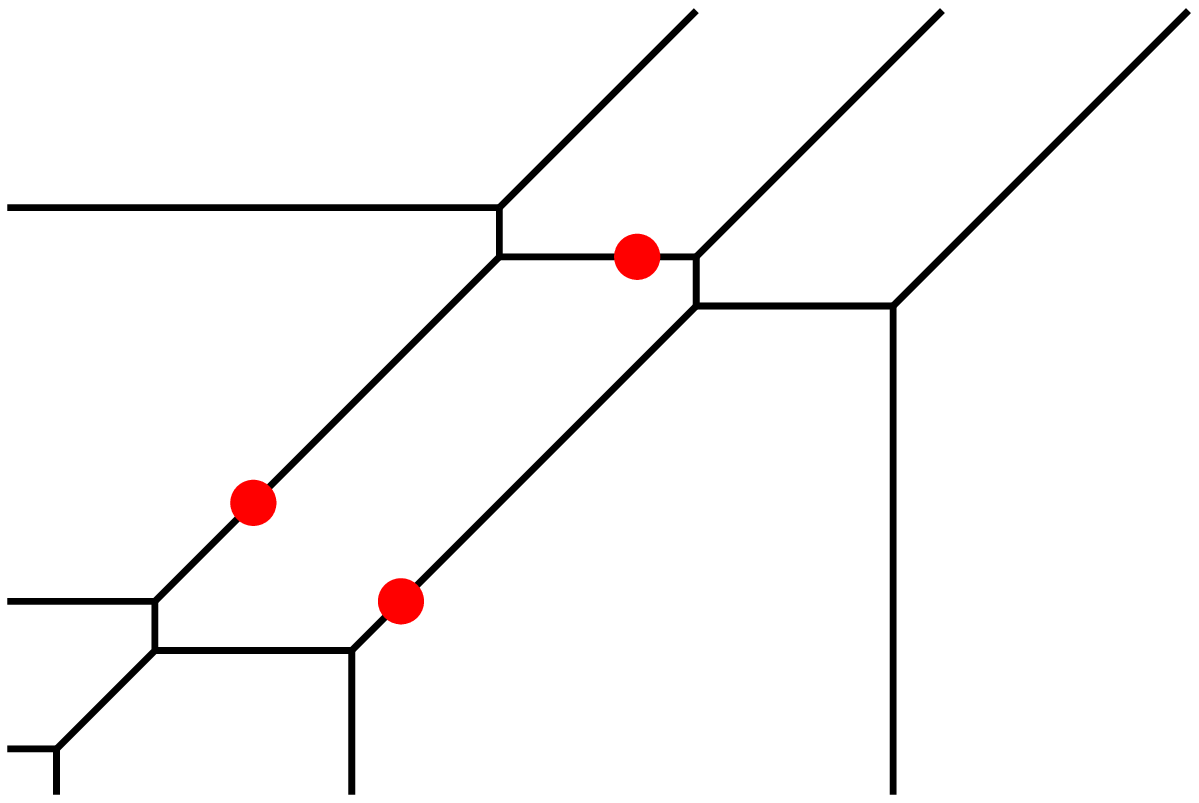}&
\includegraphics[height=5cm, angle=0]{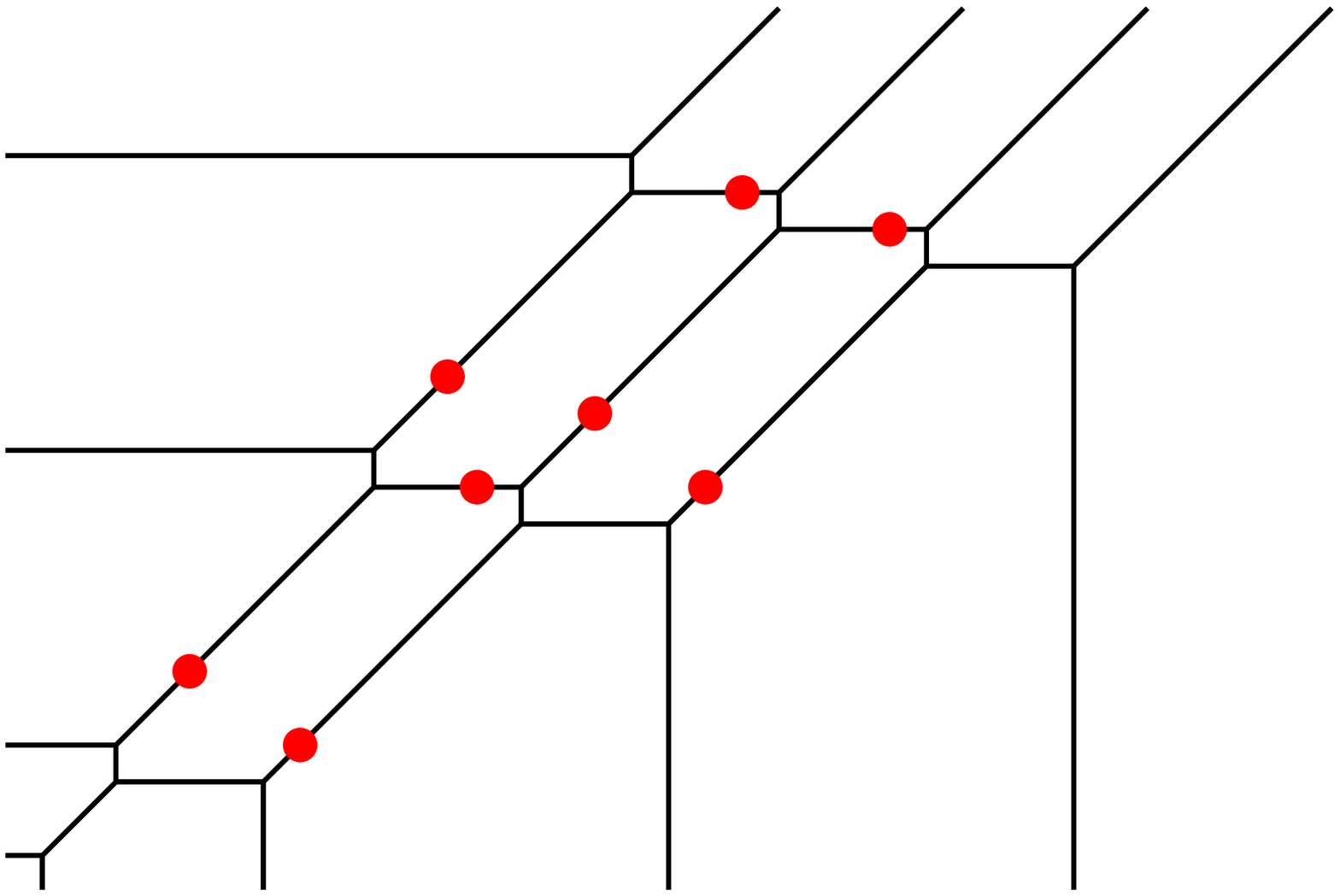}&
\includegraphics[height=5cm, angle=0]{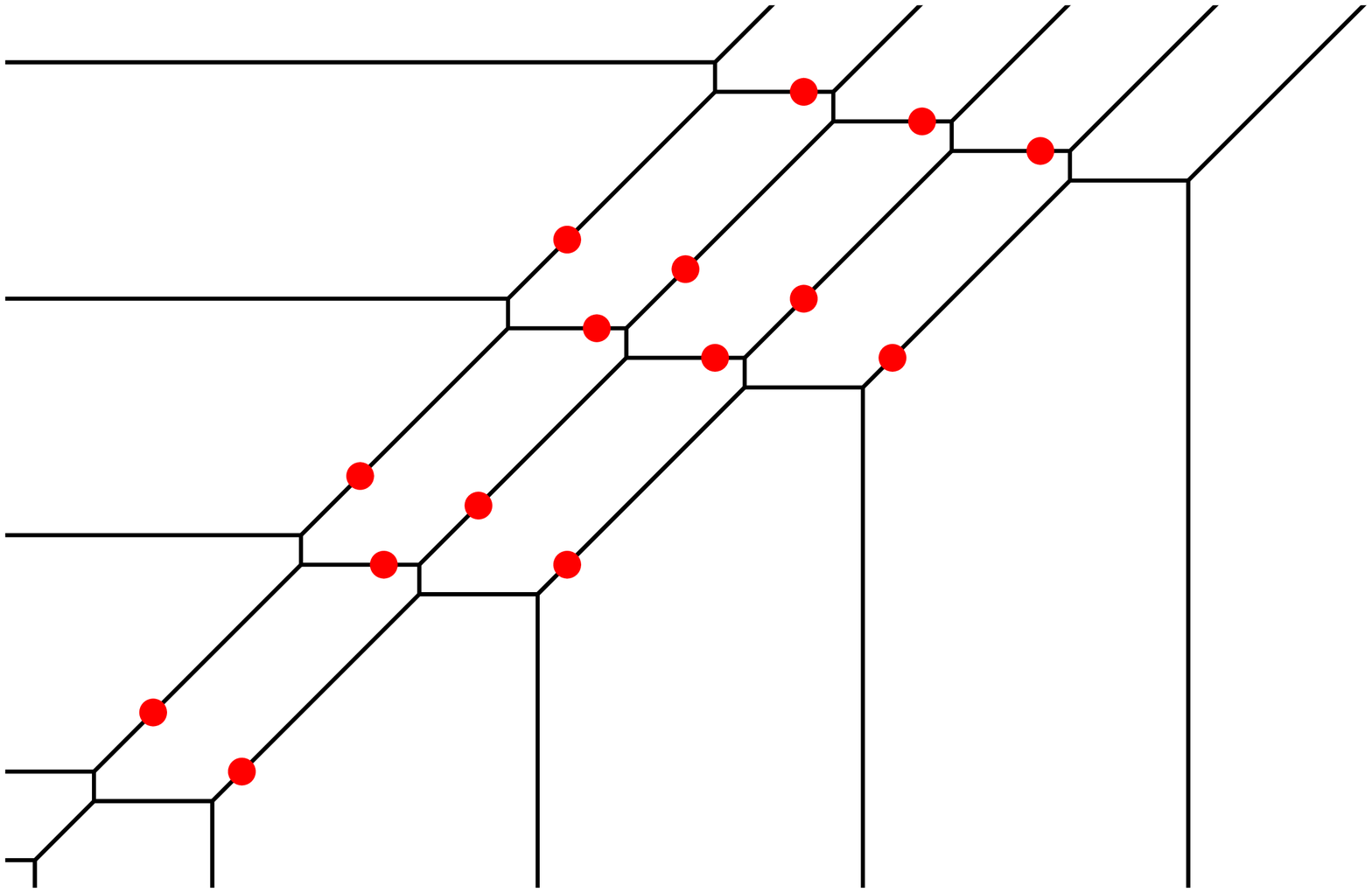} 
\end{tabular}

\caption{Some honeycomb tropical curves and their inflection points}
\label{honey compl}
\end{figure}

\begin{prop}\label{number infl combina}
For any non-singular tropical curve $C$ in $\RR^2$ with Newton polygon
$T_d$, we have
$$\sum_{\E\in\I_C} \mu_\E= 3d(d-2).$$
\end{prop}
\begin{proof}
Let us first introduce some terminology. Let  $\Delta$ be a polygon of 
 the dual subdivision of $C$, and let $\delta$ be one of its edges.
The edge $\delta$  is said to be bounded if the
 edge of $C$ dual to $\delta$ is bounded. The edge $\delta$ is said to have
 $\Delta$-degree 1 if $\delta$ is 
supported either on the line $\{i=a\}$, or $\{j=a\}$, or
    $\{i+j=a\}$, and $\Delta$ is contained in the half plane
  defined     respectively by $\{i\ge a\}$, or $\{j\ge a\}$, or
    $\{i+j\le a\}$. The number of bounded $\Delta$-degree 1 edges of $\Delta$
        is denoted by $\gamma_{\Delta}$. Finally, 
        let $\alpha$ be
        the number of bounded edges of the dual subdivision of $C$
which are parallel to an edge of $T_1$.

From section \ref{locate} and the definition of $\mu_\E$, it follows
immediately 
that for any vertex $v$ of $C$, we have
$$\sum_{\E\in \I_L} \mu_\E= i_{\Delta_v} + 3\gamma_{\Delta_v}$$ 
where $L$ is the line with vertex $v$.
Hence we deduce that
$$\begin{array}{lll}
\sum_{\E\in\I_C}  \mu_\E &= & 3Area(T_d) - 2 Card(\partial T_d\cap \ZZ²) - 3\alpha
+3\alpha  
\\ \\&=& 3d(d-2)
\end{array} $$
as announced.
\end{proof}

To get a genuine correspondence between 
inflection points of an algebraic curve and the
inflection components of its tropicalization, we actually need to pass to
projective curves.
 It is well known that the compactification process we are going to
describe now can be adapted to construct general non-singular tropical toric
varieties. However, since we will just need to deal with plane
projective curves, we restrict ourselves to the construction of 
tropical projective spaces (see \cite{St6}).

As in classical geometry, the \textbf{tropical projective space} 
$\TT P^n$ of
dimension $n\ge 1$ is defined as the quotient of the space
$\TT^{n+1}\setminus\{(-\infty,\ldots,-\infty)\}$ by the equivalence
relation 
$$v\sim \tg\lambda v \td =v+\lambda(1,\ldots, 1).$$
 that is
$$\TT P^n=(\TT^{n+1}\setminus\{(-\infty,\ldots,-\infty)\})/(1,\ldots, 1) $$
Topologically, the space $\TT P^n$ is a simplex of dimension $n$, in particular
it is a triangle when $n=2$. 

The coordinate system $(x_1,\ldots,x_{n+1})$ on $\TT^{n+1}$ induces a
\textbf{tropical homogeneous coordinate system} $[x_1:\ldots,x_{n+1}]$ on
$\TT P^n$, and we have the natural
embedding
$$\begin{array}{ccc}
 \RR^n &\longrightarrow & \TT P^n
\\ (x_1,\ldots,x_n)&\longmapsto & [x_1:\ldots,x_{n}:0].
\end{array}$$
Hence any tropical variety $V$ in $\RR^n$ has a natural compactification
$\overline V$ in $\TT P^n$. Also, any non-compact inflection component of a
non-singular tropical curve $C$ in $\RR²$ compactifies in an inflection
component of $\overline C$.
The map $Val:\KK^{n+1}\to \TT^{n+1}$ induces a map 
$Val:\KK P^{n}\to \TT P^n$, and if $X$ is an algebraic variety in
$(\KK^*)^n$ with closure $\overline X$ in $\KK P^n$, we have
$$Trop(\overline X)= \overline{Trop(X)}. $$

\begin{thm}\label{main} 
Let $C$ be a non-singular tropical curve in $\RR^2$ with Newton
polygon the triangle $T_d$ with $d\ge 2$, and let $X$ be any
realization of $C$. Then for any inflection point $p$ of $\overline X$, the
point $Val(p)$ is contained in an inflection component of $\overline C$, and for
any inflection component $\E$ of $\overline C$, exactly $\mu_\E$ inflection
points of $\overline X$ have valuation in $\E$.
\end{thm}
We postpone the proof of Theorem \ref{main} to section \ref{proof main}.
The fact that any inflection point of $\overline X$ tropicalizes in
some inflection component of $C$ has already been proved in
Proposition \ref{inclusion}.
The fact that exactly $\mu_\E$ inflection
points of $\overline X$ have valuation in $\E$ for any inflection
component $\E$ follows from Lemmas \ref{inv primitive}, 
\ref{vertex modif 1}, \ref{edge 1}, and \ref{edge 2}.

\subsection{Application to real algebraic geometry}
 Here we give a real version of Theorem \ref{main}, which 
 implies
immediately Theorem \ref{main intro}.
Given $\E$ an inflection component of a non-singular tropical curve
$C$, we define its real multiplicity $\mu_\E^\RR$ by 
$$\mu_\E^\RR=0 \ \text{if}\  \mu_\E \ \text{is even}, 
\quad \quad\mu_\E^\RR=1 \ \text{if}\ \mu_\E \ \text{is odd}.$$

\begin{thm}\label{main real}
Let $C$ be a non-singular tropical curve in $\RR^2$ with Newton
polygon the triangle $T_d$ with $d\ge 2$. Suppose that if $v$ is a
vertex of $C$ adjacent to 3 bounded edges and such that $\Delta_v=T_1$, 
then these edges have 3 different length.
Then given any realization $X$ of $C$ over $\RR\KK$ and given
any inflection component $\E$ of $\overline C$, exactly $\mu^\RR_\E$ inflection
points of $\overline X$ have valuation in $\E$. 

In particular, the curve $\overline X$ has exactly
$d(d-2)$ inflection points in $\RR\KK P²$, and 
the curve $\overline X(t)$ has also exactly
$d(d-2)$ inflection points in $\RR P²$ for $t>0$ small enough. 
\end{thm}

\begin{proof}
Since $\overline X$ is defined over
$\RR\KK$, its inflection points are either in  $\RR\KK P^2$
or they come in pairs
of conjugated points. Hence, for each inflection component $\E$ of
$C$, at least $\mu_\E^\RR$ inflection points of $\overline X$ are
real and have valuation in $\E$.

If $C$ satisfies the hypothesis of the theorem, any of its
inflection component $\E$ has multiplicity at most 3. Let us prove that 
the number of inflection points of $C$ of multiplicity $1$ is equal to
the number of inflection points of $C$ of multiplicity $2$:
this is obviously true when $C$ is a honeycomb tropical curve
(i.e. all its edges have direction $(1,0)$, $(0,1)$, or $(1,1)$); one
checks easily that if
this is true for $C$, then this is also true for any tropical curve
whose dual subdivision is obtained from the one of $C$ by a flip; in conclusion
this is true for any non-singular tropical curve since any two
primitive regular integer triangulations of $T_d$
 can be obtained one from the other by a
finite sequence of flips.

As a consequence, we deduce
$$\begin{array}{lll}
\sum_{\E\in\I_C}  \mu^\RR_\E &= & \frac{1}{3}\sum_{\E\in\I_C} \mu_\E
\\ \\&=& d(d-2).
\end{array} $$

As a consequence the  algebraic curve $\overline X$ has at least $d(d-2)$
inflection points in  $\RR\KK P^2$. Since $\overline X$  is defined over $\RR\KK$,
it cannot have more according to
Theorem \ref{klein}. Hence the  curve $\overline X$ has exactly $d(d-2)$
inflection points in  $\RR\KK P^2$, and exactly $\mu^\RR_\E$ inflection
points of $\overline X\cap \RR\KK P^2$ 
have valuation in $\E$ for any inflection
component $\E$ of $C$. 
\end{proof}

\begin{figure}[h]
\centering
\begin{tabular}{ccc}
\includegraphics[height=6cm, angle=0]{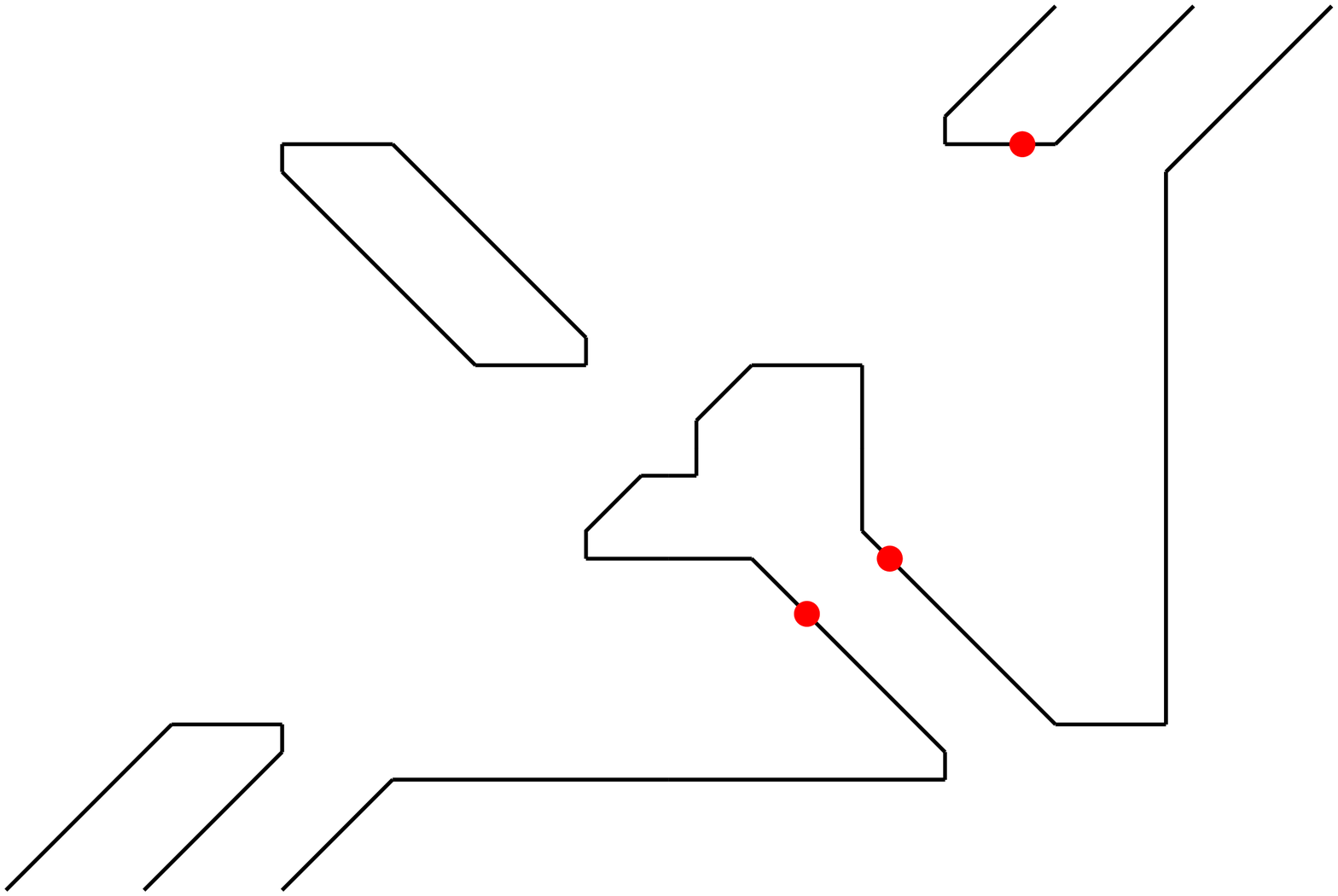}& \hspace{5ex} &
\includegraphics[height=6cm, angle=0]{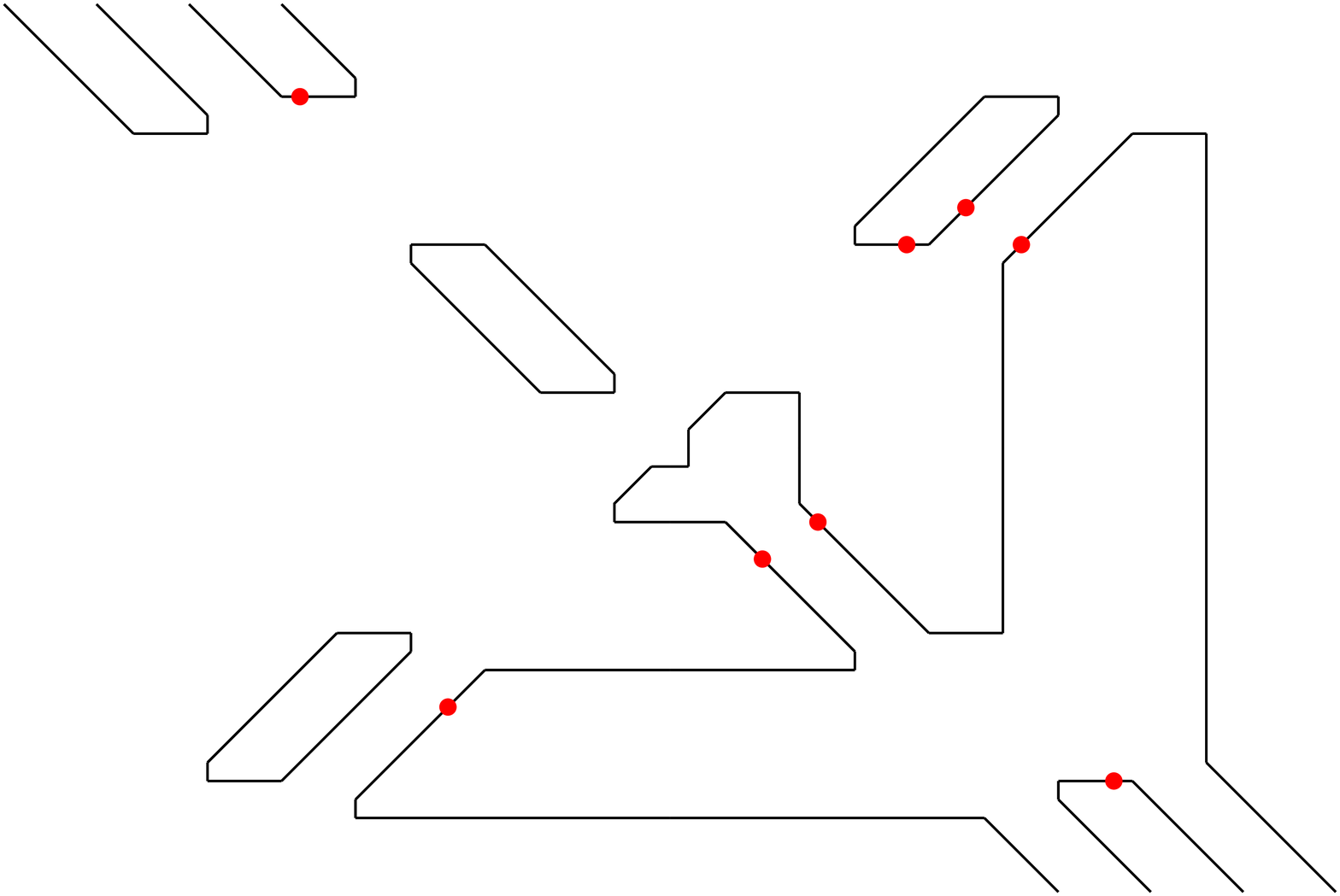}
\end{tabular}
\caption{A Harnack curve of degree 3 and 4}
\label{honey real}
\end{figure}
\begin{figure}[h]
\centering
\begin{tabular}{c}
 \includegraphics[height=6cm, angle=0]{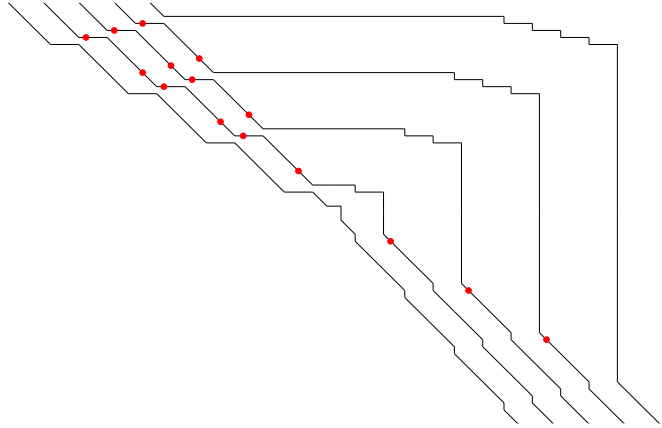} 
\end{tabular}

\caption{A hyperbolic quintic}
\label{honey real 2}
\end{figure}
\begin{exa}\label{exa1}
In Figures \ref{honey real} and  
\ref{honey real 2}, we depicted one possible patchworking of
real curve together with its real inflection points 
 for each  honeycomb tropical curve depicted in Figure
 \ref{honey compl}.
\end{exa}

\section{End of the proof of Theorem \ref{main}}\label{proof main}

Here we prove that the multiplicity of
an inflection component $\E$ of $\overline C$ corresponds to the number of
inflection points of $\overline X$ which tropicalizes in $\E$.
Thanks to tropical modifications, all computations are reduced to
elementary local considerations. 

Our first task is to study inflection points of algebraic curves in the
torus.

\subsection{Hessian of a primitive polynomial}
Given a polynomial $P(z,w)$ in $k[z,w]$, we define
$P^h(z,w,u)=z^2w^2u^2P^{hom}(z,w,u)$, and
 $H_P(z,w)=Hess_{P^h}(z,w,1)$. Clearly, 
the curves $V(P)$ and $V(P^h)$ have the same
 inflection points in $(k^*)^2$.

\begin{prop}\label{max infl points}
Let $P(z,w)$ be a polynomial in 2 variables over $k$.
If the curve $V(P)$  is reduced and 
does not contain any
line, 
then the  number of
inflection points
of $V(P)$ is at most  $i_{\Delta(P)}$ (recall that
$i_{\Delta(P)}$ has been defined in section \ref{mult}).
Moreover, this number is exactly equal to  $i_{\Delta(P)}$
 if $k$ is algebraically closed, and
 $r_\delta=0$ and $P^\delta$ has no multiple roots in $(\CC^*)^2$
for all edges $\delta$ of $\Delta$.
\end{prop}
\begin{proof}
To prove the Proposition, we may suppose that $k$ is algebraically
closed. By assumption on $V(P)$, it has finitely many inflection points.
The Newton polygon of ${H_P}$ is, up to
translation, $3\Delta(P)$. Given $\delta$ an edge of
$\Delta(P)$, 
we denote by 
$\delta_{he}$
the corresponding edge of 
$\Delta(H_{P})$, by $n_\delta$ the number of common
roots of the polynomials $P^\delta$ and $H_{P}^{\delta_{he}}$, and we define
$$i_P=\sum_{\delta\text{ edge of }\Delta(P)}n_\delta.$$
 Hence, according to  Bernstein
Theorem, the number of inflection points of $V(P)$ is at
most
$$\frac{1}{2}\left(Area(4\Delta(P)) - Area(3\Delta(P)) - Area(\Delta(P))\right) 
 -i_P = 3Area(\Delta(P))
 -i_P$$
with equality if $i_P=0$.

\vspace{1ex}
Hence, we are left to the study of $n_\delta$ when $\delta$
 ranges over all edges of
$\Delta(P)$. Since $P^h(z,w,u)$ is divisible by
$z^2w^2u^2$,  we  have 
$H_{P^\delta}=(H_P)^{\delta_{he}}$ for any edge $\delta$
of $\Delta(P)$. Hence we may suppose that $\Delta(P)$ is
reduced to
the edge $\delta$. In this case $P^h(z,w,u)$ splits into the product of
$Card(\delta\cap \ZZ^2) - 1$
binomials, so we may further assume that $Card(\delta\cap \ZZ^2)=2$.
Then the curve $V(P)$ is non-singular in $(k^*)^2$, so the only possibility 
for $n_\delta$ to be equal to 1, is for $V(P)$ to be a line. That is,
$\delta$
 must
be parallel to an edge of $T_1$.

\vspace{1ex}
In the case when
$\Delta(P)$ has an edge $\delta$  supported on the
line with equation $i=a$ (resp. $j=a$, $i+j=a$) and $\Delta$ is
contained in the half-plane defined by $i\ge a $ (resp. $j\ge a $,
$i+j\le a$), we can refine the upper bound. Without loss of
generality,
 we may assume that $\delta$ is  supported on the
line with equation $i=0$, and we define $Q(z,w,u)=w^2u^2P^{hom}(z,w,u)$.
The polygon $\Delta(Hess_{Q})$ is contained, up
to translation, in
the polygon 
$3\Delta(P)\cap\{i\ge 2\}$,
so Bernstein Theorem implies
 that the number of inflection points of $V(P)$ is at
most
$$\frac{1}{2}\left(Area(4\Delta(P)\cap\{i\ge 2\}) - Area(3\Delta(P)\cap\{i\ge 2\}) -
Area(\Delta(P)) \right) 
 - \sum_{\delta'\ne\delta \text{ edge of }\Delta(Q)}n_{\delta'}.$$
Since
$$Area(m\Delta(P)\cap\{i\le 2\})= Area(\Delta(P)\cap\{i\le 2\}) +
4(m-1)(Card(\delta\cap\ZZ^2)-1),$$ 
we see that we can in fact substract $2(Card(\delta\cap\ZZ^2)-1)$ to
$3Area(\Delta(P))$ in the upper bound for the number of inflection
points of $V(P)$.
\end{proof}

\begin{exa}\label{prim1}
Any curve over an algebraically closed field, whose Newton
polygon is a primitive triangle without any edge 
parallel to an edge of $T_1$, has exactly 3 inflection points
in $(k^*)^2$. Indeed, such a curve is non-singular in $(k^*)^2$ and
does not contain any line.
\end{exa}

\begin{exa}
Proposition \ref{max infl points} might not be sharp, even if $k$ is
algebraically closed. Indeed, let $X$ be a cubic curve in $kP^2$.
It is  classical  (see for example \cite{Sha}) 
that a line passing through 
two inflection points
 of $X$ also passes through
a third inflection 
point of $X$. Hence, any algebraic curve with
Newton polygon the triangle $\Delta$ with vertices $(0,0)$, $(1,2)$, and
$(2,1)$, cannot have more than 6 inflection points in  $(k^*)^2$,
although in this case $i_{\Delta}=7$.
\end{exa}

However, 
 Proposition \ref{max infl points} is sharp for curves with primitive
 Newton polygon.

\begin{lemma}\label{inv primitive}
If $\Delta(P)$   is  primitive
and distinct from $T_1$, then the curve $V(P)$ has exactly 
$i_{\Delta(P)}$  inflection points in $(k^*)^2$.
\end{lemma}
\begin{proof}
According to Example \ref{prim1}, 
 it remains 
 to check by hand the lemma
in the following 
two particular cases
(all coefficients may
be chosen equal to 1 since $\Delta(P)$ is primitive):
\begin{itemize}
\item $P(z,w)=w+ z^{l-1}+ z^l$ with $l\ge 2$: inflection points are
  given by the roots in $k^*$ of the second derivative of the 
   polynomial $z^{l-1}+ z^l$. We have 1 (resp. no) such  root when
  $l\ge 3$ (resp. $l=2$).

\item $P(z,w)=w+ zw+ z^l$ with $l\ge 0$: inflection points are
  given by the roots in $k^*$ of the second derivative of the 
  function $f(z)=-\frac{z^l}{1+z}$. We have 2 (resp. no) such  roots when
  $l\ge 3$ (resp. $l\le 2$).
\end{itemize}
\end{proof}

Let us fix  a polynomial $P(z,w)$  in $\CC[z,w]\subset\KK[z,w]$ such that
$\Delta(P)$ is 
primitive
and different from $T_1$, and  define $C=Trop(V(P))$. 
Recall that both tropical curves $C$ and $Trop(V(H_{P}))$ have the same
underlying set. 
Let $H'_{C}$ (resp. $W$) be the tropical
modification of $Trop(V(H_{P}))$ (resp. $\RR^2$) 
given by $P(z,w)$, and  $e_1,e_2$, and $e_3$ be the three edges of
$W$ such that $\pi(e_i)$ is an edge of $C$.
Let $(x_i,y_i,z_i)$ be the primitive integer direction of $e_i$ which points
to infinity, and by $\widetilde e_{i,1},\ldots\widetilde e_{i,s_i}$
the ends of $H'_C$  such that
$\pi(e_{i,j})=\pi(e_i)$.
Finally 
we denote by 
$(x_{i},y_{i},\widetilde z_{i,j})$ 
the primitive
integer direction of  $\widetilde e_{i,j}$ 
 pointing to
infinity, and we define
$$z_{H'_C,i}=\sum_{j=1}^{s_i} w(\widetilde e_{i,j})\widetilde z_{i,j}.$$

\begin{lemma}\label{constant coeff}
The tropical curve  $H'_C$ has a unique vertex, which is also the
vertex of $W$. Moreover
$H_C'$ has a vertical end with weight $i_{\Delta(P)}$, 
and for all $i=1,2,3$ we have
\begin{itemize}
\item $z_{H'_C,i}=3z_i$ if $(x_i,y_i)\ne \pm ( 1,0), \pm ( 0,1)$ and
  $\pm ( 1,1)$;
\item $z_{H'_C,i}=3z_i-1$ if $(x_i,y_i)= ( 1,0), ( 0,1)$ or
  $ ( -1,-1)$;
\item $z_{H'_C,i}=3z_i-2$ if $(x_i,y_i)= ( -1,0), ( 0,-1)$ or
  $ ( 1,1)$.
\end{itemize}
\end{lemma}
\begin{proof}
The only thing which does not follow straightforwardly
 from Proposition \ref{max infl points} and Lemma
\ref{inv primitive} is the difference $3z_i-z_{H'_C,i}$. However,
this difference 
corresponds exactly to the common roots of the truncation of $P(z,w)$ and
$H_{P}(z,w)$ along the corresponding edge of $\Delta(P)$ and
$\Delta(H_{P})$, which have been computed in the proof of
Proposition \ref{max infl points} and Lemma 
\ref{inv primitive}.
\end{proof}

\subsection{Localization}\label{auxiliary}

In this whole section $C$ is non-singular tropical curve in $\RR^2$
with Newton polygon the triangle $T_d$, 
and $P(z,w)$ is a  polynomial of degree $d$ 
in $\KK[z,w]$ such that 
$Trop(V(P))=C$.

The proof of next Lemma is the same 
as the one of \cite[Proposition
  2.1]{Br15}. 
\begin{lemma}\label{subd}
Let $F$ be a cell of the dual subdivision of $C$. Then the Newton
polygon $F'$ of 
$H_{P_{\CC, F}}(z,w)$
is a cell of the dual subdivision 
of the
  tropical curve $Trop(V(H_P))$, and
$(H_P)_{\CC,F'}(z,w)=H_{P_{\CC,F}}(z,w)$.
\end{lemma}
It follows easily from Lemma \ref{subd}  that
 the tropical curve $Trop(V(H_P))$ has the same
underlying set 
as $C$, and all its edges are of weight 3.

Let $H_C'$  (resp. $W$) be the tropical modification of
$Trop(V(H_P))$ (resp. $\RR^2$) given by 
$z^2w^2P(z,w)$. It follows from Lemma \ref{subd} that $H'_C$ lies
entirely in $\pi_{|W}^{-1}(C)$.
Following Lemma \ref{intersection}, Theorem \ref{main}
reduces to
estimating the weight and the direction of all vertical ends of $H_C'$.

\begin{lemma}\label{vertex modif 1}
Let $v$ be a vertex of $C$ with $\Delta_v\ne T_1$. Then $H'_C$
has a vertex $v'$ with $\pi(v')=v$ and 
which is also a vertex of $W$. Moreover if $B$ is a ball centered in
$v'$ of radius
$\varepsilon$ small
enough, then 
$B\cap H'_C$ is equal to a translation of $B'\cap H'_{C_v}$,
 where $H'_{C_v}$ is the tropical modification of 
$Trop(V(H_{P_{\CC,v}}))$ 
given by $P_{\CC,v}(z,w)$ and $B'$ is a ball centered in
the origin of radius
$\varepsilon$.
\end{lemma}
\begin{proof}
The tropical curve $H'_C$ is the tropicalization of the curve
in $(\KK^*)^3$ given by the system of equations
\begin{equation}\label{equ modif}
\left\{\begin{array}{l}
H_P(z,w)=0 
\\
\\ u-P(z,w)=0
\end{array}\right.
\end{equation}
(the coordinates in  $(\KK^*)^3$ are $z,w,$ and $u$).
Without loss of generality,
we may assume that $v=(0,0)$, that the point $v''=(0,0,0)$ is a vertex
of $W$, and that the coefficients of the monomials of $P$ corresponding to the vertex of $\Delta_v$ have valuation $0$.
 Hence if $B$ is a small ball
centered in $v''$, we have that 
 $B\cap H'_C$ is given by the tropicalization of the curve obtained by
plugging $t=0$ in the system (\ref{equ modif}). According to Lemma
\ref{subd}, this tropical curve is exactly the tropical modification of
$Trop(V(H_{P_{\CC,v}}))$ 
given by $P_{\CC,v}(z,w)$. 
\end{proof}
Lemma \ref{vertex modif 1} implies that given $v$ 
 a vertex of $C$ such that $\Delta_v\ne T_1$, if $B_v\subset\RR^2$ is
 a ball small enough centered in $v$, 
then the tropical curve $H_C'$ is completely determined in
$B_v\times\RR$
 by the tropical modification of $V(H_{P_{\CC,v}})$ 
given by $P_{\CC,v}(z,w)$. Since this modification is given in Lemma
\ref{constant coeff},  we see that 
the curve $C$ determines the curve $H'_C$ in 
$\cup_v B_v\times\RR$. Hence it remains 
to study $H'_C$ in
$\RR^3\setminus (\cup_v B_v\times\RR)$.

Let $v$ be a vertex of $C$, and $L$ be a tropical line with vertex $v$
and such that $C\cap L$ contains an inflection component of multiplicity at least
3 which is not reduced to $v$. We denote by $E$ the connected
component of $C\cap L$ containing $v$.
If $\Delta_v\ne T_1$ we define
$E'=\overline E\setminus\{v\}$, and we define
$E'=\overline E$ otherwise.

\begin{lemma}\label{edge 1}
The number of inflection points of $\overline X$ with valuation
in
$E'$ is 6 if 
$E$ is made of 3 bounded edges
of $C$, and 3 otherwise.
\end{lemma}
\begin{proof}
Let $e_1,\ldots,e_r$ (resp. $\widetilde e_1,\ldots,\widetilde e_s$) 
be the vertical ends  of $H'_C$ (resp.
unbounded edges of $H'_C$ which are contained in 
$\pi_{C'}^{-1}(E)$
and
not vertical), 
and let $(\widetilde x_i,\widetilde
y_i,\widetilde z_i)$  be the primitive
integer direction of  $\widetilde e_i$ 
 pointing to
infinity. 
 If $s>0$, then there exists a unique 
 unbounded
edge $e$ of $E$ such that 
 $\pi_{W}^{-1}(e)$ contains all edges $\widetilde e_1,\ldots,\widetilde e_s$.
Let $(x_{e},y_{e},z_{e})$ be the primitive
integer direction  pointing to infinity of $e$. 
The number of inflection points of $\overline X$ with valuation
in  
$E'$ is equal to 
$$\sum_{i=1}^r w(e_i) + 3z_{e}-\sum_{i=1}^s w(\widetilde e_i) \widetilde z_{i}.$$ 
According to Proposition \ref{inclusion}, the balancing condition,
Lemma \ref{constant coeff}, and Lemma \ref{vertex modif 1}, 
this sum is precisely  equal to 6 if
$E$ is made of 3 bounded edges
of $C$, and to 3 otherwise.
\end{proof}

So far we have proved Theorem \ref{main} in all cases except when
$E$ is made of three bounded edges and contains exactly 2 inflection
components in its relative interior.

\begin{lemma}\label{edge 2}
Suppose that 
$E$ is made of three bounded edges and contains exactly 2 inflection
components $p_1$ and $p_2$ in its relative interior. 
Then  $\overline X$ has exactly 3 inflection points with valuation
in $p_i$, $i=1,2$.
\end{lemma}
\begin{proof}
The 
number of inflection points we are looking for is the weight of
the vertical end $e_i$ (if any) with $\pi(e_i)=p_i$. According to 
 Proposition \ref{inclusion}, the
balancing condition, Lemma \ref{constant coeff}, and Lemma \ref{vertex modif 1},
one sees easily that
there is no other possibility for this weight other than  to
be equal to 3. 
\end{proof}

\section{Examples}\label{constructions}
In this section we use Theorem \ref{main real} to construct some
examples of maximally inflected real curves. In Proposition
\ref{classification d=4} we classify all possible distributions of
real inflection points among the connected components of a real
quartic. In Proposition \ref{one component}, we construct maximally
inflected curves with many connected components
 and all real inflection points on only one of them.

\vspace{1ex}
Before stating Proposition \ref{classification d=4}, let us introduce
the following notation: we say that a non-singular real algebraic
curve $\RR X$ in $\mathbb R P^2$ has
inflection type $(n_1,\ldots,n_k)$ if $\RR X$ has exactly $k$ ovals
$O_1,\ldots, O_k$ and $O_i$ contains exactly $n_i$ real inflection points.
Note that a maximally 
inflected real quartic 
made of two nested ovals automatically has inflection type
$(8,0)$. 
\begin{prop}\label{classification d=4}
The inflection types realized by 
 maximally inflected quartics in $\mathbb R P^2$ are exactly
$$(8),\ (8,0), \ (6,2),\ (4,4), \ (8,0,0), \ (6,2,0), \ (4,2,2),
 \ (4,4,0),$$
$$ (6,2,0,0), \ (4,2,2,0),\ (4,4,0,0), \ (2,2,2,2).$$
\end{prop}
\begin{proof}
\begin{figure}[h]
\centering
\begin{tabular}{cccc}
\includegraphics[height=4cm, angle=0]{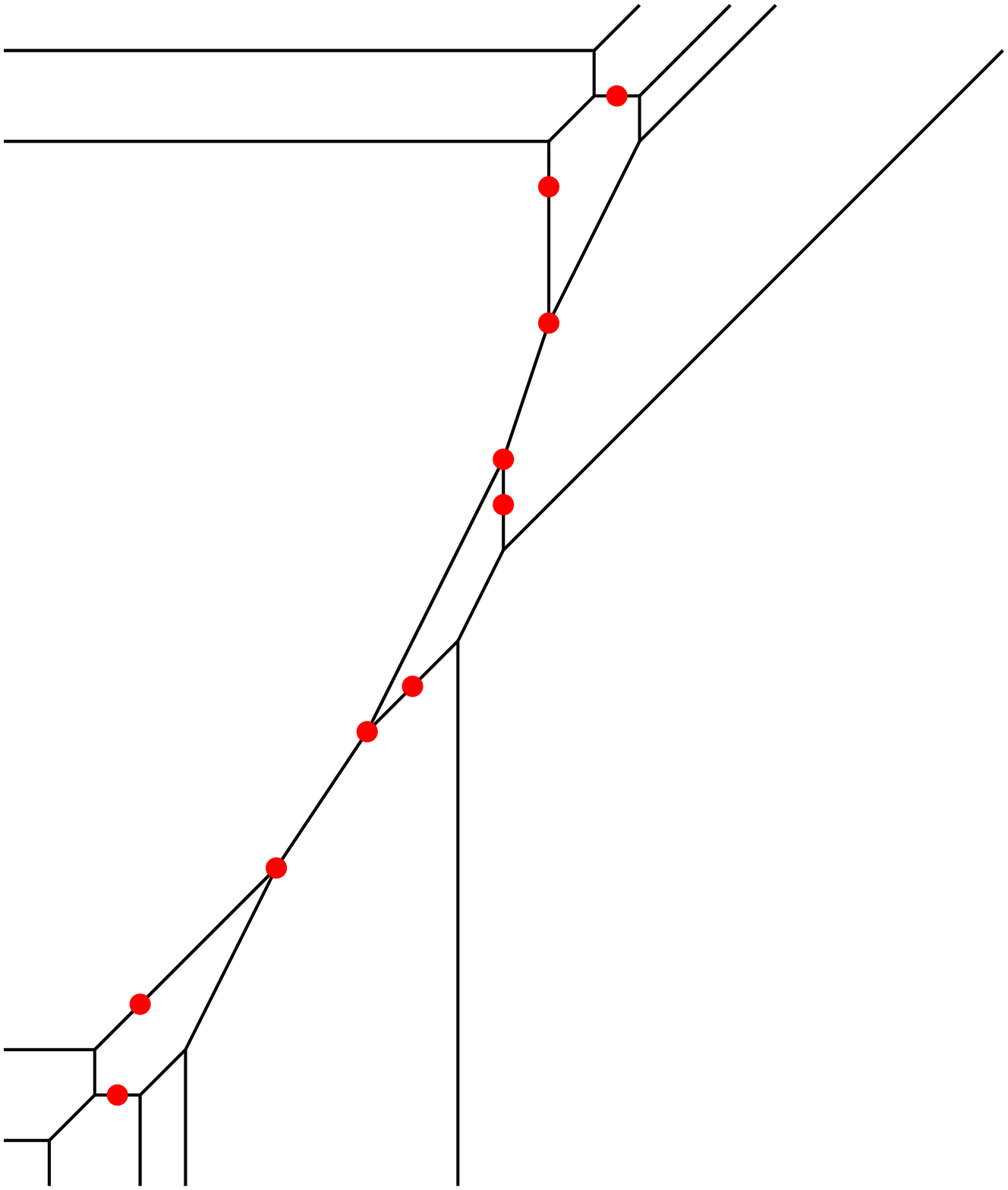}&
\includegraphics[height=4cm, angle=0]{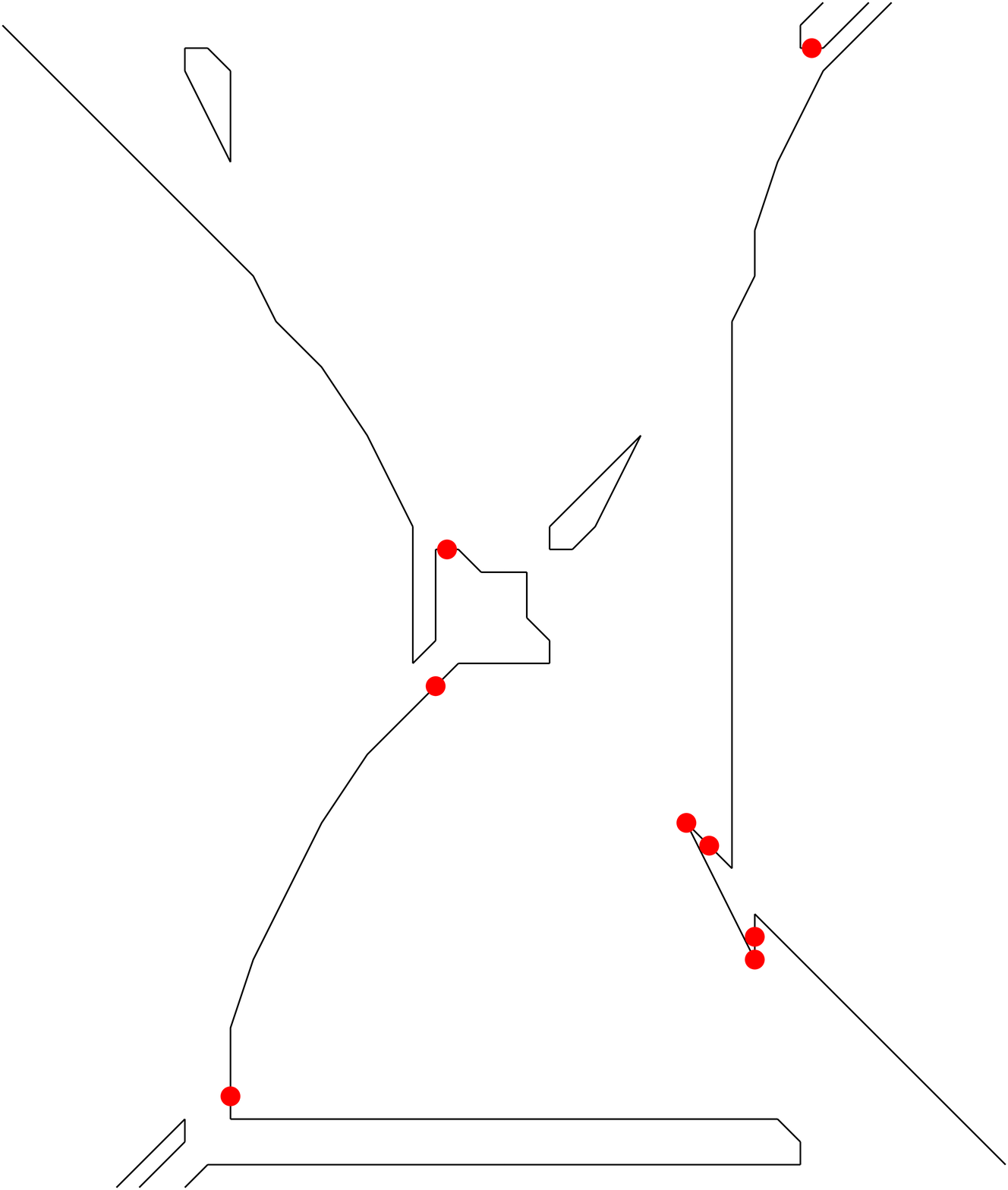}&
\includegraphics[height=4cm, angle=0]{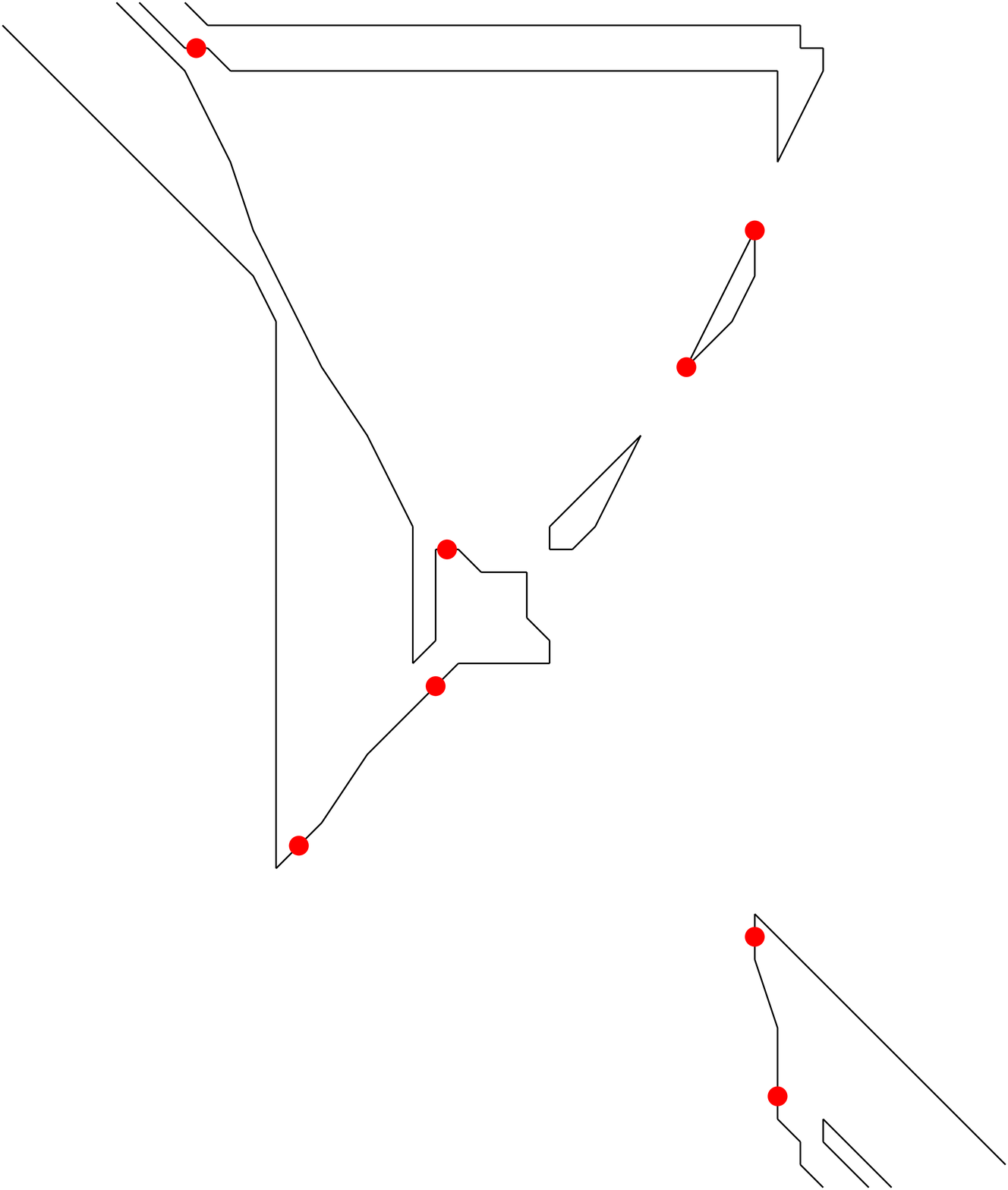}&
\includegraphics[height=4cm, angle=0]{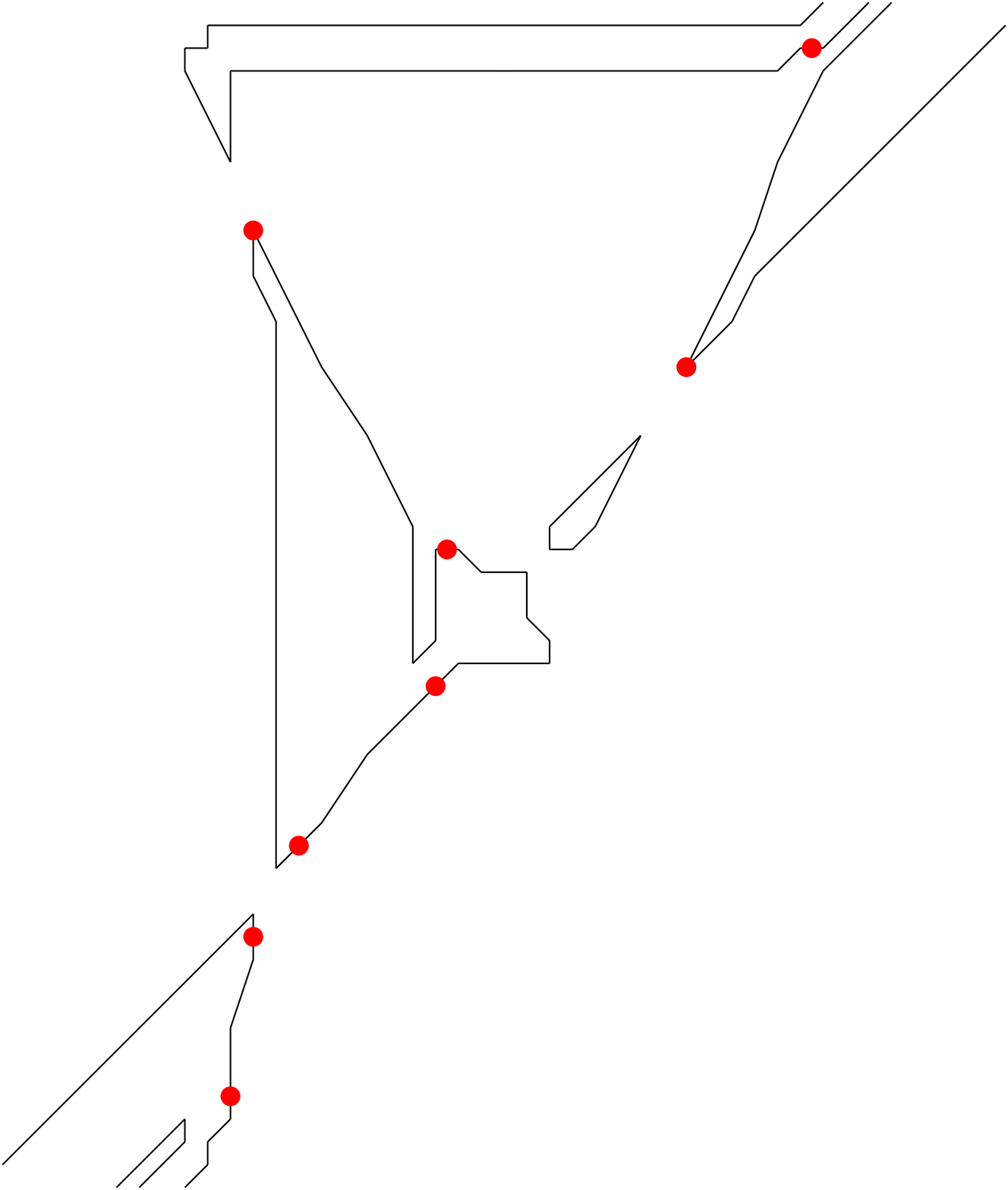}
\\ a) & b) & c) & d)
\\ \\ 
\end{tabular}
\begin{tabular}{cccc}
\includegraphics[height=4cm, angle=0]{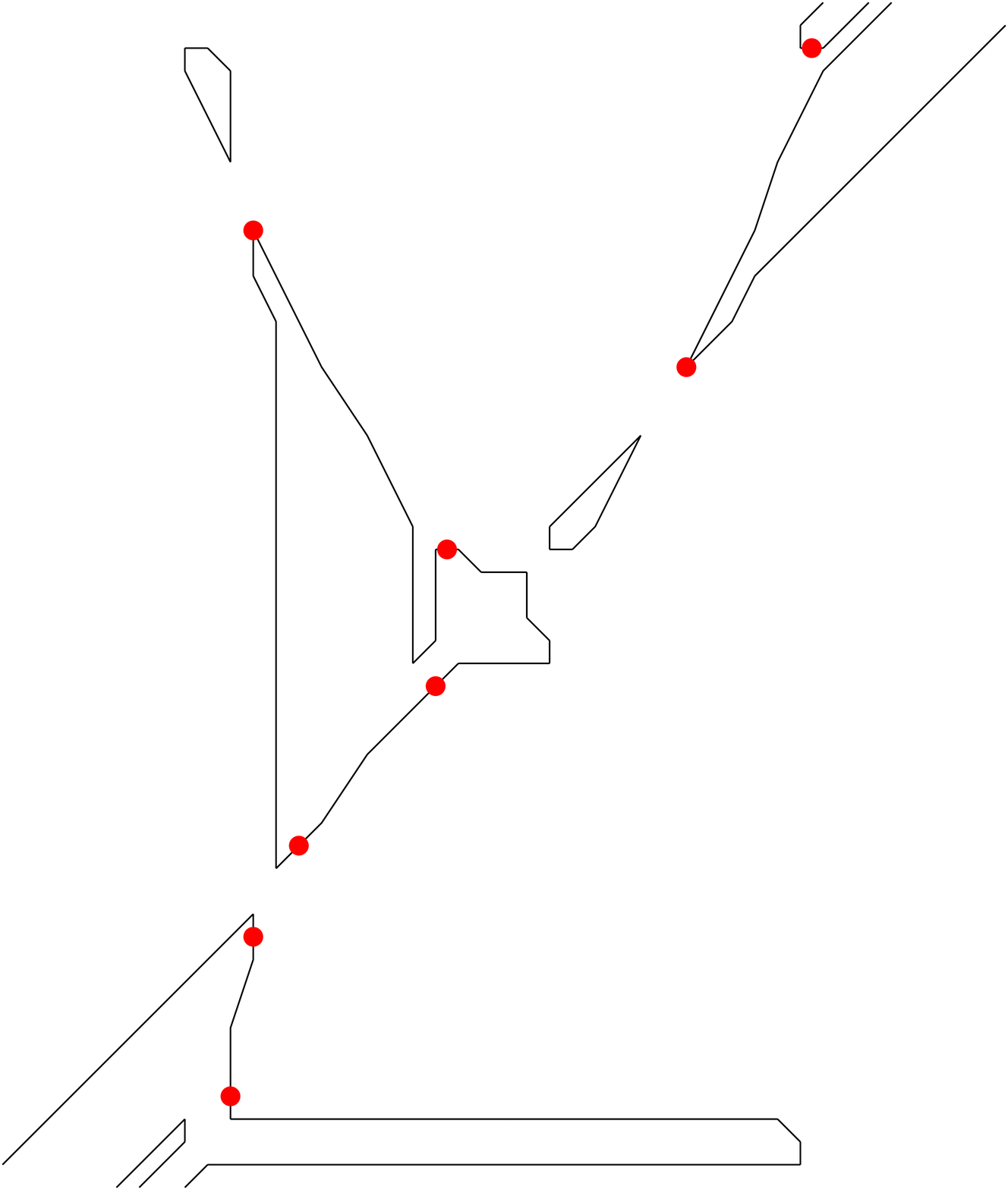}&
\includegraphics[height=4cm, angle=0]{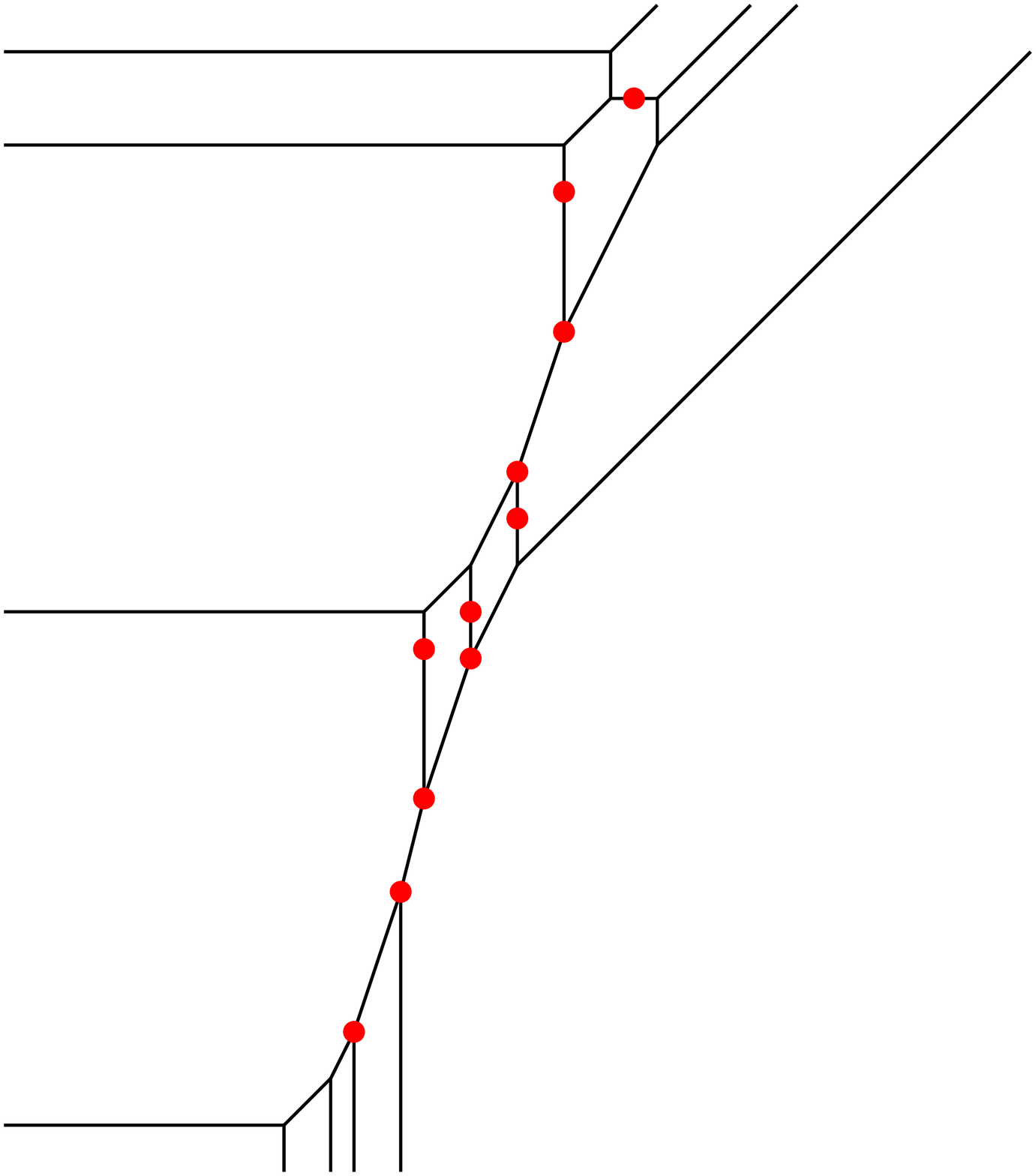}&
\includegraphics[height=4cm, angle=0]{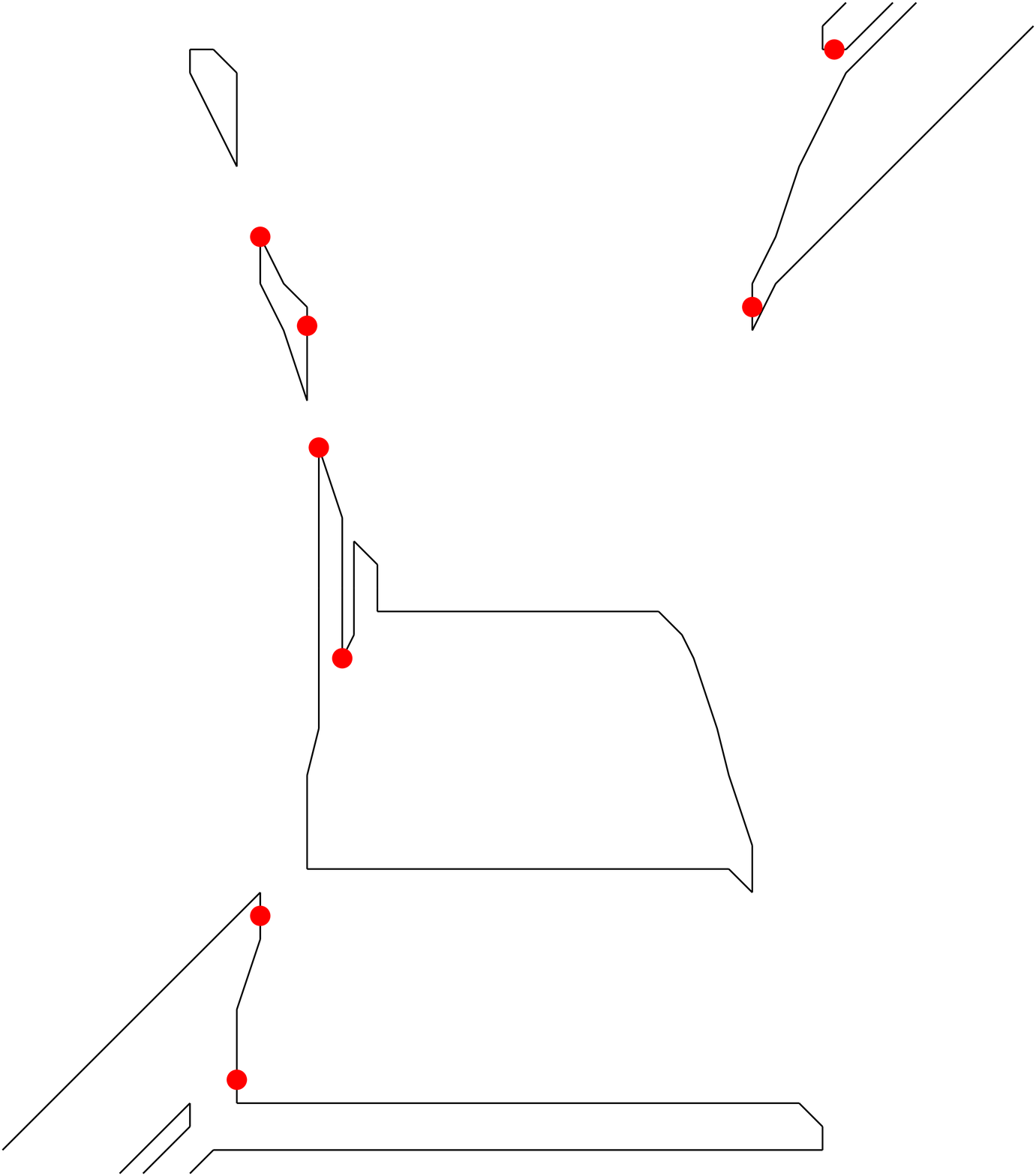}&
\includegraphics[height=4cm, angle=0]{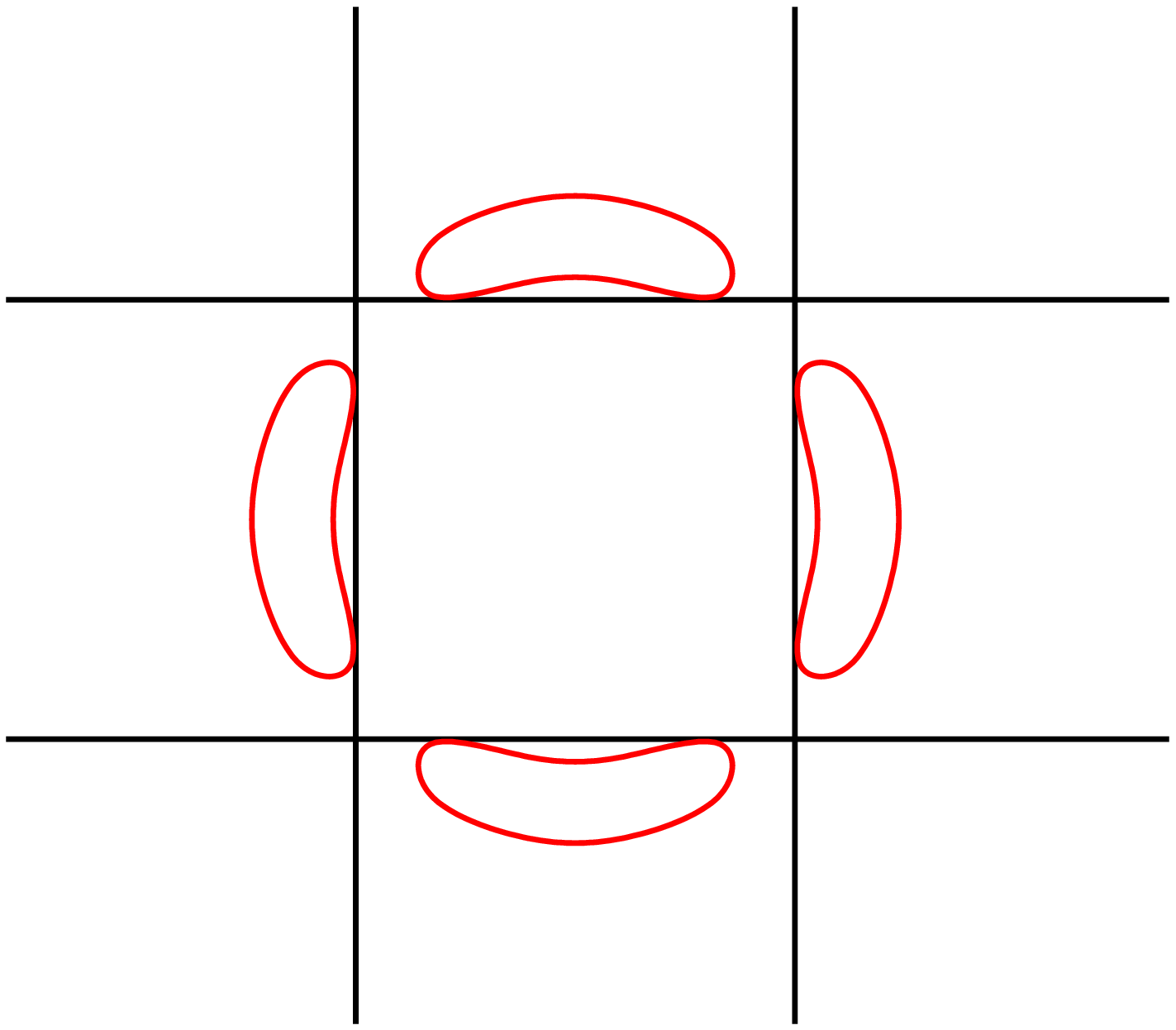}
\\ e) &  f) & g) &h)
\end{tabular}
\caption{Maximally inflected quartics}
\label{quartic real 1}
\end{figure}
The inflection types $(8), (8,0),  (6,2), (4,4),(4,2,2)$ and $(2,2,2,2)$ are
realized by perturbing the union of two ellipses intersecting in 4
real points. The inflection type $(6,2,2,0)$ is realized by the
Harnack quartic constructed in Figure \ref{honey real}.
The inflection types $(8,0,0)$, $(6,2,0)$, $(4,4,0)$, and $(4,4,0,0)$  are
 realized out of the tropical curve depicted in Figure 
\ref{quartic real 1}a
by the patchworkings depicted respectively in Figures \ref{quartic
  real 1}b,c,  d, and e.
The inflection type $(4,2,2,0)$ is
 realized out of the tropical curve depicted in Figure 
\ref{quartic real 1}f
by the patchworking depicted  in Figure \ref{quartic
  real 1}g.
Note that some of these inflection types can also be realized by
smoothing maximally inflected rational quartics from \cite{Kha4}.

Hence it remains to prove that the inflection type $(8,0,0,0)$ is not
realizable by any quartic. The following argument is due to Kharlamov
and simplifies considerably our original proof of this fact.
It is a Theorem by Klein (\cite{Kle3}) that the rigid isotopy class of a
non-singular real quartic curve is determined by its isotopy type in
$\RR P^2$. Moreover, it is easy to see that this isotopy type also
determines the number of real bitangents to the quartic. In the case of
real quartics with 4 ovals, one sees by perturbing the union of two
conics (see Figure \ref{quartic real 1}h) 
that all the 28 complex bitangents to these curves are in fact
real: 24 bitangents tangent to two distinct ovals, and 4 remaining
bitangents. These latter subdivide $\RR P^2$ into 3
quadrangles and 4 triangles, each of these triangles containing exactly
one oval of the quartic (see Figure \ref{quartic real 1}h). 
In particular, no oval 
 has 4 bitangents, which implies that no oval
contains 8 real inflection points.
\end{proof}

\begin{prop}\label{one component}
Given $k\ge 1$, there exists a maximally inflected real algebraic
curve of degree $2k$ with one oval containing all real inflection points
and $(k-1)^2$ other convex ovals, and there exists a maximally
inflected real algebraic 
curve of degree $2k+1$ with the pseudo-line containing all real inflection points
and $k^2$ other convex ovals.
\end{prop}
\begin{proof}
\begin{figure}[h]
\centering
\begin{tabular}{ccccc}
\includegraphics[height=2cm, angle=0]{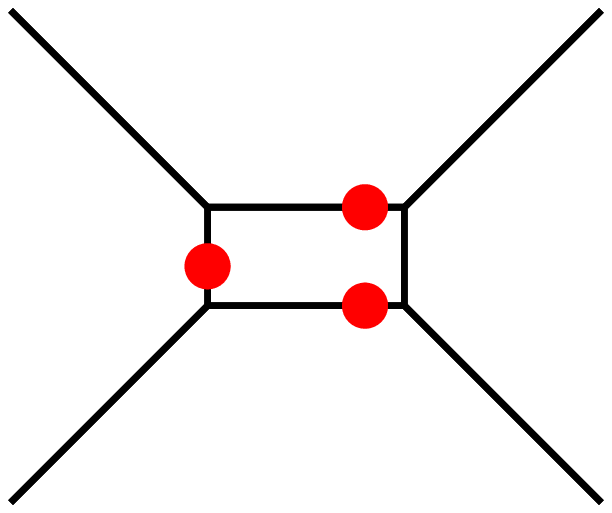} & \hspace{5ex} &
\includegraphics[height=4cm, angle=0]{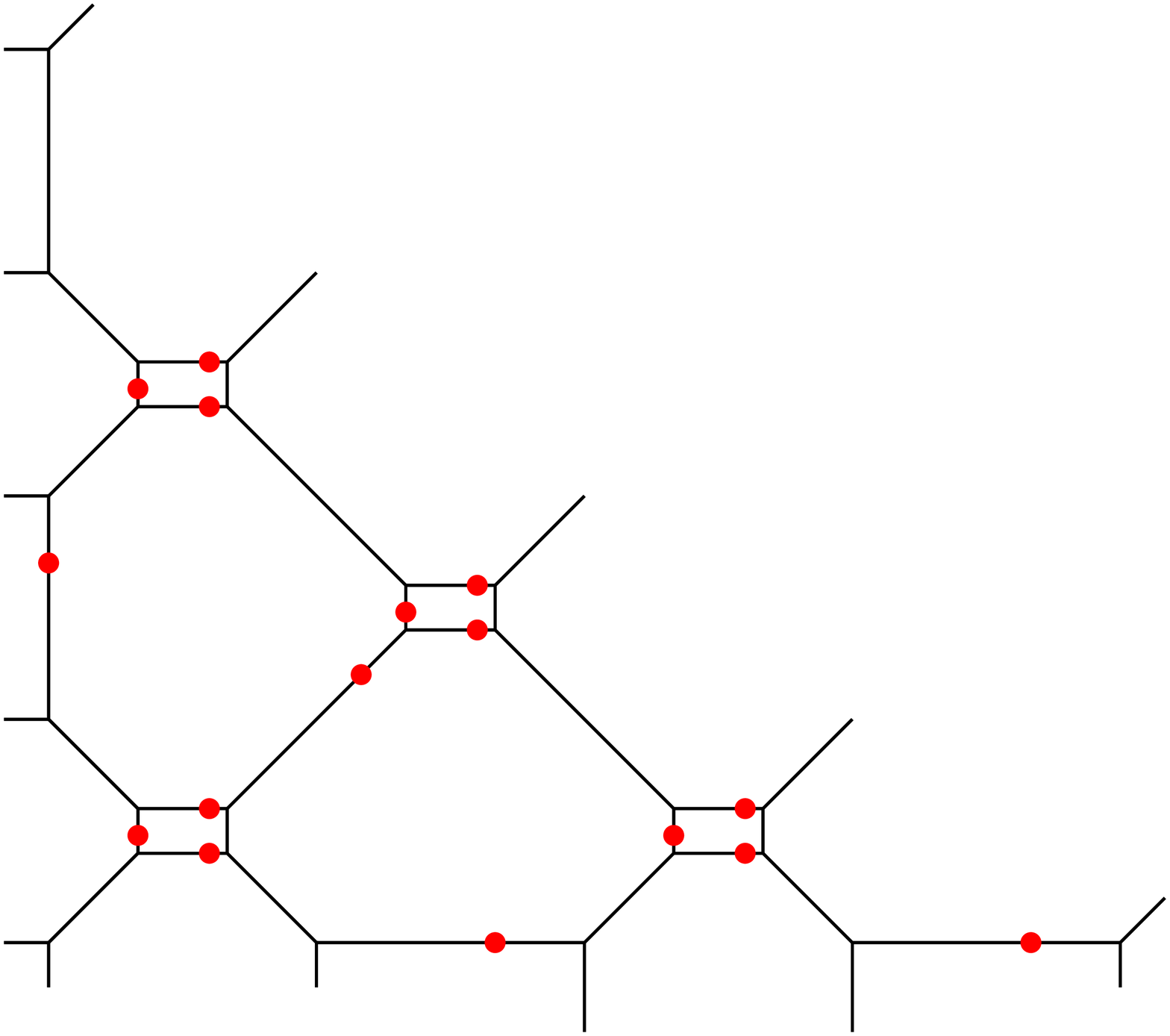}& \hspace{5ex} &
 \includegraphics[height=4cm, angle=0]{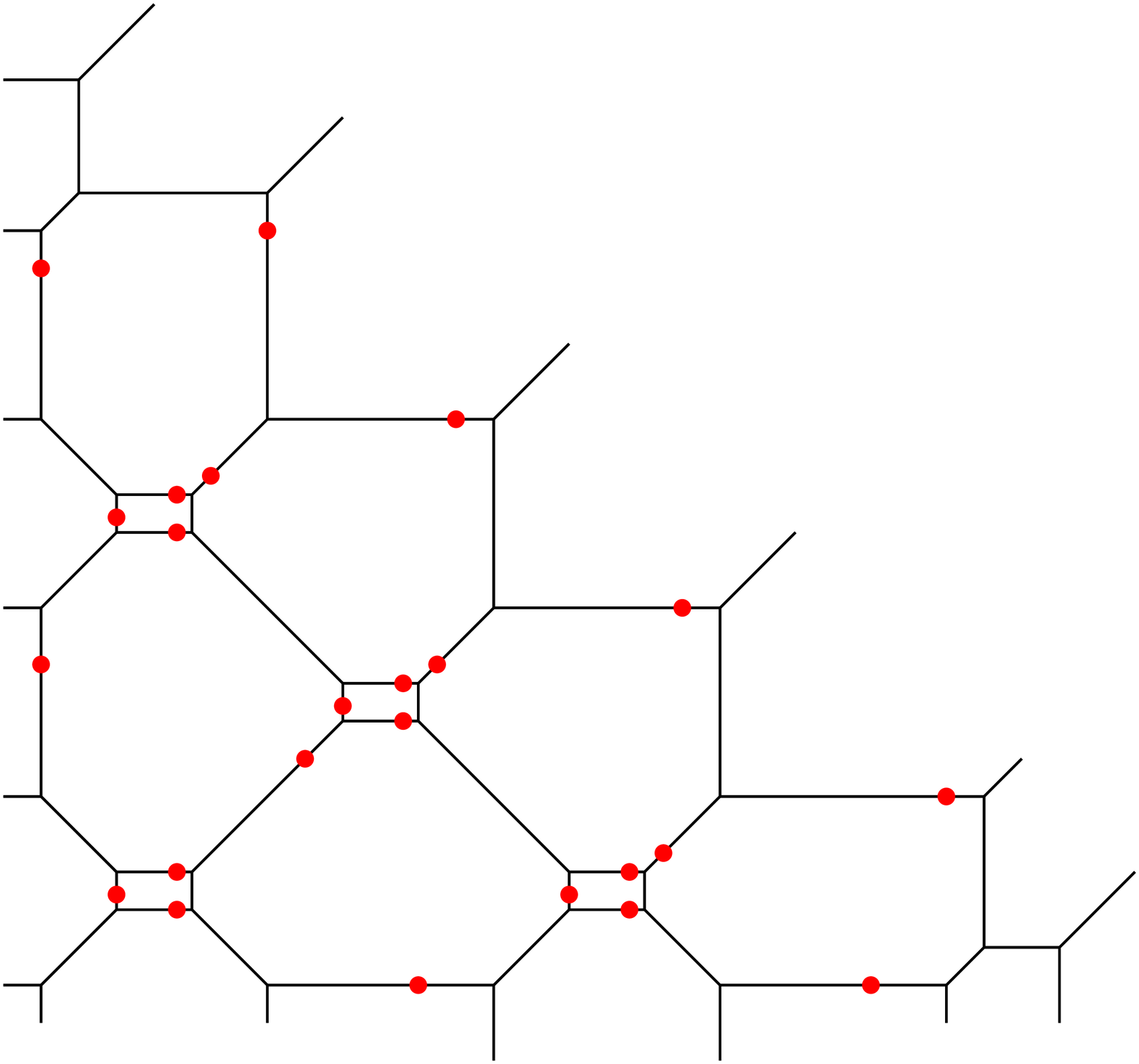}
\\ a) Fragment && b) $d=5$ && c) $d=6$
\\
\\
\includegraphics[height=3cm, angle=0]{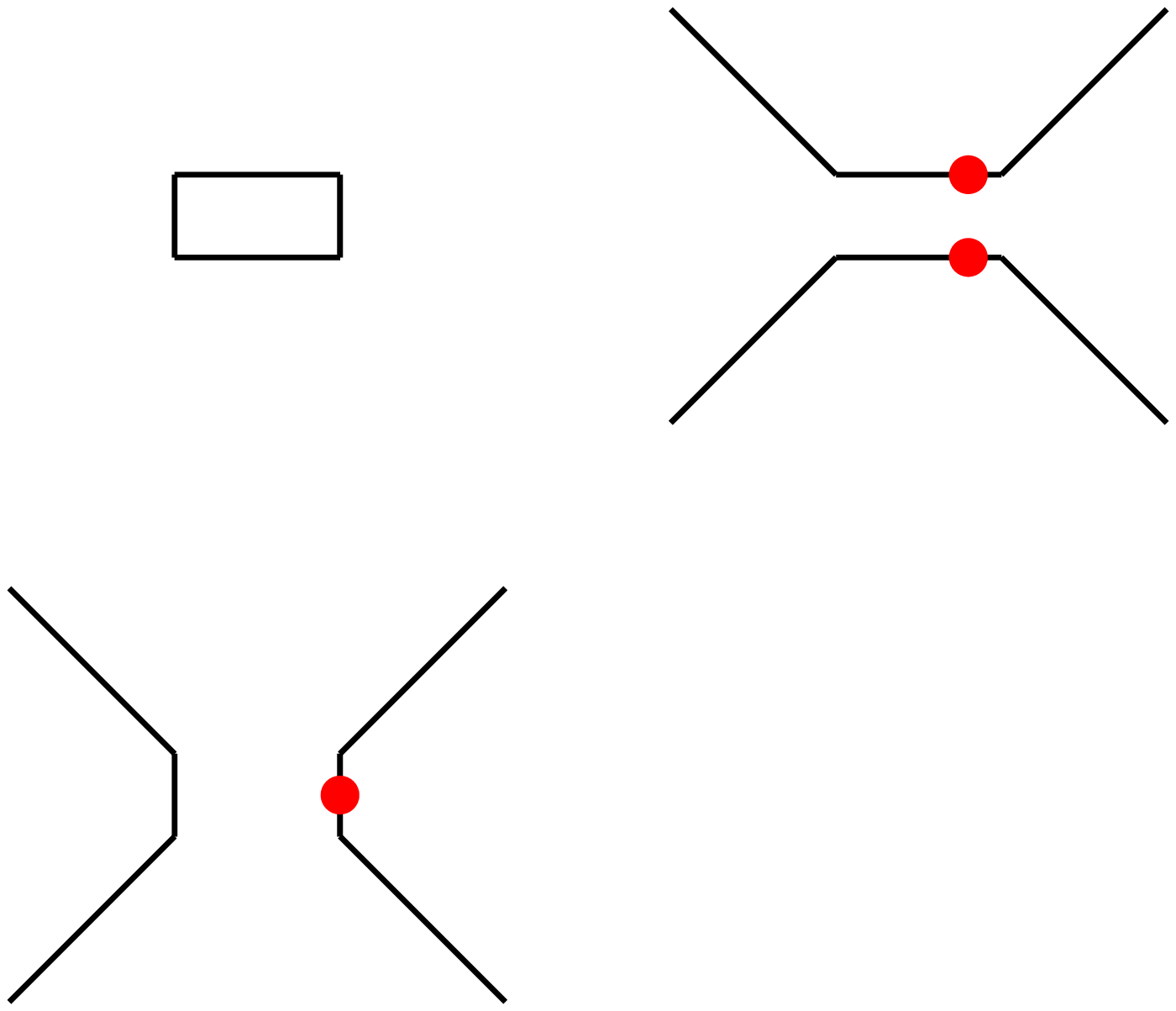}& \hspace{5ex} &
\includegraphics[height=4cm, angle=0]{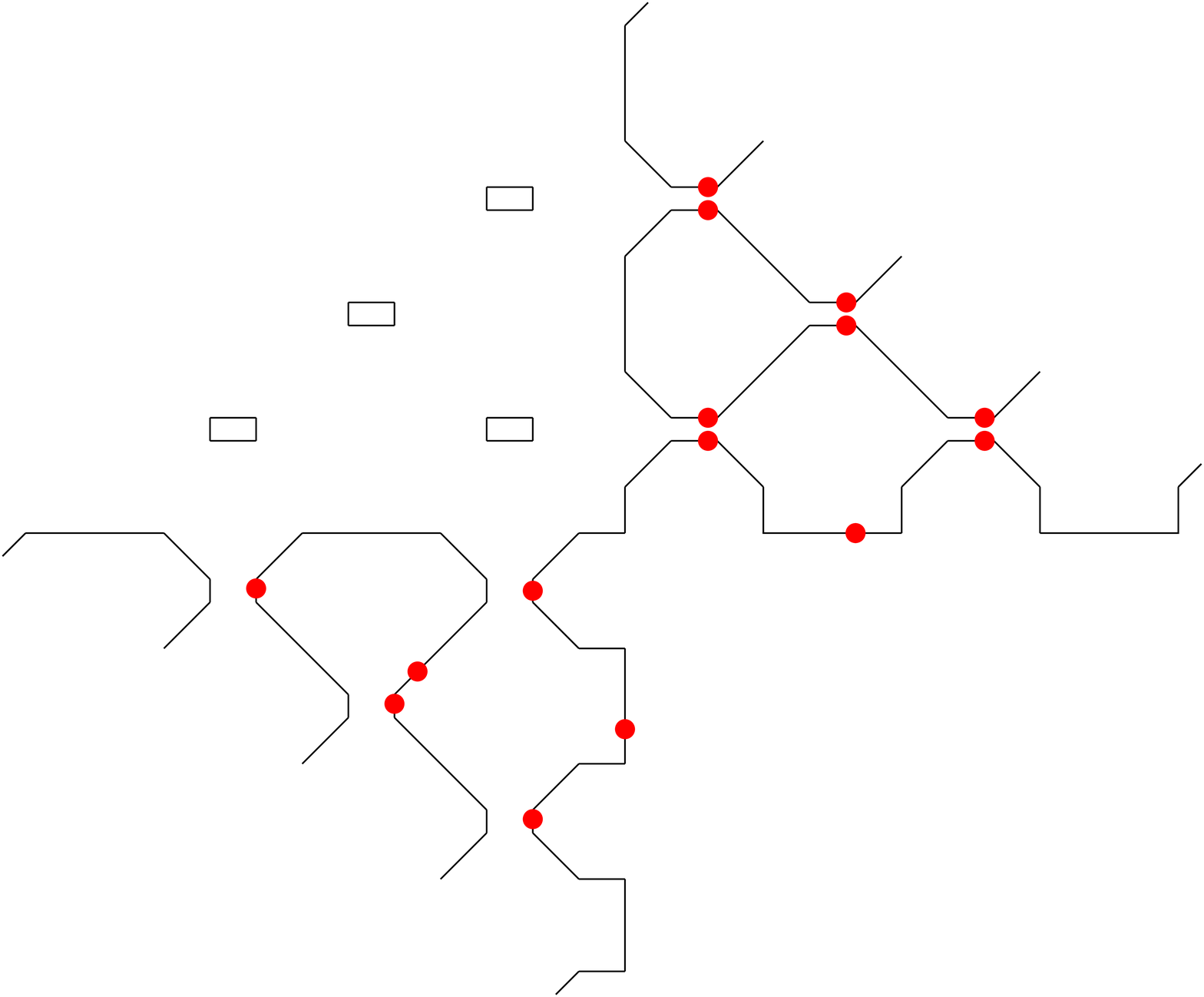} & \hspace{5ex} &
\includegraphics[height=4cm, angle=0]{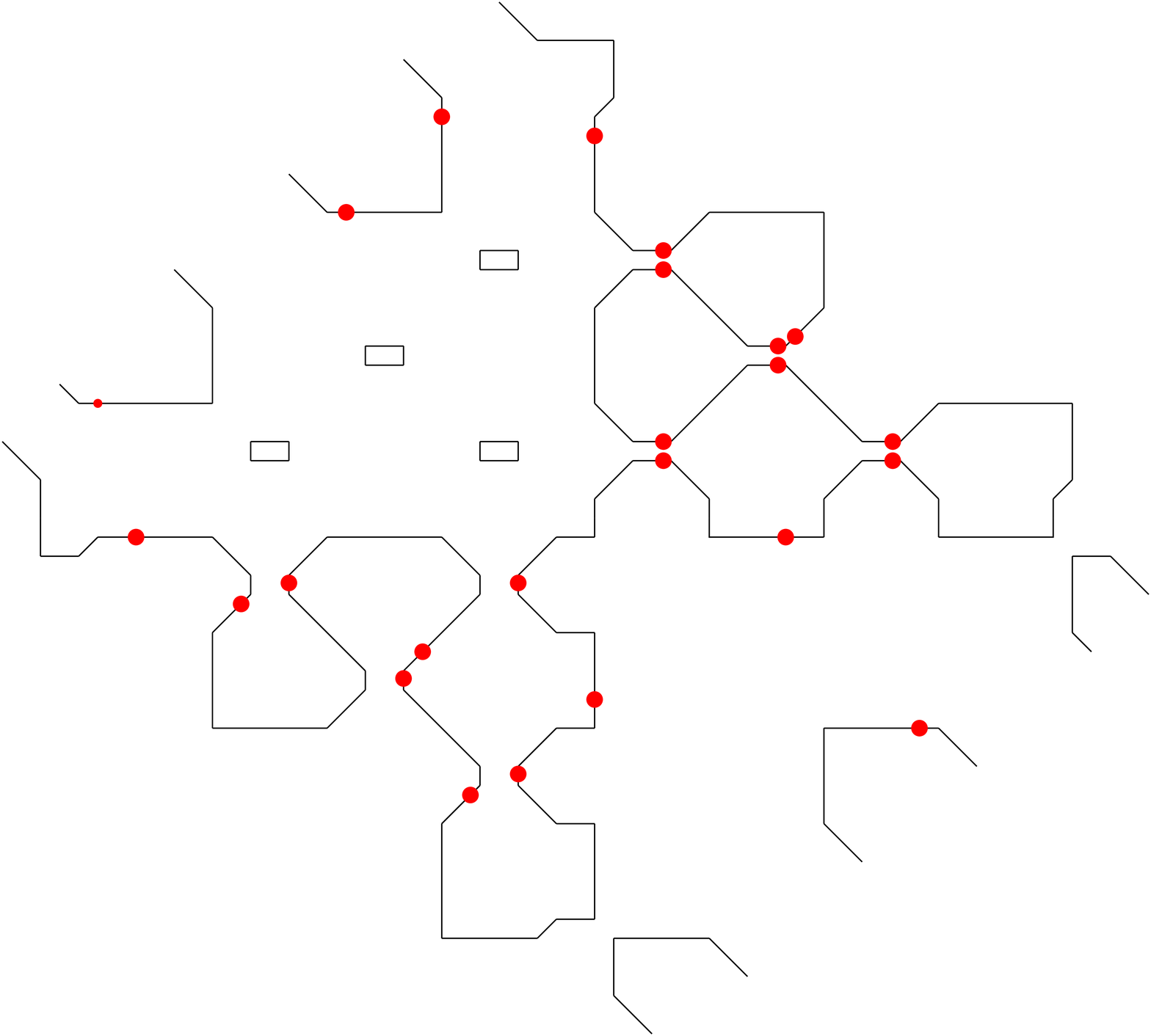}
\\d) Patchworked fragment && e) $d=5$ && f) $d=6$
\end{tabular}
\caption{Maximally inflected curves with all inflection points on a
  single component}
\label{1 oval}
\end{figure}
Let us consider a non-singular tropical curve of degree $2k+1$
(resp. $2k$) which contains $k^2$ (resp. $(k-1)^2$) copies of the
fragment $\mathcal F$ 
depicted in
Figure \ref{1 oval}a (see Figure \ref{1 oval}b in the case $2k+1=5$,
and Figure \ref{1 oval}c in the case $2k=6$). The curves whose
existence is claim in the proposition can easily be constructed by
patchworking all fragments  $\mathcal F$ as depicted in Figure \ref{1
  oval}d (see Figure \ref{1 oval}e in the case $2k+1=5$,
and Figure \ref{1 oval}f in the case $2k=6$).
\end{proof}

\small

\def\rightmark{\em Bibliography}

\bibliographystyle{alpha}

\newcommand{\etalchar}[1]{$^{#1}$}

\end{document}